\documentclass[11pt, reqno]{amsart}

\usepackage{mathrsfs}
\usepackage{mathtools}
\mathtoolsset{showonlyrefs,showmanualtags} 
\usepackage{amssymb}
\usepackage{latexsym,amsthm}
\usepackage{bm}
\usepackage{enumitem,caption}
\usepackage{calc}

\usepackage{wasysym}

\usepackage{hyperref}
\hypersetup{%
    pdfborder = {0 0 0},
    colorlinks,
    citecolor=red!50!black,
    filecolor=Darkgreen,
    linkcolor=blue,
    urlcolor=cyan!50!black!90
}

\usepackage{tikz,tikz-cd}
\usetikzlibrary{patterns,matrix,arrows,shapes,snakes}
\newcommand{\btp}{\begin{tikzpicture}[baseline=-.25em,scale=0.25,line width=0.7pt]}
\newcommand{\etp}{\end{tikzpicture}}

\allowdisplaybreaks[1]

\numberwithin{equation}{section}

\setlist[itemize,1]{leftmargin=.4in}
\setlist[enumerate,1]{leftmargin=.4in,label=(\roman*)}
\setlist[description,1]{leftmargin=.4in,font=\normalfont\itshape}

\newcommand{\nc}{\newcommand}
\newcommand{\rnc}{\renewcommand}

\nc{\andreacomment}[1]{{\color{red}//{#1}//}}
\nc{\bartcomment}[1]{{\color{green}//{#1}//}}

\newtheorem{theorem}{Theorem}[section]
\newtheorem{proposition}[theorem]{Proposition}
\newtheorem{corollary}[theorem]{Corollary}
\newtheorem{lemma}[theorem]{Lemma}

\newtheorem*{theorem*}{Theorem}

\theoremstyle{definition}
\newtheorem{definition}[theorem]{Definition}
\newtheorem{example}[theorem]{Example}

\newtheorem*{defn*}{Definition}
\newtheorem*{exam*}{Example}

\newtheorem{remark}[theorem]{Remark}

\newcommand{\rmkend}{\ensuremath{\diameter}}
\newcommand{\examend}{\ensuremath{\diameter}}
\newcommand{\defnend}{\ensuremath{\diameter}}

\nc{\spl}[1]{\begin{align}\begin{split}#1\end{split}\end{align}}
\nc{\eqd}[1]{\begin{equation}\begin{aligned}#1\end{aligned}\end{equation}}
\nc{\eqa}[1]{\begin{align}#1\end{align}}
\nc{\eqg}[1]{\begin{gather}#1\end{gather}}
\nc{\eq}[1]{\begin{equation}#1\end{equation}}
\nc{\eqn}[1]{\begin{align*}#1\end{align*}}

\nc{\eqrefs}[2]{\text{(\ref{#1}-\ref{#2})}}

\nc\el{\nonumber\\}
\nc\nn{\nonumber}

\nc{\litem}[1]{\medskip\noindent\hspace{-13pt} $\circ$ #1:}

\nc{\scsop}[1]{\scriptstyle{\sf #1}}
\nc{\mf}{\mathfrak}
\nc{\mc}{\mathcal}
\nc{\ms}{\mathsf}
\nc{\bb}{\mathbb}

\nc{\ol}{\overline}
\nc{\ul}{\underline}
\nc{\wb}{\overline}
\nc{\wc}{\widecheck}
\nc{\wh}{\widehat}
\nc{\tl}{\tilde}
\nc{\wt}{\widetilde}

\nc{\al}{\alpha}
\nc{\be}{\beta}
\nc{\ga}{\gamma}
\nc{\Ga}{\Gamma}
\nc{\del}{\delta}
\nc{\Del}{\Delta}
\nc{\eps}{\epsilon}
\nc{\veps}{\varepsilon}
\nc{\Eps}{\Epsilon}
\nc{\ze}{\zeta}
\nc{\ka}{\kappa}
\nc{\la}{\lambda}
\nc{\La}{\Lambda}
\nc{\si}{\sigma}
\nc{\Si}{\Sigma}
\nc{\ups}{\upsilon}
\nc{\Ups}{\Upsilon}
\nc{\vphi}{\varphi}
\nc{\om}{\omega}
\nc{\Om}{\Omega}

\nc{\A}{\mathbb{A}}
\nc{\C}{\mathbb{C}}
\nc{\F}{\mathbb{F}}
\nc{\K}{\mathbb{K}}
\nc{\N}{\mathbb{N}}
\nc{\Q}{\mathbb{Q}}
\nc{\R}{\mathbb{R}}
\nc{\Z}{\mathbb{Z}}

\nc{\mbA}{\mathbf{A}}
\nc{\mbb}{\mathbf{b}}
\nc{\mbB}{\mathbf{B}}
\nc{\mbc}{\mathbf{c}}
\nc{\mbC}{\mathbf{C}}
\nc{\mbd}{\mathbf{d}}
\nc{\mbD}{\mathbf{D}}
\nc{\mbe}{\mathbf{e}}
\nc{\mbE}{\mathbf{E}}
\nc{\mbf}{\mathbf{f}}
\nc{\mbF}{\mathbf{F}}
\nc{\mbg}{\mathbf{g}}
\nc{\mbH}{\mathbf{H}}
\nc{\mbh}{\mathbf{h}}
\nc{\mbi}{\mathbf{i}}
\nc{\mbI}{\mathbf{I}}
\nc{\mbj}{\mathbf{j}}
\nc{\mbJ}{\mathbf{J}}
\nc{\mbk}{\mathbf{k}}
\nc{\mbK}{\mathbf{K}}
\nc{\mbL}{\mathbf{L}}
\nc{\mbM}{\mathbf{M}}
\nc{\mbQ}{\mathbf{Q}}
\nc{\mbq}{\mathbf{q}}
\nc{\mbr}{\mathbf{r}}
\nc{\mbT}{\mathbf{T}}
\nc{\mbu}{\mathbf{u}}
\nc{\mbU}{\mathbf{U}}
\nc{\mbv}{\mathbf{v}}
\nc{\mbV}{\mathbf{V}}
\nc{\mbw}{\mathbf{w}}
\nc{\mbW}{\mathbf{W}}
\nc{\mbX}{\mathbf{X}}
\nc{\mbY}{\mathbf{Y}}
\nc{\mbZ}{\mathbf{Z}}

\nc{\mcA}{\mathcal{A}}
\nc{\mcB}{\mathcal{B}}
\nc{\mcC}{\mathcal{C}}
\nc{\mcD}{\mathcal{D}}
\nc{\mcE}{\mathcal{E}}
\nc{\mcF}{\mathcal{F}}
\nc{\mcH}{\mathcal{H}}
\nc{\mcK}{\mathcal{K}}
\nc{\mcO}{\mathcal{O}}
\nc{\mcQ}{\mathcal{Q}}
\nc{\mcR}{\mathcal{R}}
\nc{\mcS}{\mathcal{S}}
\nc{\mcP}{\mathcal{P}}
\nc{\mcU}{\mathcal{U}}
\nc{\mcT}{\mathcal{T}}
\nc{\mcV}{\mathcal{V}}
\nc{\mcY}{\mathcal{Y}}
\nc{\mcZ}{\mathcal{Z}}

\nc{\mfa}{\mathfrak{a}}
\nc{\mfA}{\mathfrak{A}}
\nc{\mfb}{\mathfrak{b}}
\nc{\mfB}{\mathfrak{B}}
\nc{\mfc}{\mathfrak{c}}
\nc{\mfC}{\mathfrak{C}}
\nc{\mfd}{\mathfrak{d}}
\nc{\mfD}{\mathfrak{D}}
\nc{\mfe}{\mathfrak{e}}
\nc{\mfE}{\mathfrak{E}}
\nc{\mff}{\mathfrak{f}}
\nc{\mfF}{\mathfrak{F}}
\nc{\mfg}{\mathfrak{g}}
\nc{\mfh}{\mathfrak{h}}
\nc{\mfH}{\mathfrak{H}}
\nc{\mfi}{\mathfrak{i}}
\nc{\mfJ}{\mathfrak{J}}
\nc{\mfk}{\mathfrak{k}}
\nc{\mfK}{\mathfrak{K}}
\nc{\mfl}{\mathfrak{l}}
\nc{\mfL}{\mathfrak{L}}
\nc{\mfM}{\mathfrak{M}}
\nc{\mfm}{\mathfrak{m}}
\nc{\mfn}{\mathfrak{n}}
\nc{\mfN}{\mathfrak{N}}
\nc{\mfo}{\mathfrak{o}}
\nc{\mfP}{\mathfrak{P}}
\nc{\mfQ}{\mathfrak{Q}}
\nc{\mfr}{\mathfrak{r}}
\nc{\mfS}{\mathfrak{S}}
\nc{\mft}{\mathfrak{t}}
\nc{\mfU}{\mathfrak{U}}
\nc{\mfu}{\mathfrak{u}}
\nc{\mfV}{\mathfrak{V}}
\nc{\mfX}{\mathfrak{X}}
\nc{\mfY}{\mathfrak{Y}}
\nc{\mfz}{\mathfrak{z}}

\nc{\msA}{\mathsf{A}}
\nc{\msB}{\mathsf{B}}
\nc{\msC}{\mathsf{C}}
\nc{\msc}{\mathsf{c}}
\nc{\msD}{\mathsf{D}}
\nc{\msd}{\mathsf{d}}
\nc{\mse}{\mathsf{e}}
\nc{\msw}{\mathsf{w}}
\nc{\msq}{\mathsf{q}}
\nc{\msg}{\mathsf{g}}
\nc{\msE}{\mathsf{E}}
\nc{\msf}{\mathsf{f}}
\nc{\msF}{\mathsf{F}}
\nc{\msh}{\mathsf{h}}
\nc{\msH}{\mathsf{H}}
\nc{\msI}{\mathsf{I}}
\nc{\msJ}{\mathsf{J}}
\nc{\msK}{\mathsf{K}}
\nc{\msL}{\mathsf{L}}
\nc{\msP}{\mathsf{P}}
\nc{\msQ}{\mathsf{Q}}
\nc{\msR}{\mathsf{R}}
\nc{\mss}{\mathsf{s}}
\nc{\msS}{\mathsf{S}}
\nc{\msT}{\mathsf{T}}
\nc{\msU}{\mathsf{U}}
\nc{\msV}{\mathsf{V}}
\nc{\msX}{\mathsf{X}}
\nc{\msY}{\mathsf{Y}}
\nc{\msZ}{\mathsf{Z}}

\nc{\mtc}{\mathtt{c}}
\nc{\mtD}{\mathtt{D}}
\nc{\mte}{\mathtt{e}}
\nc{\mtE}{\mathtt{E}}
\nc{\mtf}{\mathtt{f}}
\nc{\mtF}{\mathtt{F}}
\nc{\mth}{\mathtt{h}}
\nc{\mtH}{\mathtt{H}}
\nc{\mtV}{\mathtt{V}}
\nc{\mtX}{\mathtt{X}}
\nc{\mty}{\mathtt{y}}

\nc{\uA}{\underline{\mathcal{A}}}
\nc{\uB}{\underline{\mathcal{B}}}
\nc{\uC}{\underline{\mathcal{C}}}
\nc{\uD}{\underline{\mathcal{D}}}
\nc{\uE}{\underline{\mathcal{E}}}
\nc{\uF}{\underline{\mathcal{F}}}
\nc{\uH}{\underline{\mathcal{H}}}
\nc{\uK}{\underline{\mathcal{K}}}
\nc{\uO}{\underline{\mathcal{O}}}
\nc{\uQ}{\underline{\mathcal{Q}}}
\nc{\uR}{\underline{\mathcal{R}}}
\nc{\uS}{\underline{\mathcal{S}}}
\nc{\uP}{\underline{\mathcal{P}}}
\nc{\uU}{\underline{\mathcal{U}}}
\nc{\uT}{\underline{\mathcal{T}}}
\nc{\uV}{\underline{\mathcal{V}}}
\nc{\uY}{\underline{\mathcal{Y}}}
\nc{\uZ}{\underline{\mathcal{Z}}}

\nc{\bv}{\breve{v}}

\nc{\tc}{\tilde{c}}
\nc{\tr}{\tilde{r}}
\nc{\ts}{\tilde{s}}
\nc{\tv}{\tilde{v}}

\nc{\sal}{\check{\alpha}}
\nc{\cbeta}{\check{\beta}}
\nc{\cd}{\check{d}}
\nc{\cf}{\check{f}}
\nc{\cdelta}{\check{\delta}}
\nc{\ccr}{\check{r}}
\nc{\cs}{\check{s}}

\nc{\bT}{\bar{T}}
\nc{\bt}{\bar{t}}

\nc{\cT}{\mathcal{T}}

\nc{\bcT}{\bar{\mathcal{T}}}
\nc{\bct}{\bar{\tau}}

\nc{\ad}{\mathsf{ad}}
\nc{\adl}{\mathsf{ad}_\mathsf{l}}
\nc{\adr}{\mathsf{ad}_\mathsf{r}}
\nc{\alg}{\mathsf{alg}}
\nc{\bialg}{\mathsf{bialg}}
\nc{\cl}{\mathsf{cl}}
\nc{\codim}{\mathsf{codim}}
\nc{\cop}{\mathsf{cop}}
\nc{\cork}{\mathsf{cork}}
\rnc{\cos}{\mathsf{cos}}
\rnc{\dim}{\mathsf{dim}}
\rnc{\exp}{\mathsf{exp}}
\nc{\ext}{\mathsf{ext}}
\nc{\fin}{\mathsf{fin}}
\nc{\gr}{\mathsf{gr}}
\nc{\grp}{\mathsf{grp}}
\nc{\h}{\mathsf{ht}}
\nc{\id}{\mathsf{id}}
\rnc{\int}{\mathsf{int}}
\rnc{\mod}{\mathsf{mod}}
\nc{\oi}{\mathsf{oi}}
\nc{\op}{\mathsf{op}}
\nc{\rev}{\mathsf{rev}}
\nc{\rk}{\mathsf{rk}}
\nc{\sgn}{\mathsf{sgn}}
\rnc{\sin}{\mathsf{sin}}
\nc{\supp}{\mathsf{supp}}
\rnc{\t}{\mathsf{t}}
\rnc{\tan}{\mathsf{tan}}
\nc{\tw}{\mathsf{tw}}

\nc{\Ad}{{\mathsf{Ad}}}
\nc{\Ann}{\mathsf{Ann}}
\nc{\Aut}{\mathsf{Aut}}
\nc{\Cyc}{\mathsf{Cyc}}
\nc{\Dih}{\mathsf{Dih}}
\nc{\End}{\mathsf{End}}
\nc{\Ext}{\mathsf{Ext}}
\nc{\For}{\mathsf{For}}
\nc{\Fun}{\mathsf{Fun}}
\nc{\Hom}{\mathsf{Hom}}
\nc{\Id}{\mathsf{Id}}
\nc{\Ima}{\mathsf{Im}}
\nc{\Ind}{\mathsf{Ind}}
\nc{\Inn}{\mathsf{Inn}}
\nc{\Ker}{\mathsf{Ker}}
\nc{\Ob}{\mathsf{Ob}}
\nc{\Out}{\mathsf{Out}}
\nc{\Pow}{\mathsf{Pow}}
\rnc{\Re}{\mathsf{Re}}
\nc{\Rep}{\mathsf{Rep}}
\nc{\Repfd}{\mathsf{Rep}_{\scsop{fd}}}
\nc{\RHom}{\mathsf{RHom}}
\nc{\Ser}{\mathsf{Serre}}
\nc{\Sym}{\mathsf{Sym}}
\nc{\Tr}{\mathsf{Tr}}
\rnc{\Vec}{\mathsf{Vec}}
\nc{\Vect}{\mathsf{Vect}}

\nc{\Modfd}{\mathsf{mod}^{\scriptstyle\mathsf{fd}}}
\nc{\Mod}{\mathsf{Mod}}

\nc{\CD}{\mathsf{CDec}}
\nc{\GSat}{\mathsf{GSat}}
\nc{\Sat}{\mathsf{Sat}}
\nc{\WSat}{\mathsf{WSat}}

\nc{\mfgl}{\mathfrak{g}\mathfrak{l}}
\nc{\mfsl}{\mathfrak{s}\mathfrak{l}}
\nc{\mfso}{\mathfrak{s}\mathfrak{o}}
\nc{\mfsp}{\mathfrak{s}\mathfrak{p}}

\nc{\GL}{\mathsf{GL}}
\nc{\SL}{\mathsf{SL}}
\nc{\Sp}{\mathsf{span}}

\nc{\bmodN}{\bmod \hspace{-3pt} '}

\nc{\Idiff}{I_{\mathsf{diff}}}
\nc{\Ieq}{\indI_{\mathsf{eq}}}
\nc{\Ins}{\indI_{\mathsf{ns}}}
\nc{\Insf}{\indI_{\mathsf{nsf}}}

\nc{\lrh}{\leftrightharpoons}
\nc{\iso}{\stackrel{\sim}{\longrightarrow}}
\nc{\liso}{\stackrel{\sim}{\longleftarrow}}
\nc{\lra}{\longrightarrow}
\nc{\ra}{\rightarrow}
\nc{\into}{\hookrightarrow}
\nc{\onto}{\twoheadrightarrow}

\nc{\lan}{\langle}
\nc{\ran}{\rangle}

\nc{\ot}{\otimes}
\nc{\ten}{\otimes}

\nc{\red}{\color{red}}
\nc{\blu}{\color{blue}}
\nc{\brn}{\color{brown}}
\nc{\grn}{\color{green!55!black}}
\nc{\gry}{\color{gray}}

\nc{\ie}{{\em i.e.},\ }
\nc{\eg}{{\em e.g.},\ }

\nc{\qu}{\quad}
\nc{\qq}{\qquad}
\nc{\tx}[1]{\qu\text{#1}\qu}

\nc{\aand}{\qquad\mbox{and}\qquad}

\renewcommand{\,}{\kern 0.1em} 

\nc{\Omit}[1]{}

\makeatletter
\newcommand*\rel@kern[1]{\kern#1\dimexpr\macc@kerna}
\makeatother

\rnc\appendixname{}

\nc{\FEnd}[1]{\mathsf{End}(#1)}
\nc{\FAlg}[2]{\mcE_{#1}^{#2}}
\nc{\FF}[1]{F_{#1}}
\nc{\uFF}[1]{\ul{F}_{#1}}
\nc{\tm}[1]{{\color{magenta}#1}}
\nc{\uG}{\ul{G}}
\nc{\pairing}[2]{\langle#1,#2\rangle}
\nc{\bifo}[2]{(#1,#2)}
\nc{\trunc}[3]{{#1}_{#2\to#3}}
\nc{\RG}{R^G}
\nc{\RGG}{R^{GG}}
\nc{\wtRG}{R_{\km}^G}
\nc{\km}{k}
\nc{\topS}[1]{S_{#1}}
\nc{\ulm}{\ul{m}}
\nc{\dgr}{\mathbb{D}}
\nc{\Br}[1]{\mathcal{B}_{#1}}
\nc{\BDm}{\mathcal{B}_{\dgr,\ulm}}
\nc{\gpS}[1]{\mathsf{s}_{#1}}
\nc{\vspan}[2]{\mathsf{span}_{#1}(#2)}
\nc{\indI}{I} 
\nc{\indJ}{J} 
\nc{\indX}{X} 
\nc{\indY}{Y} 
\nc{\gcmA}{{A}} 
\nc{\gcmB}{{B}} 
\nc{\gcmD}{{D}} 
\nc{\n}{\mfn}
\nc{\drc}[1]{\delta_{#1}}
\nc{\iip}[2]{(#1,#2)} 
\nc{\rootsys}{\Phi} 
\nc{\Qlat}{{\mathsf{Q}}}
\nc{\Plat}{{\mathsf{P}}}
\nc{\Qlate}{{\mathsf{Q}_\ext}}
\nc{\Plate}{{\mathsf{P}_\ext}}
\nc{\Pie}{\Pi_\ext}
\nc{\brS}[1]{S_{#1}}
\nc{\brSg}[1]{\wt{s}_{#1}}

\nc{\wtKM}{\Upsilon} 
\nc{\KM}[1]{K_{#1}}
\nc{\gKM}[2]{K_{#1}^{#2}}

\nc{\wtR}{\Xi}

\nc{\gfin}{\dot{\mfg}}

\nc{\shift}[1]{\Sigma_{#1}}
\nc{\UqLg}{U_qL\gfin}
\nc{\shrep}[2]{{#1}(#2)}
\nc{\fml}[2]{{#1}[\negthinspace[#2]\negthinspace]}
\nc{\Lfml}[2]{{#1}(\negthinspace(#2)\negthinspace)}
\nc{\sRM}[2]{R_{#1}(#2)}
\nc{\sKM}[2]{K_{#1}(#2)}
\nc{\UqLsl}[1]{U_qL\mathfrak{sl}_{#1}}
\nc{\Uqsl}[1]{U_q\mathfrak{sl}_{#1}}
\nc{\asl}[1]{\widehat{\mathfrak{sl}}_{#1}}
\nc{\ev}[1]{\mathsf{ev}_{#1}}
\nc{\wtev}[1]{\wt{\mathsf{ev}}_{#1}}
\nc{\prin}{\scriptstyle{\mathsf{pr}}}

\nc{\Ons}{U_q\mfk} 
\nc{\AOns}{U_q\mfk} 
\nc{\IOns}{U_q\mfk} 

\begin{document}

\title{Universal K-matrices for quantum Kac-Moody algebras}

\author{Andrea Appel}
\address{Dipartimento di Scienze Matematiche, Fisiche e Informatiche, Universit\`a di Parma,
Parco Area delle Scienze 53/A, 43124 Parma, Italy}
\email{andrea.appel@unipr.it}

\author{Bart Vlaar}
\address{Department of Mathematics, Heriot--Watt University, Edinburgh, EH14 4AS, UK \& Max Planck Institute for Mathematics, Vivatsgasse 7, 53111 Bonn, Germany}
\email{vlaar@mpim-bonn.mpg.de}


\thanks{The first author was supported in part by the ERC Grant 637618 and the Programme \emph{FIL} of the University of Parma co-sponsored by Fondazione Cariparma. The second author was supported in part by the EPSRC Grant EP/R009465/1.}	
	
\maketitle

\begin{abstract} 
We introduce the notion of a \emph{cylindrical} bialgebra, which is a quasitriangular bialgebra $H$ endowed with a universal K-matrix, \ie a universal solution of a generalized reflection equation, yielding an action of cylindrical braid groups on tensor products of its representations.
We prove that new examples of such universal K-matrices arise from quantum symmetric pairs of Kac-Moody type and depend upon the choice of 
a pair of generalized Satake diagrams. 
In finite type, this yields a refinement of a result obtained by Balagovi\'c and Kolb, producing a family of non-equivalent solutions interpolating between the {quasi}-K-matrix and the {\em full} universal K-matrix.
Finally, we prove that this construction yields formal 
solutions of the generalized reflection equation with a spectral parameter 
in the case of finite-dimensional representations over
the quantum affine algebra $U_qL\mathfrak{sl}_{2}$.

\end{abstract}

\setcounter{tocdepth}{1} 
\tableofcontents

\section{Introduction}

\subsection{}

In this paper, we extend the construction of the universal K-matrix for quantum groups corresponding to complex semisimple Lie algebras obtained by Balagovi\'{c} and Kolb in \cite{BK19} to the case of an arbitrary symmetrizable Kac-Moody algebra. 
Our approach relies on the notion of a \emph{cylindrical} bialgebra. 
Informally, this is a bialgebra endowed with a distinguished solution of a generalized reflection equation, which yields a natural action of cylindrical braid groups on the tensor products of its representations and generalizes the notion of {\em cylinder twist} introduced by tom Dieck and H\"{a}ring-Oldenburg and later used by Balagovi\'{c} and Kolb.
It bears a simple, yet crucial, difference with the latter in that the relevant reflection equation is twisted by an algebra automorphism which does not necessarily preserve the coproduct. 
However, its defect in being a morphism of quasitriangular bialgebras is controlled by a Drinfeld twist.

This more general framework allows us to construct new examples of universal K-matrices in the context of quantum Kac-Moody algebras. 
More specifically, given a symmetrizable Kac-Moody algebra $\mfg$ 
and an additional combinatorial datum (a pair of generalized Satake diagrams), we construct an algebra automorphism $\psi$ of $U_q\mfg$ and an operator $K$ satisfying the generalized reflection equation
\begin{align}\label{eq:gen-re}
	\begin{split}
		(\psi \ot \psi)(R_{21}) \cdot (1 \ot{\KM{}}) \cdot (\psi \ot \id)(R) \cdot (\KM{} \ot 1)=
		(\KM{} \ot 1) \cdot (\id \ot \psi)(R_{21}) \cdot (1 \ot{\KM{}}) \cdot R\,,
	\end{split}
\end{align}
where $R$ is the universal R-matrix of $U_q\mfg$. 
In finite type, our construction leads to new examples of non-equivalent universal K-matrices, where the Balagovi\'c-Kolb universal K-matrix is recovered as a special case.

\subsection{}
Reflection equations 
received much attention in the mathematical physics literature from the 1980s onwards, in particular in relation to quantum integrability, see \eg \cite{Ch84,Sk88,KS92,GZ94}.  
In this case, the reflection equation depends on an additional parameter, referred to as the \emph{spectral} parameter.
In the most general case, it takes the following form: 
\eq{ \label{eq:spectral-RE-intro}
		R^{--}_{21}(\tfrac{w}{z})  \cdot \id \ten K(w) \cdot 
	R^{-+}(zw) \cdot K(z) \ten \id 
	=
		K(z) \ten \id \cdot R^{-+}_{21}(zw) \cdot \id \ten K(w) \cdot R^{++}(\tfrac{w}{z})\,,
} 
where $R^{++}(z)$, $R^{-+}(z)$, and $R^{--}(z)$ are three, possibly distinct, solutions of a system of Yang-Baxter type equations with a spectral parameter, see \cite[Eqs.~(4.12)-(4.14)]{Ch92}.

Examples of matrix solutions of the reflection equation have been constucted in the context of finite-dimensional representations of quantum affine algebras and {\em quantum affine symmetric pairs}. 
In this case, the operators $R^{\pm\pm}(z)=R(z)$ are often assumed to be equal and determined by the action of the universal R-matrix \cite{Dr86}, thus yielding the \emph{standard} reflection equation.
Moreover, $K(z)$ is generally obtained as an intertwiner of the form
\[ 
K(z)\colon V(z)\to V(\tfrac{1}{z}) 
\] 
with respect to a distinguished coideal subalgebra, see \eg~\cite{DG02,DM03,RV16,BTs18}.
Our construction is tailored to provide a universal solution to this problem in greater generality. 
Namely, in the case of quantum affine algebras, the universal $K$-matrix converges on a finite-dimensional representation $V$ to a formal intertwiner $K_V(z):V(z)\to V^{\psi}(\tfrac{1}{z})$ and yields a solution of Cherednik's generalized reflection equation. 
A major advantage of our approach is that, by carefully chosing the automorphism $\psi$, the latter reduces to the standard case and our construction recovers many of the previously known solutions. 
In the last section of this paper, we consider the case of quantum affine $\mathfrak{sl}_2$, while the general (untwisted) case is discussed in \cite{AV22}.\\

In the rest of this introduction, we review the problem in more detail and outline our main results.

\subsection{}

Let $\mfg$ be a symmetrizable Kac-Moody algebra and $U_q\mfg$ the corresponding Drinfeld-Jimbo quantum group \cite{Dr85,Ji86,Lus94}.
It is well-known that $U_q\mfg$ is a non-commutative Hopf algebra, which, up to completion, is equipped with a quasitriangular structure given by the universal R-matrix $R$. 
The use of completion is made necessary by the fact that $R$ is defined only on certain tensor products of $U_q\mfg$-modules, \eg category $\mc O$ modules.\\ 

The quantum group $U_q\mfg$ is naturally endowed with a
family of distinguished subalgebras, which are not Hopf subalgebras,
but only (one-sided) coideal subalgebras.
Let $\mfk \coloneqq \mfg^\theta$ the fixed-point subalgebra of a Lie algebra involution $\theta$.
In finite type, building on work by Gavrilik and Klimyk \cite{GK91} and Koornwinder \cite{K93} in special cases, Noumi, Sugitani, and Dijkhuizen \cite{NS95,NDS97} and, independently, Letzter \cite{Le99,Le02,Le03} proved that the Hopf subalgebra $U\mfk \subseteq U\mfg$ is naturally deformed into a coideal subalgebra $U_q\mfk\subseteq U_q\mfg$, which we refer to as a \emph{quantum fixed-point coideal subalgebra}.
For symmetrizable Kac-Moody algebras, the construction of $U_q\mfk$ was obtained by Kolb in \cite{Ko14}, in the case of $\theta$ being an automorphism \emph{of the second kind}\footnote{
An automorphism $\theta\colon \mfg\to\mfg$ is of the \emph{second kind} if $\theta(\mfb^+) \cap \mfb^+$ is finite-dimensional, where $\mfb^+ \subset \mfg$ denotes the positive Borel subalgebra, see \eg \cite[4.6]{KW92}.
}.


\subsection{} 
For certain quantized fixed-point subalgebras of finite type, Ehrig and Stroppel
\cite{ES18} and Bao and Wang \cite{BW18} developed a coideal
analogue of the theory of canonical basis (cf.~\cite{Kas90, Lus90}).
Central in their results is the use of a coideal version of Lusztig's bar involution on $U_q\mfg$, \ie a bar involution on $U_q\mfk$, which we simply refer to as the {\em internal} bar involution. This yields a canonical element, which is known as the \emph{quasi-K-matrix}, which intertwines between 
Lusztig's bar involution and the internal bar involution on $U_q\mfk$ (see in particular \cite[Sec.~2.5]{BW18}).
In \cite{BK19}, Balagovi\'{c} and Kolb extended the construction of the quasi-K-matrix to 
every quantized fixed-point subalgebra $U_q\mfk$ of the symmetrizable Kac-Moody algebra $U_q\mfg$. 
In particular, in finite type, this led to the construction of a universal K-matrix.

\subsection{} 

The main goal of the present paper is to extend the construction of the universal K-matrix to the case of a symmetrizable Kac-Moody algebra.
Note that the formula of the full universal K-matrix in \cite{BK19} is not valid for infinite-dimensional Kac-Moody algebras, since it relies on the quantum Weyl group operator corresponding to longest element of the Weyl group, which only exists if $\mfg$ is finite-dimensional. 
The main property of this operator is to provide a description of the quasi-R-matrix as a multiplicative coboundary (cf.~\cite{KR90} and Eq. \eqref{quasiRX:factorized:1}), \ie it is essentially a half-balance on $U_q\mfg$
\cite{KT09, ST09},  and it is crucial in the construction of universal solutions of the reflection equation. There have been various attempts to define this operator for infinite-dimensional Kac-Moody algebras (\eg \cite{Ti10}), but none is suited to our purposes. 
More importantly, we aim to construct universal K-matrices which specialize to finite-dimensional representations of quantum affine algebras and yield solutions of generalized reflection equations with a spectral parameter. In particular, by restriction, the automorphism $\psi$ should induce the inversion of the spectral parameter on finite-dimensional representations.
This cannot be achieved within the existing framework of cylinder braided subalgebras used in \cite{BK19}, where $\psi$ is required to be an automorphism of quasitriangular bialgebras.

\subsection{} 

Our proposal is to bypass these obstructions altogether by adopting a new framework, which does not require the use of a global half-balance, while providing a generalization of the notion of cylinder twist. 
This prompts the definition of \emph{cylindrical bialgebras} (cf. Definition \ref{def:annular-bialg}). Roughly, this is 
the datum $(H, R, \psi, J, K)$, where $(H,R)$ is a quasitriangular bialgebra, $\psi:H\to H$ is an algebra automorphism, $J\in H\ot H$ a Drinfeld twist such that $H^{\cop,\psi} = H_{J}$, and finally $K\in H$ is an invertible
element satisfying the coproduct identity
\begin{equation*}
	\Delta(K) = J^{-1} \cdot (1 \ot K) \cdot (\psi \ot \id)(R) \cdot (K \ot 1)\,.
\end{equation*}
In particular, the datum $(\psi,J)$, which we refer to as a {\em twist pair}, is a twisted homomorphism $H^{\cop} \to H$ in the terminology of \cite{Da07}. 
It follows from the coproduct identity that $K$ is indeed a universal K-matrix, as it satisfies the generalized reflection equation \eqref{eq:gen-re}.
In particular, it yields an action of cylindrical braid groups on tensor products of its representations (cf. Proposition \ref{prop:twistedRE}).\\ 

Any automorphism of quasitriangular bialgebras $\varphi:(H,R)\to(H,R)$ automatically gives rise to the twist pair $(\varphi,R_{21}^{-1})$. 
Therefore, our definition recovers as a special case the notion of cylinder twists from \cite{tD98,tDHO98,BK19}. We shall refer to this case as a {\em strongly cylindrical bialgebra}. Note that, choosing the twist pair $(\varphi, R)$, we recover the analogue notion with the opposite convention used \eg in \cite{BZBJ18}.
More generally, in this framework,
we are able to describe a larger pool of operators which naturally appear in representation theory. 
For instance, the notion of a balance is one of the simplest cases of a cylinder twist, studied in detail in \cite{DKM03}. 
In contrast, a half-balance is not a cylinder twist.
However, both balances and half-balances are obtained as examples of solutions of generalized reflection equations in the context of cylindrical bialgebras.

\subsection{}
Our main result is the construction of a family of cylindrical structures on $U_q\mfg$ arising from quantum fixed-point coideal subalgebras. 
In fact, we prove that, given a quantum fixed-point coideal subalgebras $U_q\mfk$ with generalized Satake diagram $(X,\tau)$, there is a natural family of twist pairs $(\psi_{Y,\eta}, \mc R_{Y,\eta})$, indexed by an auxiliary generalized Satake diagram $(Y,\eta)$.
The Drinfeld twist $J=\mc R_{Y,\eta}$ is obtained by a suitable Cartan modification of the parabolic R-matrix corresponding to the subdiagram of finite type $Y$ and allows us to avoid the first obstruction due to the non-existence of a global half-balance in general.
We then adapt the approach of \cite{BK19} to this new setting,
constructing an operator $K_{Y,\eta}$, which acts on integrable category $\mc O$ $U_q\mfg$-modules as a $\psi$-twisted $U_q\mfk$-intertwiner and yields a (topological) cylindrical structure on $U_q\mfg$ (cf. Proposition \ref{thm:twistpair} and Theorem \ref{thm:kYeta:int-Delta}).\\

In finite type, we obtain a refinement of \cite{BK19}.
Suppose $U_q\mfg$ is a quantum group of finite type with Dynkin diagram $I$ and opposition involution $\oi_I$ (\ie diagram automorphism corresponding to the longest element of the Weyl group). Then $(Y,\eta)=(I,\oi_I)$ is a Satake diagram. 
In this case, we recover the universal K-matrix $K_{I,\oi_I}$ obtained by Balagovi\'c and Kolb (up to conventions).
On the other hand, if $(Y,\eta)=(X,\tau)$, then $K_{X,\tau}$ coincides with quasi-K-matrix (up to a Cartan factor).
Therefore, we obtain a discrete family of universal K-matrices interpolating between the Balagovi\'c-Kolb universal K-matrix and the quasi-K-matrix.

\subsection{}
The construction of universal K-matrices for Kac-Moody algebras involves a number of additional generalizations 
and simplifications with respect to the construction given in \cite{BK19}, which we briefly summarize below. The first two  describe the more general setting in which the main results are valid (Proposition \ref{thm:twistpair} and Theorem \ref{thm:kYeta:int-Delta}).

\subsubsection{Quantized pseudo-fixed-point subalgebras}

The K-matrix construction of \cite{BK19} applies to coideal subalgebras $U_q\mfk$ which are q-deformed enveloping algebras of fixed-point subalgebras with respect to an involutive automorphism of $\mfg$. 
In \cite{RV20} this construction was extended to more general subalgebras of $\mfg$, called
\emph{pseudo-fixed-point subalgebras} and defined in terms of \emph{generalized Satake diagrams} (see also~\cite{RV21}).
Note that in this setting the description of the automorphism of $\mfg$ 
and its quantization is somewhat simpler, as one no longer needs to keep track of the correction given by a multiplicative character of the root lattice with values in $\{ \pm 1 \}$ (see Section \ref{s:coideal:1}).
Our construction of universal K-matrices is presented in this more general setting.

\subsubsection{The quasi-K-matrix and parameter constraints} 
In \cite{BK19}, the parameters involved in the definition of $U_q\mfk$ are assumed to be invariant under a particular diagram automorphism (cf.~\cite[Eq.~(7.4)-(7.5)]{BK19}). 
In our approach we do not need this assumption.
Additional constraints on the parameters are imposed in \cite[Sec.~5.4]{BK19} in order to guarantee the existence of an {\em internal}
bar involution on $U_q\mfg$.
The latter is indeed a crucial ingredient in the construction of the quasi-K-matrix given in \cite{BK19}.
In this paper, we provide a construction of the quasi-K-matrix, which does not rely on the internal bar involution and therefore applies to a larger class
of coideal subalgebras. Moreover, as later observed by Kolb in \cite{Ko21}, this construction of the quasi-K-matrix can be used to {\em define} the internal bar involution. 
We obtain this generalization by directly extending the arguments in \cite[Sec.~6]{BK19}, making use of the \emph{fundamental lemma of quantum symmetric pairs} \cite[Thm.~ 4.1]{BW21} and of the simplification discussed in \cite[Sec.~3.5]{DK19}.
Note that in the quasi-split case an alternative construction 
of a (weakly) universal K-matrix without parameter constaints was given in \cite{KY20}.

\subsubsection{Coproduct identity} 
Beyond the quasi-K-matrix, the formula for the universal K-matrix given by Ba\-lagovi\'{c} and Kolb involves the quantum Weyl group operators\footnote{That is, the braid group operators constructed in \eg~\cite[Ch.~5]{Lus94} and given in terms of q-deformed triple exponentials.} and a correcting factor in a completion of the quantum deformed Cartan subalgebra, see \cite[Eq.~(8.1)]{BK19}. 
This makes the computation of the coproduct identity of the universal
K-matrix rather complicated, see \cite[Sec. 8-9, Thm.~9.5]{BK19}.
Following \cite{KT09}, we introduce Cartan-modified quantum Weyl group operators, whose Cartan correction depends upon the choice
of a generalized Satake diagram. These can be thought of as modified
\emph{diagrammatic half-balances} (see Sec.~\ref{ss:diag-half-bal}) and yield 
universal K-matrices whose coproduct identity is easier to compute.

\subsubsection{Intertwining equation} 
The quasi-K-matrix constructed in Sec.~\ref{s:quasiK} is related to the 
original one via Lusztig's bar involution. This simple, and yet subtle, difference allows us to straightforwardly derive the intertwining equation of the standard K-matrix from those of its factors.
Note that this is in contrast with the proof of the intertwining equation in \cite[Thm.~7.5]{BK19}, which does not directly use the intertwining equation of the quasi-K-matrix \cite[Prop.~6.1]{BK19}.
Moreover, if $\mfk$ is a fixed-point subalgebra, it becomes clear that at $q=1$ the standard K-matrix reduces to an element of the centralizer of $U\mfk$ in a completion of $U\mfg$.


\subsection{Outline}
In Section \ref{s:alm-cyl}, we introduce the notions of twist pair and cylindrical bialgebra
 (Definition \ref{def:annular-bialg}). 
This more general framework is first described in purely algebraic terms. 
We then rely on the usual Tannakian formalism to extend it to topological bialgebras. 
In Section \ref{s:km}, we recall several facts about symmetrizable Kac-Moody algebras and their automorphism groups.
In particular, we recall the definition of {\em framed realizations compatible with a diagram automorphism} and we prove that such realizations do not always exist, providing a necessary and sufficient condition in the corank one case (Proposition \ref{prop:taucompatible}).
In Section \ref{s:QG}, we review the basic theory of Drinfeld-Jimbo quantum groups, their category $\mc O$ representations, and the universal R-matrix. 
In particular, we describe a factorization of the quasi-R-matrix with respect to a subdiagram of arbitrary type (Proposition \ref{prop:quasiR:quotient}). 
In Section \ref{s:qW-int}, we recall the definition of the quantum Weyl group operators on integrable representations and their basic properties. 
In Section \ref{s:coideal}, we consider classical and quantum \emph{pseudo}-fixed-point subalgebras, combinatorially described in terms of \emph{generalized} Satake diagrams (Definition \ref{Uqk:def}).
The corresponding quantum \emph{pseudo}-involutions are defined in terms of \emph{modified diagrammatic half-balances} (cf. Section \ref{ss:mod-lusz-op}).
In Section \ref{s:quasiK}, we revisit and generalize the construction of the quasi-K-matrix 
(Theorem \ref{thm:sum-quasi-k}). 
In Section \ref{s:universalk} we modify the quasi-K-matrix with the multiplicative difference of two modified diagrammatic half-balances corresponding to a pair of generalized Satake diagrams. 
This leads to a family of solutions of the generalized reflection equation, inducing on $U_q\mfg$ a cylindrical structure with respect to which $U_q\mfk$ is a cylindrically invariant coideal subalgebra (Theorems \ref{thm:kX} and \ref{thm:kYeta:int-Delta}).  
In Section \ref{s:affine}, we briefly discuss the application of our constructions to the case of quantum symmetric pairs for the quantum loop algebra $U_qL\mathfrak{sl}_2$, showing that universal K-matrices constructed in Section \ref{s:universalk} give rise to formal solutions of a generalized reflection equation with a spectral parameter.


\subsection{Acknowledgements}
The authors would like to thank Martina Balagovi\'{c}, Ivan Cherednik, Anastasia Doikou, Pavel Etingof, Sachin Gautam, David Jordan, Stefan Kolb, Vidas Regelskis, Nicolai Reshetikhin, Jasper Stokman, Valerio Toledano Laredo, Tim Weelinck, and Robert Weston for useful comments and
discussions. 


\section{Cylindrical bialgebras} \label{s:alm-cyl}


In this section, we introduce the notion of a \emph{cylindrical} bialgebra, 
which is roughly a quasitriangular bialgebra $H$ together with an action of the cylindrical braid group on its representations. The main ingredient is a distinguished solution of a generalized 
reflection equation which depends upon the choice of an algebra automorphism $\psi: H \to H$ (a twisting operator) whose defect in being a morphism from $H$ to $H^\cop$ is controlled by a
Drinfeld twist. As a special case, we recover the notion of balanced and half-balanced bialgebras \cite{KT09,ST09}, and that of cylinder-braided bialgebras as they appeared in \cite{tD98, tDHO98, BK19} (see also \cite{DKM03, E04, Bro12}). 
This more general framework shall be used in Section \ref{s:universalk} to describe the representations of the cylindrical braid group arising from quantum Kac-Moody algebras. 


\subsection{Quasitriangular bialgebras}\label{ss:qt-hopf}

Recall that by \cite{Dr90a} a \emph{quasitriangular bialgebra} is a pair $(H,R)$ 
where $H$ is a bialgebra (over a base field $\F$) and $R$ is an element of $(H\ot H)^\times$, 
called \emph{universal R-matrix}, satisfying the intertwining identity
\begin{gather}
\label{R:intw:sec2} R \; \Del(x) = \Del^\op(x) \; R \,,
\end{gather}
for any $x\in H$, and the coproduct identities
\begin{gather}
\label{R:Del} (\Del \ot \id)(R) = R_{13} R_{23} \qq\mbox{and} \qq (\id \ot \Del)(R) = R_{13} R_{12}\,, 
\end{gather}
where $\Del$ denotes the coproduct and $\Del^\op = (12) \circ \Del$ the opposite coproduct. 
Note that, if $(H,R)$ is a quasitriangular bialgebra and $\eps$ is the counit of $H$, 
then
\eq{ \label{R:eps}
(\varepsilon\ot\id)(R)=1=(\id\ot\varepsilon)(R).
}
Moreover, $(H,R_{21}^{-1})$ and $(H^\cop,R_{21})$  are also quasitriangular bialgebras, 
where $H^\cop$ denotes the co-opposite bialgebra, obtained from $H$ by replacing $\Del$ by $\Del^\op$ and leaving the other structure maps as they are. 
From \eqref{R:intw:sec2} and either coproduct formula in \eqref{R:Del} 
it follows that $R$ is a solution of the {\em Yang-Baxter equation}
\begin{equation}\label{eq:QYB}
R_{12}R_{13}R_{23}=R_{12} (\Delta\ot\id)(R)=(\Delta^{\op}\ot\id)(R) R_{12}=R_{23}R_{13}R_{12}
\end{equation}
and thus it induces a representation of the standard braid groups on the tensor powers of $H$. 
Namely, let $\Del^{(n)} : H \to H^{\ot n}$ for $n \in \Z_{\ge 1}$ be the {\em iterated} coproducts
defined by setting 
\eq{
\Delta^{(1)}\coloneqq \id_H \qq\mbox{and}\qq \Delta^{(n)}\coloneqq (\Delta\ot\id^{\ot(n-2)} ) \circ\Delta^{(n-1)} \qq (n>1).
} 
In particular, $\Del^{(2)} = \Del$. Note that $\Del^{(n)}$ yields a natural action of $H$ on 
$H^{\ten n}$ given by
\eq{
	x \cdot (h_1 \ot \cdots \ot h_n) = \Del^{(n)}(x) (h_1 \ot \cdots \ot h_n).
} Let $\Br{n}$ be the braid group of $n$ strands in the plane, presented on the generators $\topS{1},\dots,\topS{n-1}$ subject to the Artin relations
\begin{equation}\label{eq:braid-rel}
\topS{i}\cdot\topS{i+1}\cdot\topS{i}=\topS{i+1}\cdot\topS{i}\cdot\topS{i+1}
\aand
\topS{i}\cdot\topS{j}=\topS{j}\cdot\topS{i}
\end{equation}
for any $i=1,\dots, n-2$ and $|i-j|>1$, respectively. 
For any $n \in \Z_{\geqslant 2}$, the assignment
\begin{equation}\label{eq:braid-action-H}
\mu_{R}^n(\topS{i}) = (i\, i+1) \circ R_{i,i+1}, \qq i \in \{ 1,\ldots, n-1 \},
\end{equation}
where $R_{i,i+1}$ is shorthand for left multiplication by $R_{i,i+1}$, defines a morphism of groups $\mu_{R}^n:\Br{n}\to\Aut_H(H^{\ot n})$, \ie an action of $\Br{n}$ on $H^{\ten n}$ which commutes 
with the action of $H$.


\subsection{Artin-Tits groups}\label{ss:braid-gp-diag}
The braid group $\Br{n}$ is the Artin-Tits group corresponding to the Coxeter group $S_n$. 
It is well-known that Artin-Tits group have a combinatorial description in 
terms of {\em labelled diagrams},
where a {\it diagram} is an undirected graph $\dgr$ with no multiple edges or loops and
a {\it labelling} $\ulm$ on $\dgr$ is the assignment of an integer $m_{ij}\in\{2,3,\ldots,\infty\}$ 
to any pair $i,j$ of distinct vertices of $\dgr$ such that
\[
m_{ij}=m_{ji} \aand m_{ij}=2 \text{ if and only if } i \text{ and } j \text{ are not joined by an edge.}
\]
By \cite{BS72,Del72}, the \emph{Artin-Tits group} corresponding to a diagram $\dgr$ with labelling $\ulm$ is the group $\BDm$ with generators $\topS{i}$, where $i$ runs through the vertices of $\dgr$, and relations
\begin{flalign}\label{eq:gen-braid}
	&&	\underbrace{\topS{i}\cdot \topS{j}\cdot \topS{i}\;\cdots\;}_{m_{ij}}=
	\underbrace{\topS{j}\cdot \topS{i}\cdot \topS{j}\;\cdots\;}_{m_{ij}} &&
\end{flalign}

For $n \in \Z_{\geqslant 2}$, consider the Coxeter-Dynkin diagram $\dgr = {\sf A}_{n-1}$ with the following (standard) labelling: the vertex set is $\{ 1,2,\ldots,n-1 \}$ with $m_{ij}=3$ if $|i-j|=1$ and $m_{ij}=2$ otherwise.
The corresponding braid group $\BDm$ 
coincides with $\Br{n}$, see \eqref{eq:braid-rel}.  
The diagram $\dgr={\sf B}_n$ arises as an extension of ${\sf A}_{n-1}$ by including a vertex 0 with additional labelling datum $m_{01}=m_{10}=4$ and $m_{0i}=m_{i0}=0$ if $i>0$.
The corresponding braid group $\BDm$ is presented on the generators $\topS{0},\topS{1},\dots, \topS{n-1}$ subject to the relations \eqref{eq:braid-rel} and 
\begin{equation}\label{eq:ann-braid-rel}
	\topS{0}\cdot\topS{1}\cdot\topS{0}\cdot\topS{1}=\topS{1}\cdot\topS{0}\cdot\topS{1}\cdot\topS{0}
\end{equation}
Moreover, it contains an isomorphic copy of $\Br{\mathsf{A}_{n-1}}$ and identifies with the group $\Br{n}^{\sf cyl}$ of  \emph{cylindrical braids} (such  topological interpretations of Artin-Tits groups of finite type were given in general in \cite{Bri71}). We are interested in producing representations of 
cylindrical braid groups $\Br{n}^{\sf cyl}$ in terms of suitable bialgebras 
as in Section~\ref{ss:qt-hopf}.



\subsection{Drinfeld twists}\label{ss:twist}

A {\em Drinfeld twist} of a bialgebra $H$ is an element $J\in (H\ot H)^\times$ satisfying the normalization $(\varepsilon\ot\id)(J)=1=(\id\ot\varepsilon)(J)$ and the {\em cocycle identity}
\begin{equation}\label{eq:cocycle}
(J\ot 1) \cdot (\Delta\ot\id)(J) = (1\ot J) \cdot (\id\ot\Delta)(J) \,.
\end{equation}
Drinfeld twists allow us to modify the quasitriangular structure of $(H,R)$.
Indeed, given a Drinfeld twist $J$, one obtains a new quasitriangular bialgebra $(H_J, R_J)$ where $H_J$ is the bialgebra $H$ with $\Del$ replaced by the twisted coproduct $\Del_J$ defined by
\begin{equation*}
\Delta_J(x)=J\cdot\Delta(x)\cdot J^{-1}\qquad x\in H
\end{equation*}
and with the other structure maps unchanged; furthermore the twisted R-matrix is given by
\eq{
R_J\coloneqq J_{21}\cdot R\cdot J^{-1}.
} 
If $J'$ is a Drinfeld twist for $H$ and $J$ is a Drinfeld twist for $H_{J'}$, then $J \cdot J'$ is a Drinfeld twist for $H$ satisfying $H_{J \cdot J'} = (H_{J'})_J$ and $R_{J \cdot J'} = (R_{J'})_J$.
In general, $H$ and $H_J$ are not isomorphic bialgebras. However, they give rise to isomorphic braid group representations as $\mu_{H_J,R_J}^n=\Ad(J^{(n)})\circ\mu_{H,R}^n$.
Here $\Ad(\mathsf{X})$ denotes the conjugation by an invertible element $\mathsf{X}$ and $J^{(n)}$ is defined recursively by
\eq{
J^{(2)}\coloneqq J \qq\mbox{and}\qq J^{(n)}\coloneqq (J^{(n-1)}\ot 1) \cdot (\Delta^{(n-1)}\ot\id)(J) \qq(n>2).
}
New Drinfeld twists can be obtained by gauging (see \eg~\cite{ATL19a}).

\begin{remark}
One checks immediately that the Yang-Baxter equation for a quasitriangular bialgebra $(H,R)$ coincides with the cocycle identity for the R-matrix. 
Thus, the R-matrix $R\in H\ot H$ is a Drinfeld twist and $(H_R,R_R) = (H^{\cop},R_{21})$. 
Hence $(H_{R_{21}R},R_{R_{21}R}) = (H,R)$, that is, $R_{21}R\in H\ot H$ is an $(H,R)$-invariant Drinfeld twist.
\hfill\rmkend
\end{remark}


\subsection{Twist pairs}\label{ss:triv-gauge}
Let $(H,R)$ be a quasitriangular bialgebra and $\psi:H\to H$ an algebra automorphism.
The \emph{$\psi$-twisting} of $(H,R)$ is the quasitriangular bialgebra $(H^{\psi}, R^{\psi\psi})$ obtained from $(H,R)$ by pullback through $\psi$, \ie $H^\psi$ is the bialgebra with modified coproduct and counit:
\eq{
\Delta^{\psi} \coloneqq (\psi\ot\psi) \circ \Del \circ \psi^{-1}, \qq \eps^{\psi} \coloneqq \eps \circ \psi^{-1}
}
and the modified universal R-matrix given by $R^{\psi\psi} \coloneqq (\psi\ot\psi)(R)$. Note
that, by construction, $\psi$ is an isomorphism of quasitriangular bialgebras $(H,R) \to (H^{\psi}, R^{\psi\psi})$.

\begin{definition}\label{def:annular-bialg-twist-pair}
	Let $(H,R)$ be a quasitriangular bialgebra. 
		A {\em twist pair} $(\psi,J)$ is the datum of an algebra automorphism $\psi:H\to H$ and a Drinfeld twist $J\in H\ot H$ such 
		that $H^{\cop,\psi} = H_{J}$, \ie
		\begin{flalign}
&& \Delta^{\op,\psi}=\Ad(J) \circ \Delta\,,\qq\epsilon^{\psi}=\epsilon\, , \qq \mbox{and} \qq R_{21}^{\psi\psi}=J_{21}\cdot R\cdot J^{-1}\, . && \defnend
		\end{flalign}	
\end{definition}
Note that, in the terminology of \cite[Sec. 2.1]{Da07}, $(\psi, J)$ is a {\it twisted homomorphism of bialgebras $H^{\cop} \to H$}.

\subsection{Cylindrical bialgebras}\label{ss:alm-cyl-hopf}

We now introduce a class of bialgebras which naturally give rise to representations of cylindrical braid groups, in analogy with the case of quasitriangular bialgebras and braid groups of type $\mathsf{A}$.

\begin{definition}\label{def:annular-bialg}
Let $(H,R)$ be a quasitriangular bialgebra. 
\begin{enumerate}\itemsep0.25cm
\item 
We say that $(H,R)$ is \emph{cylindrical} if there exists a twist pair $(\psi,J)$ and an element $K\in H^\times$, called a {\em universal K-matrix}, such that the following coproduct identity holds
\begin{equation}\label{eq:k-coprod}
\Delta(K) = J^{-1} \cdot (1 \ot K) \cdot (\psi \ot \id)(R) \cdot (K \ot 1)
\end{equation}
\item 
A subalgebra $B\subseteq H$ is said to be \emph{cylindrically invariant} if
\begin{flalign}\label{eq:k-intertw}
&& K \cdot b = \psi(b) \cdot K  &&
\end{flalign}
for all $b\in B$. \hfill \defnend
\end{enumerate}
\end{definition}

We shall prove that any cylindrical bialgebra $H$ gives rise to a representation of the cylindrical braid group $\Br{n}^{\sf cyl}$ on $H^{\ot n}$. 
More precisely, the action of $\Br{n}^{\sf cyl}$ extends the action of $\Br{n}$ given by the R-matrix and it is therefore determined by the K-matrix. 
Whenever the subalgebra $B$ is a right coideal, i.e.
\eq{
\Del(B) \subseteq B \ot H,
}
it allows us to describe this action {\em internally}, that is, in terms of $B$-intertwiners.

\begin{proposition} \label{prop:twistedRE}
Let $(H,R,\psi, J,K)$ be a cylindrical bialgebra.
\begin{enumerate}\itemsep0.25cm
\item The $(\psi, J)$-twisted K-matrix $K\in H$ satisfies the {\em generalized reflection equation}
\begin{equation}\label{eq:psi-Q-re}
(K\ot 1) \cdot (R^\psi)_{21} \cdot (1\ot K) \cdot R = R^{\psi \psi}_{21} \cdot (1\ot K) \cdot R^\psi \cdot (K\ot 1)
\end{equation}
where $R^\psi \coloneqq (\psi \ot \id)(R)$.
\item Let $B\subseteq H$ be a cylindrically invariant coideal subalgebra. 
There is a canonical morphism of groups $\mu_{R,K}^n:\Br{n}^{\sf cyl}\to\Aut_B(H^{\ot n})$
given by the assignment
\begin{equation}\label{eq:braid-action-BH}
\mu_{R,K}^n(\topS{0}) = (\psi^{-1}\ot\id^{n-1}_{H}) \circ (K\ot 1^{\ot n-1}) 
\qquad
\mbox{and}
\qquad
\mu_{R,K}^n(\topS{i})=(i\, i+1)\circ R_{i,i+1}.
\end{equation}
\end{enumerate}
\end{proposition}

\begin{proof}
(i) It is enough to observe that, since $R\cdot\Delta(K)\cdot R^{-1}=\Delta^\op(K)$, one has
\begin{align*}
R\cdot J^{-1} \cdot (1 \ot K) \cdot R^\psi \cdot (K \ot 1) =
J_{21}^{-1} \cdot (K \ot 1) \cdot (R^\psi)_{21} \cdot (1 \ot K) \cdot R
\end{align*}
Then \eqref{eq:psi-Q-re} follows from $(\psi\ot\psi)(R_{21})=J_{21}\cdot R\cdot J^{-1}$.\\

(ii) We have to show that $\mu_{R,K}^n$ preserves the four--term relation \eqref{eq:ann-braid-rel}.
We may assume $n=2$. Set $\mu_{R,K} \coloneqq \mu_{R,K}^2$. Then, we have
\begin{align*}
\mu_{R,K}(\topS{0})&\circ\mu_{R,K}(\topS{1})\circ\mu_{R,K}(\topS{0})\circ\mu_{R,K}(\topS{1})
=\\
&=
(\psi\ot\psi)^{-1}\circ \big( (K\ot 1) \cdot (R^\psi)_{21} \cdot (1\ot K) \cdot R \big)  \\
&=
(\psi\ot\psi)^{-1}\circ \big( R^{\psi \psi}_{21} \cdot (1\ot K) \cdot R^\psi \cdot (K\ot 1) \big)\\
&=\mu_{R,K}(\topS{1})\circ\mu_{R,K}(\topS{0})\circ\mu_{R,K}(\topS{1})\circ\mu_{R,K}(\topS{0})
\end{align*}
where the second identity is the generalized reflection equation \eqref{eq:psi-Q-re}.
The result follows.
\end{proof}


\subsection{The trivial example}
It is important to observe that the representations of $\Br{n}^{\sf cyl}$ arising from cylindrical bialgebras are in general \emph{genuinely} cylindrical in that they cannot be recovered by the inclusion $\Br{n}^{\sf cyl}\subset\Br{n+1}$.
Indeed, $\Br{n}^{\sf cyl}$ identifies with the subgroup of braids on $n+1$ strands which fix a distinguished strand, mapping $\topS{i}\mapsto\topS{i+1}$ if $i\neq 0$ and $\topS{0}\mapsto\topS{1}^2$. 
Therefore, if $(H,R)$ is a quasitriangular bialgebra, we obtain an $H$-invariant action of $\Br{n}^{\sf cyl}$ on $H^{\ot(n+1)}$
\[
\wt{\mu}_R^{n+1}:\Br{n}^{\sf cyl}\to\Aut_{H}(H^{\ot(n+1)})
\]
which is the restriction of $\mu_R^{n+1}$ and therefore it is given by $\wt{\mu}_R^{n+1}(\topS{0})=R_{21}R$. 
Clearly, this can be further restricted to an action $\ol{\mu}_R^{n+1}$ on the subspace $H^{\ot n}=1\ot H^{\ot n}\subset H^{\ot(n+1)}$,  relying on the projection $\varepsilon\ot\id^{\ot n}: H^{\ot(n+1)}\to H^{\ot n}$ given by the counit. 
By \eqref{R:eps} the result is quite uninteresting as one gets $\ol{\mu}_R^{n+1}(\topS{0})=(\varepsilon\ot\id)(R_{21}R)=1$. 
This shows that any quasitriangular bialgebra $(H,R)$ is endowed with a {\em trivial}  cylindrical structure given by $\psi=\id_H$, $J=R$, and $K=1$. There are on the other
hand many non-trivial examples as we describe below.


\subsection{Balanced and half-balanced bialgebras}\label{ss:bal-half-bal}
By \cite{KT09, ST09}, a quasitriangular bialgebra $(H,R)$ is
\begin{enumerate}\itemsep0.25cm
	\item
	{\em balanced} if there exists an element $b\in H^\times$, called {\em balance}, such that $b \in Z(H)$ and $\Delta(b) = (b\ot b) R_{21} R$;
	\item 
	{\em half-balanced} if there exists an element $h\in H^\times$ called {\em half-balance}, such that $h^2\in Z(H)$ and $\Delta(h) = (h\ot h) R$.
\end{enumerate}
Note that, if $h$ is a half-balance, then $h^{2}$ is a balance, since
$R \Delta(h) R^{-1}=\Delta^{\op}(h)$ and $R(h\ot h) = (h\ot h) R_{21}$.
Balances and half-balances are examples of universal K-matrices.

\begin{proposition} \label{prop:balances}
\hfill
\begin{enumerate}\itemsep0.25cm
\item 
Let $(H,R)$ be a quasitriangular bialgebra with balance $b$. 
Then $H$ is cylindrical with $\psi=\id$, $J=R_{21}^{-1}$, and $K=b$. 
Moreover, $H$ is cylindrically invariant.
\item 
Let $(H,R)$ be a quasitriangular bialgebra with half-balance $h$. 
Then $H$ is cylindrical with $\psi=\Ad(h)$, $J=1\ot 1$, and $K=h$. 
Moreover, $H$ is cylindrically invariant.
\end{enumerate}
\end{proposition}

\begin{proof} \hfill
\begin{enumerate}\itemsep0.25cm
\item 
It is clear that $H^{\cop}=H_{R_{21}^{-1}}$, so that $(\id, R_{21}^{-1})$ is a twist pair.
Assuming that $K$ is central, the coproduct identity \eqref{eq:k-coprod} becomes
\[
\Delta(K)=R_{21} (1\ot K) R (K\ot 1) = (K\ot K) R_{21}  R
\]
so that $K=b$ is an admissible solution, which clearly commutes with every element in $H$.
\item Note that a half-balance $h\in H$ is a {\em gauge transformation} which {\em trivializes} the R-matrix, \ie we have $(h\ot h) \Delta(h)^{-1} = R_{21}^{-1}$ and $(h^{-1}\ot h^{-1}) \Delta(h) = R$.
Note that $\Ad(h)^2=\id$ and indeed 
\[
H^{\Ad(h)} =  H_{R_{21}^{-1}} = H^{\cop} = H_R = H^{\Ad(h)^{-1}}.
\]
In particular, $(\Ad(h), 1\ot 1)$ is a twist pair and the coproduct identity \eqref{eq:k-coprod} becomes
\[
\Delta(K) = (1\ot K) \cdot (\Ad(h)\ot \id)(R) \cdot (K\ot 1).
\]
Thus, $K=h$ is a solution. Finally, note that the intertwining equation \eqref{eq:k-intertw}
becomes trivial, since $K \cdot x  = \Ad(h)(x) \cdot h=\psi(x) \cdot K$ for any $x\in H$. \hfill \qedhere
\end{enumerate}
\end{proof}

\begin{remark}
Recall that a quasitriangular Hopf algebra is \emph{ribbon} if it admits a balance $b$ fixed by the antipode, in which case it is called a ribbon element.
The interplay between ribbon elements and the reflection equation was first observed by Donin, Kulish and Mudrov \cite{DKM03}.
The notion of half-balance is due to Kamnitzer--Tingley, Snyder--Tingley \cite{KT09, ST09} and Enriquez \cite{E10}. 
It would be interesting to see if the approach in \cite{DKM03} extends to half-balances.
\hfill\rmkend
\end{remark}


\subsection{Strongly cylindrical bialgebras}\label{ss:cylindrical-bialg}
We describe now a special case of cylindrical bialgebras, which first appeared in the work of tom Dieck and H\"aring-Oldenburg \cite{tD98,tDHO98} and later in the work of Balagovi\'{c}-Kolb \cite{BK19}, under the name \emph{bialgebras with a (twisted) cylinder twist}. 
It corresponds to setting $J= R_{21}^{-1}$ in Definition \ref{def:annular-bialg}; equally we may set $J=R$ which corresponds to the convention used in \cite{BZBJ18}.

\begin{definition}
We call a quasitriangular bialgebra $(H,R)$ {\em strongly cylindrical} if there exists a bialgebra automorphism 
$\varphi$ of $(H,R)$ and an element $K\in H^\times$ such that
\begin{flalign}\label{eq:k-coprod-BK}
&& \Delta(K)=R_{21} \cdot (1\ot K) \cdot R^\varphi \cdot (K\ot 1), && 
\end{flalign}
where $R^\varphi\coloneqq \varphi\ten\id(R)$.
\hfill\defnend
\end{definition}

The following motivates our choice of terminology in Definition \ref{def:annular-bialg}.

\begin{proposition}
Let $(H,R, \varphi, K)$ be a strongly cylindrical bialgebra. 
Then, $(\varphi, R_{21}^{-1},K)$ is a cylindrical structure on $(H,R)$, i.e. $(H,R)$ is a cylindrical bialgebra with twist pair $(\varphi,R_{21}^{-1})$ and universal K-matrix $K$.
\end{proposition}

\begin{proof}
Since $\varphi$ is a quasitriangular bialgebra automorphism, $H^{\varphi}=H$ and $R^{\varphi\varphi}=R$. 
Thus, $H^{\cop,\varphi}=H_{R_{21}^{-1}}$ and  $(\varphi, R_{21}^{-1})$ is a twist pair. 
The coproduct identity \eqref{eq:k-coprod} then reduces to \eqref{eq:k-coprod-BK}.
\end{proof}

\begin{remark}
Note that, for any quasitriangular bialgebra automorphism $\varphi$, $(\varphi, R)$ and $(\varphi, R_{21}^{-1})$ are always twist pairs and represent two standard choices. 
Our more general notion of cylindrical bialgebra aims to relax this condition on $\varphi$ by allowing less obvious twist pairs $(\psi, J)$ and new examples of universal K-matrices. 
For instance, note that, while balances define strongly cylindrical structures on quasitriangular bialgebras, half-balances in general do not.
However, they arise from the more general notion of cylindrical structure under consideration here. 
\hfill\rmkend
\end{remark}


\subsection{Tannakian formalism and completions}\label{ss:tannakian}

It is well--known that the purely algebraic setting we described above is in general too restrictive to describe interesting solutions of the Yang-Baxter and the reflection equations. 
Indeed, in the cases of our interest, we should rather consider \emph{pseudo} structures (cf.\ \cite{Dr86}) in that the defining 
operators, \eg R-matrices and K-matrices, are not algebraic but rather \emph{topological}, \ie they correspond to elements in a suitable completion $\wh{H}$ of the bialgebra $H$.
In general $\wh{H}$ is only a topological bialgebra, whose structure involves {\em completed} tensor products.\\

Our approach to describe such topological bialgebras is based on the well--known \emph{Tannakian formalism} \cite{D90}. 
We implicitly describe the completion $\wh{H}$ in terms of operators acting on a distinguished subcategory of $H$-modules and commuting with every $H$-intertwiner. 
This approach yields a canonical morphism $H\to\wh{H}$. Namely, let $H$ be an algebra, $\mcC\subseteq\Mod(H)$ a distinguished full subcategory and $F:\mcC\to\Vect$ the forgetful functor. 
Let $H^{\mcC}\coloneqq\End(F)$ be the algebra of natural transformations of $F$. 
Recall that, by definition, an element $\xi\in H^{\mcC}$ is a collection of operators $\xi_V: F(V)\to F(V)$, indexed by $V\in\mcC$, such that the diagram
\[
\begin{tikzcd}
F(V) \arrow[r,"\xi_V"]\arrow[d,"F(f)"']& F(V)\arrow[d,"F(f)"]\\
F(W) \arrow[r,"\xi_W"']& F(W)
\end{tikzcd}
\]
commutes for any $V,W\in\mcC$ and $f:V\to W$ in $\mcC$. 
The product on $H^{\mcC}$ is given by the composition of natural transformations.
There is a canonical map $\iota: H\to H^{\mcC}$ given by the assignment $u\mapsto u_V\coloneqq\pi_V(u)$. A subcategory $\mcC\subset\Mod(H)$ {\em separates points} if $\iota$ 
is injective or, equivalently,
if an element in $H$ is uniquely determined by its action on the objects in $\mcC$.
Intuitively, this condition forbids the category $\mcC$ from being \emph{too small}.
The existence of a canonical embedding $H\to H^{\mcC}$ yields a natural interpretation of $H^{\mcC}$ as a {\em completion} of $H$.

\begin{remark}
Every algebraic structure described in Section \ref{s:alm-cyl} admits a categorical counterpart (\eg {\em tensor categories with a cylinder twists} or, more generally, {\em braided module categories}, cf.\ \cite{tD98, Ko20} and references therein).
We will avoid to describe such categorical structures in details. 
Instead, we shall fix a monoidal subcategory of representations $\mcC$ and consider distinguished operators in the corresponding completion $H^{\mcC}$. 
It is worth {noting} that $H^{\mcC}$ in general is not a bialgebra, but rather a \emph{cosimplicial} algebra, see e.g. \cite[Sec.~8.8]{ATL19b}. \hfill \rmkend
\end{remark}


\section{Kac-Moody algebras}\label{s:km}

In this section, we recall several facts about symmetrizable Kac-Moody algebras and their group of automorphisms, following mainly \cite{Ka90, KW92}. 
Moreover, we recall the definition of {\em framed realizations compatible with a diagram automorphism} from \cite[Sec.~2.6]{Ko14}. We show that, in the case of generalized Cartan matrices 
of indefinite type and corank one, such realizations do not always exist. 


\subsection{Realizations and lattices}\label{ss:realizations}
 From now on we will work over\footnote{In fact, $\C$ may be replaced throughout by any algebraically closed field of characteristic 0} $\C$ (and, later on, also over formal extensions of $\C$).
Let $\indI$ be a finite set with a strict total order $<$, $\gcmA=(a_{ij})_{i,j\in\indI}$ a matrix with entries in $\C$ and $(\mfh,\Pi,\Pi^{\vee})$ a {\em realization} of $\gcmA$, \ie $\mfh$ is a $\C$-vector space, $\Pi\coloneqq\{\alpha_i\}_{i\in\indI}\subset\mfh^*$ and $\Pi^{\vee}\coloneqq\{h_{i}\}_{i\in\indI}\subset\mfh$ are linearly independent subsets such that $\alpha_i(h_{j})=a_{ji}$. 
It is well--known that for any $N\geqslant2|\indI|-\rk(\gcmA)=|\indI|+\cork(\gcmA)$ there exists a realization with $\dim(\mfh)=N$.  
This is said to be {\em minimal} precisely when $\dim(\mfh)=|\indI|+\cork(\gcmA)$.
Attached to any realization one has the following {\em (co)root subspaces and lattices}
\begin{equation*}
\Qlat^\vee\coloneqq\vspan{\Z}{\Pi^\vee}\subset\vspan{\C}{\Pi^\vee}\eqqcolon\mfh'
\qquad\mbox{and}\qquad
\Qlat\coloneqq\vspan{\Z}{\Pi}\subset\vspan{\C}{\Pi}\eqqcolon(\mfh^*)'
\end{equation*}
Note that these do not depend on the dimension of $\mfh$. 
Recall that the \emph{height functions} are the group homomorphisms $\Qlat,\Qlat^\vee \to \Z$ given by $x_i \mapsto 1$ for all $i \in I$ with $x_i=\al_i, h_i$, respectively.
Similarly, the  \emph{support functions} $\supp:\Qlat, \Qlat^\vee \to \Pow(\indI)$ are given by 
\eq{
\supp\bigg(\sum_{i \in\indI} m_i x_i\bigg) \coloneqq \{ i \in\indI \, | \, m_i \ne 0 \}
}
with $x_i=\al_i,h_i$, respectively.
The {\em weight lattice} is $\Plat\coloneqq\{\lambda\in\mfh^*\;|\; \lambda(\Qlat^\vee)\subseteq\Z\}\subset\mfh^*$.
Finally, we set $\mfz\coloneqq\{h\in\mfh\;|\; \alpha_i(h)=0 \text{ for all } i \in \indI\}$.
The {\em essential Cartan} is the $|\indI|$-dimensional space $\mfh/\mfz$, which naturally identifies with the dual of the root lattice $\Qlat^*$ through the projection 
\eq{ \label{dualQ:projection}
\mfh\simeq (\mfh^*)^* \to\Qlat^*.
} 
Henceforth, we fix a minimal realization $(\mfh, \Pi,\Pi^\vee)$.


\subsection{Diagram automorphisms}\label{ss:diag-aut}

A \emph{diagram automorphism} of $\gcmA$ is a permutation $\tau:\indI\to\indI$ such that $a_{\tau(i) \, \tau(j)} = a_{ij}$ for all $i,j \in I$. 
Diagram automorphisms form a group denoted $\Aut(\gcmA)$, whose action on $\indI$ naturally extends to the subspaces $\mfh'$ and $(\mfh^*)'$ if we set $\tau(h_i)\coloneqq h_{\tau(i)}$ and $\tau(\al_i)\coloneqq\al_{\tau(i)}$ for $\tau\in\Aut(\gcmA)$ and $i \in\indI$.
By \cite[4.19]{KW92}, any diagram automorphism can be lifted to an element in $\GL(\mfh)$ as follows. 
The identification $\Qlat^*\simeq\mfh/\mfz$ given by the projection \eqref{dualQ:projection} allows us to extend the action of $\Aut(\gcmA)$ on $\Qlat$ to $\mfh/\mfz$ in such a way that $\alpha_i(\tau(h))=\alpha_{\tau(i)}(h)$ for any $i\in\indI$. 
Since the subspace $\mfh'/\mfz\subseteq\mfh/\mfz$ is preserved by any element in the finite group $\Aut(\gcmA)$, then there exists a complement $\mfh''\subseteq\mfh$ such that $\mfh'\oplus\mfh''=\mfh$ and $(\mfh''+\mfz)/\mfz$ is $\Aut(\gcmA)$-stable. 
Then the action of $\tau$ on $\mfh/\mfz$ is lifted by pullback to an action on $\mfh$.


\subsection{Generalized Cartan matrices and Kac-Moody algebras}\label{ss:km-recap}

The matrix $\gcmA$ is a \emph{generalized Cartan matrix} if $a_{ii}=2$ and, for $i\neq j$, $a_{ij}\in\Z_{\leqslant 0}$, and $a_{ij}=0$ implies $a_{ji}=0$.
We say that $\gcmA$ is \emph{of finite type} if all the principal minors of $\gcmA$ are positive; \emph{of affine type} if $\det(\gcmA)=0$ and all proper principal minors of $\gcmA$ are positive; \emph{of indefinite type} if it is neither of finite nor of affine type.

Let $\wt\mfg$ be the Lie algebra generated by $\mfh$ and $\{e_i, f_i\}_{i\in\indI}$ with relations 
\eq{
[h,h']=0
\qq
[h,e_i]=\alpha_i(h)e_i
\qq
[h,f_i]=-\alpha_i(h)f_i
\qq
[e_i,f_j]=\drc{ij}h_{i} 
\qq 
}
for all $h,h' \in \mfh$ and $i,j \in I$. 
The Kac-Moody algebra corresponding to $\gcmA$ is the Lie algebra $\mfg=\wt\mfg/\mfr$, where $\mfr$ is the sum of all two--sided ideals in $\wt\mfg$ having trivial intersection with $\mfh\subset\wt\mfg$.
If $\gcmA$ is a generalized Cartan matrix, the ideal $\mfr$ contains 
$\mathsf{ad}(e_i)^{1-a_{ij}}(e_j)$ and $\mathsf{ad}(f_i)^{1-a_{ij}}(f_j)$ for any $i\neq j$.
The center of $\mfg$ coincides with the subspace $\mfz\subseteq\mfh$.
Set $\Qlat_+\coloneqq\bigoplus_{i\in\indI}\Z_{\geqslant0}{\alpha}_i\subseteq{\mfh}^*$. 
Then, $\mfg$ admits a triangular decomposition $\mfg=\n_-\oplus\mfh\oplus\n_+$,
where 
\eq{
\n_\pm\coloneqq\bigoplus_{\alpha\in\Qlat_+\setminus\{0\}}\mfg_{\pm\alpha} \qq \qq \text{and} \qq \qq \mfg_{{\alpha}}\coloneqq\{x\in\mfg\;|\;[h,x]=\alpha(h)x,\;\forall h\in{\mfh}\}.
}
Then, ${\rootsys}_+\coloneqq\{\alpha\in\Qlat_+\;|\; \mfg_{\alpha}\neq0\}$ is the set of {\em positive roots of} $\mfg$ and $\Phi\coloneqq\Phi_+\sqcup(-\Phi_+)$ is the {\em root system of} $\mfg$. 
We have $\dim(\mfg)<\infty$ (and thus $\Phi$ is finite) if and only if $\gcmA$ is of finite type.


\subsection{The derived subalgebra $\mfg'$}\label{ss:derived}

The derived subalgebra $\mfg'\subseteq\mfg$ is independent of the choice of the realization. 
As a vector space, $\mfg'=\n_-\oplus \mfh'\oplus\n_+$ and admits a presentation similar to that 
of $\mfg$. Namely, let $\wt \mfg'$ the Lie algebra generated by elements $\{e_i,f_i,h_{i}\}_{i \in \indI}$ with relations
\eq{
[h_i,h_j]=0
\qq
[h_i,e_j]=a_{ij}e_j
\qq
[h_i,f_j]=-a_{ij}f_j
\qq
[e_i,f_j]=\drc{ij}h_{i}
}
for all $i,j \in I$.
The Lie algebra $\wt \mfg'$ is graded by $\Qlat$, with $\wt \mfg'_0=\mfh'$. 
The quotient of $\wt \mfg'$ by the sum $\mfr'$ of its graded ideals with trivial intersection with $\wt \mfg'_0$ is canonically isomorphic to $\mfg'$. 
Moreover, $\mfz\subseteq\mfh'$ and the centre of $\mfg'$ is $\mfz$.


\subsection{Symmetrizable Kac-Moody algebras}\label{sss:sym-ext-km}\label{ss:bil on g}

Assume that the matrix $\gcmA$ is symmetrizable and choose a tuple 
$(\eps_i)_{i \in I} \in\Z_{>0}^I$ of coprime positive integers such that $\eps_i a_{ij} = \eps_j a_{ji}$ for all $i,j \in I$.
Note that generalized Cartan matrices of finite or affine type are always symmetrizable.
Let $\mfh''\subset\mfh$ be a complementary subspace to $\mfh'$. 
By \cite{Ka90}, the choice of $\mfh''$ induces a symmetric, non--degenerate bilinear form $\iip{\cdot}{\cdot}$ on $\mfh$ given by
\eq{ \label{bilinearform:g}
\iip{h_{i}}{\cdot}=\epsilon_i^{-1}\alpha_i(\cdot) \aand \iip{\mfh''}{\mfh''}=0\,.
} 
In particular, $\iip{h_{i}}{h_{j}}=a_{ji}\epsilon_i^{-1}=a_{ij}\epsilon_j^{-1}$.
Let $\nu:\mfh\to\mfh^*$ be the linear isomorphism given by $\nu(h)(h')\coloneqq\iip{h}{h'}$ for any $h,h'\in\mfh$. Note that $\nu$ restricts to an isomorphism $\mfh'\simeq(\mfh^*)'$, but it does not preserve the lattices unless $\gcmA$ is symmetric and defined over $\Z$. 
We also denote by $\iip{\cdot}{\cdot}$ the induced bilinear form on $\mfh^*$. The latter uniquely extends to an invariant symmetric bilinear form on $\wt\mfg$ such that $\iip{e_i}{f_j}=\drc{ij}\epsilon_i^{-1}$. 
The kernel of this form is precisely $\mfr$, and therefore $\iip{\cdot}{\cdot}$ descends to a nondegenerate form on $\mfg$. 
Set $\mfb^{\pm}\coloneqq\mfh\oplus\bigoplus_{\alpha\in\rootsys_+}\mfg_{\pm\alpha}\subset\mfg$. 
The bilinear form induces a canonical isomorphism of graded vector spaces $\mfb^{+}\simeq(\mfb^{-})^{\star}$, where $(\mfb^{-})^{\star}\coloneqq\mfh^*\oplus\bigoplus_{\alpha\in\rootsys_+}\mfg_{-\alpha}^*$ denotes the graded dual.
If $\gcmA$ is a symmetrizable generalized Cartan matrix, it is well-known that the ideal $\mfr$ is generated by the Serre relations and $\mfg$ is completely presented by generators and relations.


\subsection{Weyl groups}\label{ss:km-weyl}
The matrix $A=(a_{ij})_{i,j \in I}$ \emph{indecomposable} if for all $X \subseteq I$ there exists $(i,j) \in X \times I \backslash X$ such that $a_{ij} \ne 0$.
Henceforth we assume that $\gcmA$ is an indecomposable symmetrizable generalized Cartan matrix.
The Weyl group associated to the realization of $\gcmA$ is the subgroup $W\subseteq\mathsf{GL}(\mfh)$ generated by the fundamental reflections $s_i:\mfh\to\mfh$, $i\in\indI$, given by
\[
s_i(h)=h-\alpha_i(h)h_{i}
\]
with $h\in\mfh$. 
As an abstract group, 
\[
W\simeq\langle s_i\;|\; s_i^2=1, (s_is_j)^{m_{ij}}=1, i,j\in\indI, i\neq j \rangle
\]
where $m_{ij} = \pi/\Re\big(\cos^{-1}(\tfrac{1}{2}\sqrt{a_{ij}a_{ji}})\big)$, given explicitly
by the following table:
\[
\begin{array}{|c|c|c|c|c|c|}
\hline
a_{ij}a_{ji} & 0 & 1 & 2 & 3 & \geqslant 4\\
\hline
m_{ij} & 2 & 3 & 4 & 6 & \infty\\
\hline
\end{array}
\vspace{0.25cm}
\]
The Weyl group naturally acts on $\mfh^*$ through the {\em dual} fundamental reflections $s_i:\mfh^*\to\mfh^*$, which we denote by the same symbol, given by  
\[
s_i(\lambda)=\lambda-\lambda(h_i)\alpha_i
\]
with $\lambda\in\mfh^*$. 
One verifies that, for any $w\in W$, $h\in\mfh$, $\lambda\in\mfh^*$, $(w\lambda)(h)=\lambda(w^{-1}h)$. 
Moreover, the bilinear forms on $\mfh$ and $\mfh^*$ are $W$-invariant and $\nu:\mfh\to\mfh^*$ is an intertwiner.
The Weyl group action on $\mfh^*$ preserves the weight lattice and the root system $\mfh^*\supset\Qlat\supset\Phi$. 
A root $\alpha\in\Phi$ is {\em real} if $\alpha\in W(\Pi)$ (moreover, in this case, $\dim(\mfg_\al)=1$) and {\em imaginary} otherwise. 


\subsection{ Braid groups and integrable modules}\label{ss:km-braid-int} \label{ss:km-braid}

The braid group associated to the Weyl group $W$ is the Artin--Tits group
$\Br{W}$, generated by the elements $\brS{i}$, $i\in\indI$, with relations
\begin{equation*}
\underbrace{\brS{i}\cdot \brS{j}\cdot \brS{i}\;\cdots\;}_{m_{ij}}=
\underbrace{\brS{j}\cdot \brS{i}\cdot \brS{j}\;\cdots\;}_{m_{ij}}
\end{equation*}

Let $\dgr$ be the diagram associated to the matrix $\gcmA$ is the (unoriented) diagram $\dgr$ with vertices $\indI$, no loops, and an edge between $i$ and $j$ whenever $a_{ij}\neq0$. 
Then, $(\dgr,\ul{m})$ is the Coxeter--Dynkin diagram of $W$ and $\Br{W}=\Br{\dgr,\ul{m}}$ in terms of the notation introduced in Section~\ref{ss:braid-gp-diag}.\\

Recall that an {\em integrable} $\mfg$-module $M$ is an $\mfh$-diagonalizable module, \ie $M=\bigoplus_{\lambda\in\mfh^*}M_{\lambda}$ with 
\eq{
M_{\lambda}\coloneqq\{m\in M\;|\;\forall h\in\mfh,\, h\cdot m=\lambda(h)m\},
} 
such that the action of $e_i$ and $f_i$, $i\in\indI$, on $M$ is locally nilpotent. 
It is useful to observe that the latter condition is equivalent to the local finiteness of the action of the fundamental Lie subalgebras $\mfg_{\{ i \}} \coloneqq \langle e_i,f_i \rangle \cong \mfsl_2$,
\ie $\dim(\mfg_{\{ i \}} \cdot m) < \infty$ for all $m \in M$ and $i\in\indI$, cf.\ \cite[Ex. 3.16-3.19]{Ka90}.
We denote by $\mc{W}_{\int}$ the category of integrable $\mfg$-modules.\\

For any $i\in\indI$ and $x\in\mfg_{\{i\}}$, the operator $\exp(x)\coloneqq\sum_{n\geqslant 0}x^n/n!$ is well-defined on every $M\in\mc{W}_{\int}$ and can be regarded as an element of the algebra of endomorphisms of the forgetful functor $\mc{W}_{\int}\to\Vect_{\C}$. For any $i\in\indI$, set 
\[
\brSg{i}\coloneqq\exp(e_i)\exp(-f_i)\exp(e_i)=\exp(-f_i)\exp(e_i)\exp(-f_i)
\]
It is well--known that the assignment $\brS{i}\mapsto\brSg{i}|_M$ defines a representation of the braid group $\Br{W}$ on $M\in\mc{W}_{\int}$ \cite{Tits}. 
Moreover, $\brSg{i}(M_{\lambda})=M_{s_i(\lambda)}$ for any $\lambda\in\mfh^*$.


\subsection{Kac-Moody group}
Associated to $\mfg'$ there is a Kac-Moody group $G$, see e.g. \cite[1.3]{KW92}.
Roughly, this can be thought of as (a central extension of) a group generated by $\exp(\mfg_{\pm\al_i})$, $i\in\indI$. 
Moreover, $G$ naturally acts on any integrable $\mfg$-module and thus on $\mfg$ itself.
For any real root $\al \in \Phi$, one has a group embedding $\exp: \mfg_\al \to G$ and a group homomorphism $\Ad: G \to \Aut(\mfg)$ such that, for any real root $\al \in \Phi$ and $x \in \mfg_\al$, $\Ad(\exp(x)) = \exp(\ad(x))$. 
In the following, we shall consider the subgroup $\Ad(G)<\Aut(\mfg)$.
Finally, note that the triple exponentials $\brSg{i}$, $i\in I$, are elements of $G$ and determine a morphism of groups $\Br{W}\to G$ such that $\Ad(\brSg{i})(\mfg_\al) = \mfg_{s_i(\al)}$ and $\Ad(\brSg{i})|_\mfh =s_i$ for all $i\in\indI$ and $\alpha\in\Phi$. In particular, we obtain an the action of $\Br{W}$ on integrable $\mfg$-modules


\subsection{Automorphisms of $\mfg$}

We briefly recall the structure of the group $\Aut(\mfg)$ as given by Kac-Wang in \cite{KW92}.
Let $\wt H = \Hom_\grp(\Qlat,\C^\times)$.
There is a group homomorphism $\Ad: \wt H \to \Aut(\mfg)$ given by $\Ad(\chi)(e_i) = \chi(\al_i)e_i$, $\Ad(\chi)(f_i) = \chi(-\al_i)f_i$ and $\Ad(\chi)(h)=h$ for $i \in I$, $h \in \mfh$ and $\chi \in \wt H$.
Following \cite[1.10 and 4.23]{KW92}, we may consider the normal subgroup $\Ad(\wt H \ltimes G) < \Aut(\mfg)$.
We also denote by $\Aut(\mfg;\mfg')$ the subgroup of $\Aut(\mfg)$ of all automorphisms which fix $\mfg'$ pointwise, see \cite[4.20]{KW92}.
The action of $\Aut(A)$ on $\mfh$ can be further extended to a Lie algebra automorphism of $\mfg$ by the assignments $\tau(e_i) = e_{\tau(i)}$ and $\tau(f_i) = f_{\tau(i)}$ for all $i \in I$.
We denote by $\om \in \Aut(\mfg)$ the Chevalley involution defined by 
\eq{
\om(e_i) = -f_i, \qq \om(f_i) = -e_i, \qq \om(h)=-h 
} 
for $i\in I$ and $h\in\mfh$.
We denote $\Out(A) = \Aut(A)$ if $\gcmA$ is of finite type and $\Out(A) = \{ \id, \om \} \times \Aut(A)$ otherwise.
By \cite[4.23]{KW92} we have the decomposition
\eq{ \label{autg:structure}
	\Aut(\mfg) = \Out(A) \ltimes (\Aut(\mfg;\mfg') \times \Ad(\wt H \ltimes G)).
}


\subsection{Automorphisms of the first and second kind}

Let $\theta$ be an automorphism of $\mfg$.
Following \cite[4.6]{KW92}, we say that $\theta$ is \emph{of the first kind} if there exists $g \in G$ such that $\theta(\mfb^+) = \Ad(g)(\mfb^+)$ or, equivalently, $\theta(\mfb^+) \cap \mfb^-$ is finite-dimensional.
We say that $\theta$ is \emph{of the second kind} if there exists $g \in G$ such that $\theta(\mfb^+) = \Ad(g)(\mfb^-)$ or, equivalently, $\theta(\mfb^+) \cap \mfb^+$ is finite-dimensional.
The set of all automorphisms of the first kind $\Aut_I(\mfg)$ is a subgroup of $\Aut(\mfg)$ and the set $\Aut_{II}(\mfg)$ of all automorphisms of the second kind is the corresponding coset $\om \Aut_I(\mfg)$. 
If $\mfg$ is of finite type then $\Aut_I(\mfg) = \Aut_{II}(\mfg) = \Aut(\mfg)$ and otherwise $\Aut(\mfg)$ is the disjoint union of $\Aut_I(\mfg)$ and $\Aut_{II}(\mfg)$.
In \cite[4.38-4.39]{KW92} a combinatorial factorization is given for semisimple automorphisms of $\mfg$ of the second kind; in addition to a diagram automorphism, this requires as input a subdiagram of finite type of $I$.
We will come back to this in Section \ref{sec:pseudoinvolutions} but for now review some basic concepts associated to such subdiagrams.


\subsection{Subdiagrams of finite type} \label{sec:subdiagrams}

If $\gcmA$ is a symmetrizable generalized Cartan matrix and $X \subseteq I$ then the principal submatrix $\gcmA_X \coloneqq (a_{ij})_{i,j \in X}$ is also a symmetrizable generalized Cartan matrix.
Throughout this section we let $X$ be a subdiagram of finite type, i.e. a subset $X \subseteq I$ such that $\gcmA_X$ of finite type.
The subalgebra $\mfg_X = \langle \{ e_i,f_i \}_{i \in X} \rangle$ of $\mfg$ is a finite-dimensional semisimple Lie algebra.
In particular, $\mfh_X\coloneqq \mfh \cap \mfg_X \subseteq \mfh'$ is the $\C$-span of $\Pi^\vee_X = \{ h_i \, | \, i \in X\}$ and $\mfn^\pm_X \coloneqq \mfn^\pm \cap \mfg_X $
are the Lie subalgebras generated by $\{ e_i \, | \, i \in X \}$ and $\{ f_i \, | \, i \in X \}$, respectively.
We set $\Pi_X \coloneqq \{ \al_i \, | \, i \in X \}$, $\mfh^*_X \coloneqq \Sp_\C(\Pi_X)$, and
\eq{
	\Qlat_X = \Qlat \cap \mfh^*_X, \qq \Qlat^+_X = \Qlat^+ \cap \mfh^*_X, \qq \Phi_X = \Phi \cap \mfh^*_X, \qq \Phi^+_X = \Phi^+ \cap \mfh^*_X.
}
Similarly, we have the coroot system $\Phi^\vee_X \subset \mfh_X$ associated to the Cartan matrix $A_X^\t$ and the positive subsystem $\Phi^{\vee,+}_X$.
The root systems $\Phi_X$ and $\Phi^\vee_X$ are finite and the sum of the corresponding fundamental weights and coweights are given by 
\eq{
	\rho_X = \tfrac{1}{2} \sum_{\al \in \Phi^+_X} \al, \qq \qq \rho^\vee_X = \tfrac{1}{2} \sum_{h \in \Phi^{\vee,+}_X} h.
}
In particular, $\rho_X(h_i)=1$ if $i \in X$.\\

We denote by $\Aut_X(A)$ the subgroup of all diagram automorphisms $\tau$ such that $X$ is $\tau$-stable: $\tau(X)=X$.
Note that restriction to $X$ induces a group homomorphism $\Aut_X(A) \to \Aut(A_X)$ for all $X \subseteq I$ which is in general neither injective or surjective.
The Weyl group $W_X$ is the subgroup of $W$ generated by $\{ s_i \}_{i \in X}$.
The group $W_X$ is finite and has a unique longest element $w_X$ which is hence involutive.
There exists a (necessarily unique and involutive) $\oi_X \in \Aut(A_X)$, called the \emph{opposition involution} of $X$, such that $w_X(\al_i) = -\al_{\oi_X(i)}$ for all $i \in X$.
It is well-known that the element $\wt w_X \coloneqq \wt s_{i_1} \cdots \wt s_{i_\ell} \in G$, where 
$s_{i_1} \cdots s_{i_\ell}$ is a reduced expression of the longest element $w_X\in W_X$, 
does not depend on the choice of the reduced expression.
By \cite[Lem.~4.9 and Corollary 4.10.3]{BBBR95}, the corresponding Lie algebra automorphism of $\mfg$ satisfies
\begin{align}
\label{AdmX:X} 		&& \Ad(\wt w_X)|_{\mfg_X} &= \oi_X \circ \om|_{\mfg_X}, \\
\label{AdmX:square}	&& \Ad(\wt w_X^2)(x) &= (-1)^{2 \la(\rho^\vee_X)} x && 
\end{align}
for all $x \in \mfg_\la$ and $\la \in \Phi$.


\subsection{Dynkin diagrams}

Recall that a generalized Cartan matrix $\gcmA$ can be fully represented by 
its {\em Dynkin diagram}, a partially oriented multi--edge diagram defined as follows
\footnote{Note that this is different from the Coxeter diagram mentioned in Section \ref{ss:braid-gp-diag},
	which does not allow multi--edges, but rather \emph{labelled} edges.}. 
For simplicity, we assume that $\gcmA$ is of finite or affine type so that $a_{ij}a_{ji} \in \{ 0,1,2,3,4\}$ for all $i,j \in\indI$. 
Given two nodes $i\neq j$, there is no edge if $a_{ij}a_{ji}=0$; there is a single or double undirected edge if $\eps_i=\eps_j$ and $a_{ij}a_{ji}$ equals 1 or 4, respectively; there is a double, triple, or quadruple edge directed from $i$ to $j$ if $\eps_i>\eps_j$ and $a_{ij}a_{ji}$ equals $2$, $3$ or $4$, respectively.


\subsection{Framed realizations}
Motivated by the theory of quantum Kac-Moody algebras (cf.\ Section \ref{s:QG}), we are interested in larger (co)weights and (co)root lattices, capturing information about the {\em full} Cartan subalgebra $\mfh$. 
To this end, given a minimal realization $(\mfh,\Pi,\Pi^\vee)$ of $\gcmA$, we extend $\Pi^\vee$ to a basis $\Pi^\vee_\ext$ of $\mfh$ by adjoining a tuple $(d_r)_{r=1}^{\cork(\gcmA)}$ such that $\al_i(d_r) \in \Z$ for all $i \in I$, $r \in \{ 1, 2, \ldots, \cork(\gcmA) \}$; we call $\Pi^\vee_\ext$ an \emph{extended basis} and the triple $(\mfh,\Pi,\Pi^\vee_\ext)$ a \emph{framed (minimal) realization} of $A$.
The $d_r$ are called \emph{scaling elements}; whenever $\cork(\gcmA)=1$ we have a single scaling element which we simply denote $d$.
Setting $\mfh'' = \Sp_\C \{ d_r \}_{1 \le r \le \cork(A)}$, from \eqref{bilinearform:g} we obtain that $(d_r,d_s)=0$ for all $1 \le r,s \le \cork(A)$.
In analogy with Section \ref{ss:realizations}, we obtain the \emph{extended coroot lattice} and \emph{extended weight lattice}, respectively given by
\eq{
{\sf Q}_{\sf ext}^\vee = \Sp_\Z(\Pi^\vee_\ext)	\qq \mbox{and}\qq \Plate = \{ \la \in \mfh^* \, | \, \la({\sf Q}_{\sf ext}^\vee) \subseteq \Z \}.
}
Let $\rho \in \Plate$ be defined by $\rho(h_i) = 1$ for all $i \in I$ and $\rho(d_r)=0$ for all $r \in \{ 1,2,\ldots,\cork(\gcmA) \}$. 
Given $X \subseteq I$ of finite type, if $i \in X$ then $(\rho-\rho_X)(h_i) = 0$ and hence $\rho-\rho_X$ is fixed by $s_i$.
Therefore we have
\eq{ \label{wX:rhominusrhoX}
	w_X(\rho-\rho_X)  = \rho-\rho_X.
}


\subsection{Framed extensions of diagram automorphisms}
Let $\tau \in \Aut(A)$. 
The construction of a framed realization and the lift of $\tau$ to $\Aut(\mfh)$ (cf.\ \ref{ss:diag-aut}) both depend upon the choice of a complementary subspace $\mfh''\subset\mfh$ such that $\mfh'\oplus\mfh''=\mfh$.
Therefore, it is not surprising that, in general, $\tau$ does not necessarily preserve $\Pi^\vee_\ext$ or ${\sf Q}_{\sf ext}^\vee$ or, by duality, $\Plate$.

\begin{remark}
If $\tau$ preserves ${\sf Q}_{\sf ext}^\vee$, it can be extended to an algebra automorphism of 
the quantum group $U_q\mfg$ (cf.\ Section \ref{s:QG}). More importantly, in order to construct
solutions of the generalized reflection equation \eqref{eq:psi-Q-re} for $U_q\mfg$ we shall need
the automorphism $\tau$ to extend to $\Plate$ and $U_q\mfg$.
\hfill \rmkend
\end{remark}

Following \cite{Ko14}, we say that a framed realization $(\mfh,\Pi, \Pie^\vee)$, the set of scaling elements, ${\sf Q}_{\sf ext}^\vee$, and $\Plate$ are {\em $\tau$-compatible} if there exists a permutation $\wh \tau$ of $\{ 1,2,\ldots,\cork(\gcmA)\}$ such that 
\eq{ \label{taucompatible:def}
	\al_{\tau(i)}(d_{\wh \tau(r)}) = \al_i(d_r) \qq \text{for all } r \in \{ 1,2,\ldots,\cork(\gcmA) \}.
}
In this case, $\tau$ extends to an automorphism of ${\sf Q}_{\sf ext}^\vee$ by setting $\tau(d_r)=d_{\wh \tau(r)}$ and the corresponding dual map on $\mfh^*$, denoted by the same symbol, preserves $\Plate$.

A framed realization $(\mfh,\Pi, \Pie^\vee)$, the set of scaling elements, ${\sf Q}_{\sf ext}^\vee$, and $\Plate$ are \emph{$\tau$-minimal} if it is $\tau$-compatible and each function $j \mapsto \al_j(d_r)$ is the characteristic function of a $\tau$-orbit (possibly depending on $r$), \ie if
\begin{flalign}
&&	\forall r \in \{ 1,2,\ldots,\cork(\gcmA) \} \; \exists \{ i,\tau(i)\} \subseteq I \text{ such that } \al_j(d_r) = \begin{cases} 1 & \text{if } j \in \{ i,\tau(i) \}, \\ 0 & \text{otherwise}. \end{cases} && 
\end{flalign}

Clearly, any scaling element in a $\tau$-compatible set is a $\Z$-linear combination 
of scaling elements in a $\tau$-minimal set.


\subsection{Existence of $\tau$-compatible realizations in corank one} \label{sec:compatibleextension}
If $\gcmA$ is invertible or $\tau = \id$, any framed realization is clearly $\tau$-compatible.
By \cite[Prop.~2.12]{Ko14}, if $\gcmA$ is a generalized Cartan matrix of affine type, a $\tau$-compatible framed realization always exists. The problem is open for $\gcmA$ of 
indefinite type with $\cork(\gcmA)\neq0$ and $\tau\neq\id$ \cite[Rmk.\ 2.13]{Ko14}.
In the following, we consider the case $\cork(\gcmA)=1$, where the problem reduces to the existence of a single scaling element $d$ such that $\tau(d)=d$.
More precisely, we provide a criterion for arbitrary matrices with integer entries $A$ such that $\cork(A)=1$. 

\begin{proposition} \label{prop:taucompatible}
	Let $\gcmA = (a_{ij})_{i,j \in I}$ be a matrix with integer entries such that $\cork(\gcmA)=1$ and $\tau \in \Aut(A)$. Then $\tau|_{\Ker(A)}=\pm\id$.
	Moreover, a $\tau$-compatible scaling element exists if and only if $\tau|_{\Ker(A)}=\id$.
\end{proposition}

\begin{proof}
We note that $\Ker(A)$ is of the form $\C (a_i)_{i \in I}$ for some rational numbers $a_i$, not all zero. By clearing denominators we may assume that the $a_i$ are coprime integers; this determines them uniquely up to an overall sign. Consider the \emph{basic imaginary root} $\del = \sum_{j \in I} a_j \al_j \in \Qlat$. Note that the natural $\C$-linear left $\Aut(A)$-action on $\C^I$, defined by $(\tau(\bm x))_i = x_{\tau^{-1}(i)}$ for all $\bm x = (x_i)_{i \in I} \in \C^I$, $i \in I$, $\tau \in \Aut(A)$, stabilizes the one-dimensional space $\Ker(A)$. Thus, there exists $\zeta \in \C^\times$ such that
$a_{\tau(j)} = \zeta a_j $ for all $j\in I$.
Since $\tau$ is of finite order, $\zeta$ must be a root of unity.
On the other hand, since both $a_j$ and $a_{\tau(j)}$ are integers for all $j \in I$, it follows that $\zeta \in \Q$. Thus, $\tau|_{\Ker(A)}=\pm\id$.\\

Suppose that $\tau|_{\Ker(A)}=\id$. We show that there exists a $\tau$-compatible scaling element, by a direct generalization of the proof given in \cite[Prop.~2.12]{Ko14}.
Assume that we have a finite-dimensional vector space $\mfh$ and a basis $\Pi^\vee_\ext = \Pi^\vee \cap \{ d \}$ of $\mfh$ where $\Pi^\vee = \{ h_i \}_{i \in I}$ and $\al_{\tau(i)}(d) = \al_i(d)$ for all $i \in I$.
Now fix $k \in I$ so that $a_k \ne 0$ (such $k$ exist since $\Ker(A)$ is one-dimensional) and define $\al_j \in \mfh^*$ for $j \in I$ by:
\[ 
\al_j(h_i) = a_{ij} \text{ for all } i \in I, \qq \al_j(d) = \begin{cases} 1 & \text{if } j \in \{ k, \tau(k) \}, \\ 0 & \text{otherwise}. \end{cases}
\]
We claim that $(\mfh,\Pi,\Pi^\vee_\ext)$ with $\Pi = \{ \al_i \}_{i \in I}$ is a $\tau$-compatible framed minimal realization of $\gcmA$.
Note that if the set $\Pi = \{ \al_i \}_{i \in I}$ is linearly independent then $(\mfh,\Pi,\Pi^\vee)$ is a minimal realization by definition.
Moreover, by setting $\tau(d) = d$ we can check that \eqref{taucompatible:def} is true, thus obtaining the $\tau$-compatibility.
Suppose therefore that $\sum_{j \in I} m_j \al_j = 0$ for some $(m_j)_{j \in I} \in \C^I$; it suffices to show that $m_j=0$ for all $j \in I$.
By applying to $h_i$ for arbitrary $i \in I$ we deduce that $(m_j)_{j \in I} \in \Ker(A)$.
It follows that $\sum_{j \in I} m_j \al_j = m \del$ for some $m \in \C$.
Since $\ze =1$ we have 
\[
0 = \sum_{j \in I} m_j \al_j(d) = m \sum_{j \in \{ k, \tau(k) \}} a_j = m |\{ k, \tau(k) \}| a_k
\]
and we deduce that $m=0$ as required.\\

Suppose that $\tau|_{\Ker(A)}=-\id$. We show that there exists no $\tau$-compatible set $\{ d \}$.
This follows from the claim that the existence of such a set implies $\Ker(\del) = \mfh$, so that $\del$ is the zero element of $\mfh^*$, contradicting the linear independence of $\Pi$.
To this end, note that the definition of $\del$ directly implies that $\del(h_i)=0$ for all $i \in I$. 
It remains to show that $\del(d)=0$.
We denote the $\tau$-orbits in $I$ by $I_1,I_2,\ldots,I_\ell$ for some $\ell \in \Z_{>0}$.
For each $r \in \{1,2,\ldots,\ell \}$ choose a representative $i_r \in I_r$; then $\al_i(d)= \al_{i_r}(d)$ for all $i \in I_r$ as a consequence of $\tau$-compatibility and $I_r = \{ \tau^e(i_r) \, | \, 0 \le e < |I_r| \}$.
Furthermore, for each such $r$ we have $\al_{i_r} = \tau^{|I_r|}(\al_{i_r}) = (-1)^{|I_r|} \al_{i_r}$, so that $|I_r|$ is even and hence $\sum_{e = 0}^{|I_r|-1} (-1)^e = 0$.
Finally, we conclude that
\begin{flalign*}
&& \del(d) = \sum_{j \in I} a_j \al_j(d) = \sum_{r=1}^\ell \sum_{j \in I_r} a_j \al_j(d) = \sum_{r=1}^\ell \sum_{e = 0}^{|I_r|-1} (-1)^e a_{i_r} \al_{i_r}(d) = 0.&& 
\end{flalign*}
The result follows.
\end{proof}

If $\gcmA$ is a generalized Cartan matrix of affine type, the $a_i$ can be chosen to be positive integers. Therefore, we automatically get $\tau|_{\Ker(A)}=\id$ and recover \cite[Prop.~2.12]{Ko14}.

\begin{example}
	There are examples of non--invertible indecomposable symmetrizable generalized Cartan matrices $\gcmA$ of indefinite type with nontrivial diagram automorphisms, both with and without a $\tau$-compatible weight lattice.
	For instance, in the case of the corank one generalized Cartan matrices 
	\[
	A_1 = \begin{pmatrix} 
	2 & -1 & 0 & -3 \\
	-1 & 2 & -3 & 0 \\
	0 & -3 & 2 & -1 \\
	-3 & 0 & -1 & 2
	\end{pmatrix}
	\qq\mbox{and}\qq
	A_2 = \begin{pmatrix} 
	2 & -1 & 0 & -1 \\
	-9 & 2 & -1 & 0 \\
	0 & -1 & 2 & -9 \\
	-1 & 0 & -1 & 2
	\end{pmatrix}.
	\]
	we have that $A_1$ has a $\tau$-compatible weight lattice whereas $A_2$ does not.\hfill \examend
\end{example}


\section{Drinfeld-Jimbo quantum groups}\label{s:QG}

In this section we review the basic theory of Drinfeld-Jimbo quantum groups \cite{Dr85,Dr86,Dr90a,Ji86,Lus94}.
In particular, we discuss the factorization properties of the universal R-matrix and we define
(highest and lowest weight) category $\mc O$ representations.


\subsection{Quantum Kac-Moody algebras}
Let $q$ be an indeterminate and denote by $\F$ the algebraic closure of $\C(q)$. We shall use
the fact that the multiplicative group $\F^\times$ is a divisible abelian group\footnote{
	While it is possible to use only  a finite extension of $\C(q)$ (cf. \cite[Rmk.\ 2.3]{BK19}),
	the latter depends on the generalized Cartan matrix $A$. Therefore, we prefer to  
	 work with the algebraic closure of $\C(q)$.}.
Let $\gcmA$ be a generalized Cartan matrix and $(\mfh, \Pi,\Pie^\vee)$ a framed realization.
Following \cite{Dr85,Dr86,Ji86,Lus94} we denote by $U_q\mfg$ the unital associative $\F$-algebra with generators $E_i, F_i$ ($i \in \indI$) and $t_h$ ($h \in \Qlat^\vee_\ext$) subject to the following relations for $h,h' \in \Qlat^\vee_\ext$, $i,j \in \indI$:
\begin{gather}
\nonumber t_0 = 1, \qq t_h t_{h'} = t_{h+h'}, \\
\nonumber t_h E_i = q^{\al_i(h)} E_i t_h, \qq t_h F_i = q^{-\al_i(h)} F_i t_h, \qq [E_i,F_j] = \del_{ij} \frac{t_i-t_i^{-1}}{q_i-q_i^{-1}}, \\
\nonumber \Ser_{ij}(E_i,E_j) = 0=\Ser_{ij}(F_i,F_j)  \qq\qq (i \ne j)
\end{gather}
where $t_i = t_{\eps_i h_i}$, $q_i\coloneqq q^{\eps_i}$ and $\Ser$ denotes the q-deformed
Serre relations (\eg \cite[3.1.1 (e)]{Lus94}).
We endow $U_q\mfg$ with the bialgebra structure
determined by the coproduct
\eq{ \label{Uqg:bialgebra}
\begin{aligned}
\Del(E_i) &= E_i \ot 1 + t_i \ot E_i, \qu & \Del(F_i) &= F_i \ot t_i^{-1} + 1 \ot F_i, \qu & \Del(t_h) &= t_h \ot t_h\, .
\end{aligned}
}


\subsection{Triangular decomposition and diagrammatic subalgebras}
We consider the standard subalgebras
\eqn{
U_q\mfn^+ = \langle E_i \,\vert\,{i \in \indI} \rangle, \qq U_q\mfn^- =\langle F_i \,\vert\,{i \in \indI} \rangle,, \qq U_q\mfh = \langle t_h \,\vert\,{h \in \Qlat^\vee_\ext} \rangle
}
so that $U_q\mfg=U_q\mfn^+U_q\mfh U_q\mfn^-$. We set $U_q\mfb^\pm = U_q\mfn^\pm U_q\mfh$
and consider the following quantum analogue of the derived subalgebra $\mfg'\subseteq\mfg$
\eq{
U_q\mfg' = \langle E_i,F_i,t_i^{\pm 1} \,\vert, {i \in \indI} \rangle
\qq\mbox{and}\qq U_q\mfh' = \langle t_i^{\pm 1} \,\vert\, {i \in \indI} \rangle.
}

For any subset $X \subseteq \indI$, the derived quantum Kac-Moody algebra corresponding to $\gcmA_X$ embeds in $U_q\mfg'$, yielding the diagrammatic subalgebras
\eqn{
U_q\mfg_X = \langle E_i,F_i,t_i^{\pm 1}\,\vert\, {i \in X} \rangle, \qq U_q\mfh_X = U_q\mfg_X \cap U_q\mfh, \qq\mbox{and}\qq U_q\mfn^\pm_X = U_q\mfg_X \cap U_q\mfn^\pm.
}
Note that, for any $i\in I$, $U_q\mfg_{\{ i \}}\simeq U_{q_i}\mfsl_2$.
The assignment $\al_i \mapsto t_i$ yields an algebra isomorphism $\F \Qlat \to U_q\mfh'$ and, 
for $\la \in \Qlat$, we set $t_\la = \prod_{i \in \indI} t_i^{\ell_i}$ if $\la = \sum_{i \in \indI} \ell_i \al_i$.
These elements satisfy
\eq{ \label{Uqg:rels:Cartan2}
t_\la E_i = q^{(\la,\al_i)} E_i t_\la, \qq t_\la F_i = q^{-(\la,\al_i)} F_i t_\la \qq \text{for all } \la \in \Qlat, i \in \indI.
}
In terms of the linear isomorphism $\nu$ from Section \ref{sss:sym-ext-km} we have $t_\la = t_{\nu^{-1}(\la)}$ for all $\la \in \Qlat$.


\subsection{Root space decomposition}

Let $M$ be a $U_q\mfh$-module. 
For any $\mu \in \Plate$, we set
\eqn{
M_\mu = \{ m \in M \, | \, \forall h \in \Qlat^\vee_\ext : t_h \cdot m = q^{\mu(h)} m \}.
}
Note that, for all $\la \in \Qlat$ and $\mu \in \Plate$, $t_\la$ acts on $M_\mu$ as multiplication by $q^{(\la,\mu)}$.
By considering the action of $U_q\mfh$ on $U_q\mfg$ and $U_q\mfg'$ by conjugation, 
we obtain the root space decomposition
\eq{ \label{Uqg:rootspacedecomposition}
	U_q\mfg = \bigoplus_{\la \in \Qlat} (U_q\mfg)_\la, \qq U_q\mfg' = \bigoplus_{\la \in \Qlat} (U_q\mfg')_\la
}
as $\Qlat$-graded algebras. Note that $U_q\mfn^\pm$ are graded by $\pm \Qlat^+$, \ie
$U_q\mfn^\pm = \bigoplus_{\la \in \Qlat^+} (U_q\mfn^\pm)_{\pm \la}$.


\subsection{Automorphisms} 
We briefly review several distinguished algebra automorphisms of $U_q\mfg$.
An algebra automorphism is assumed to be $\F$-linear unless otherwise stated.
Any diagram automorphism $\tau \in \Aut(A)$ (cf. ~Section~\ref{s:km}) acts as a bialgebra automorphism on $U_q\mfg'$ by
\eqn{
\tau(E_i) = E_{\tau(i)}\,, \qq \tau(F_i) = F_{\tau(i)}\,, \qq \tau(t_i) = t_{\tau(i)}\,, 
}
for any $i\in I$.
If the framed realization is $\tau$-compatible, then this action extends automatically to a bialgebra 
automorphism of $U_q\mfg$ by setting $\tau(t_h) = t_{\tau(h)}$ for all $h \in \Qlat^\vee_\ext$.
The Chevalley involution $\om$ lifts to an involutive algebra automorphism of $U_q\mfg$, denoted by the same symbol, determined by
\eqn{
\om(E_i) = -F_i\,, \qu \om(F_i) = -E_i\,, \qu \om(t_h) = t_{-h}\,,
}
for any $i\in\indI$ and $h\in\Qlat^\vee_\ext$.
Note that $\om$ is a bialgebra isomorphism from $U_q\mfg$ to $U_q\mfg^\cop$,\ie
\eq{ \label{om:bialgebra}
\Del \circ \om = (\om \ot \om) \circ \Del^\op, \qq \eps \circ \om = \eps.
}
Finally, we discuss the \emph{bar involution}, which is not $\mathbb{F}$-linear (cf.~ \cite{Lus94}).
Note that the algebraic closure of the field of formal Laurent series $\C((q))$ is given by $\wb{\C((q))} = \bigcup_{n \ge 1} \C((q^{1/n}))$, on which we define a field automorphism $\wb{\phantom{u}}$ by the rule $\wb{q^{1/n}} = q^{-1/n}$.
The field $\mathbb{F}$ arises as the set of algebraic elements in $\wb{\C((q))} $ over $\C(q)$ and note that $\wb{\phantom{u}}$ stabilizes $\C(q) \subset \wb{\C((q))}$. 
By considering minimal polynomials of elements of $\mathbb{F}$ in $\C(q)[x]$ we obtain that $\wb{\phantom{u}}$ stabilizes $\mathbb{F}$.
We extend it to an algebra automorphism of $U_q\mfg$ by setting
\eqn{
\wb{E_i} = E_i\,, \qu \wb{F_i} = F_i\,, \qu \wb{t_h} = t_{-h}\,,
}
for any $i\in\indI$ and $h\in\Qlat^\vee_\ext$.
We use the same notation to denote the corresponding algebra automorphism of $U_q\mfg \ot U_q\mfg$, defined by $\wb{u \ot u'} \coloneqq \wb{u} \ot \wb{u'}$ for any $u,u' \in U_q\mfg$.


\subsection{The Drinfeld-Lusztig pairing}

Given a bilinear pairing $\langle \phantom{x}, \phantom{x} \rangle: A^- \times A^+ \to \F$ between algebras $A^-$ and $A^+$ over a field $\F$, it extends to an $\F$-valued bilinear pairing between $(A^-)^{\ot n}$ and $(A^+)^{\ot n}$ for all $n \in \Z_{\geqslant 1}$ by
\[
\langle a^-_1 \ot \cdots \ot a^-_n, a^+_1 \ot \cdots \ot a^+_n \rangle = \prod_{m=1}^n \langle a^-_m , a^+_m \rangle
\]
for all $a^-_1,\ldots,a^-_n \in A^-$ and $a^+_1,\ldots,a^+_n \in A^+$.
We recall that, by \cite{Dr90a, Lus94}, there exists a unique $\F$-bilinear pairing $\langle \phantom{x}, \phantom{x} \rangle: U_q\mfb^- \times U_q\mfb^+ \to \F$ such that
\eq{ \label{pairing:rel1}
\langle y, xx' \rangle = \langle \Del(y),x' \ot x \rangle, \qq \langle yy',x \rangle = \langle y \ot y' , \Del(x) \rangle
}
for any $x,x' \in U_q\mfb^+$, $y,y' \in U_q\mfb^-$, and
\eq{ \label{pairing:rel2}
\begin{aligned}
\langle t_h,t_{h'} \rangle &= q^{-(h,h')}, & \langle F_i,E_j \rangle &= \del_{ij} \frac{1}{q_i^{-1}-q_i}, \\
\langle t_h,E_j \rangle &= 0, & \langle F_i,t_{h'} \rangle &= 0.
\end{aligned}
}
for any $i,j \in \indI$, $h,h' \in \Qlat^\vee_\ext$.
In particular, for any $e \in U_q\mfn^+$, $f \in U_q\mfn^-$, $h,h' \in \Qlat^\vee_{\ext}$, one obtains
\eq{\label{pairing:rel3}
\langle f t_h, e t_{h'} \rangle = q^{-(h,{h'})} \langle f,e \rangle
}
so that $\langle f, x t_h \rangle = \langle f, x \rangle$, 
for all $x \in U_q\mfb^+$, $f \in U_q\mfn^-$ and $h \in \Qlat^\vee_{\ext}$.


\subsection{The categories $\mcO^\eps$}\label{ss:quantum-cat-O}

It is well--known that the Drinfeld-Lusztig pairing allows for the realization of $U_q\mfg$ as (a quotient of) a quantum double \cite{Dr86}. 
Thus, the canonical element of $\langle\cdot,\cdot\rangle$, which belongs to a suitable completion of the tensor product $U_q\mfb^-\ot U_q\mfb^+$, is a {\em topological} R-matrix inducing a quasitriangular structure on $U_q\mfg$. 
This is more conveniently described in terms of categories of representations as we explained in Section \ref{ss:tannakian}.\\

Let $\mc W\subset\Mod(U_q\mfg)$ be the full subcategory of
$U_q\mfg$-modules $M$ endowed with a weight space decomposition, \ie
\eqn{
	M = \bigoplus_{\mu \in\Plate} M_\mu\,.
} 
Then, for any $\la \in \Qlat$, the action of $(U_q\mfg)_\la$ maps $M_\mu$ into $M_{\mu+\la}$.
For $\eps\in\{\pm\}$, let $\mc O^\eps$ denote the full subcategory consisting of objects in $\mc W$ with a locally finite $U_q\mfn^\eps$-action and finite-dimensional weight spaces
(see \eg \cite[Ch.~9]{Ka90} or \cite[Ch.~3]{Lus94}). Note that $\mc O^\eps$ is closed under 
tensor products.\\ 

We denote by $(U_q\mfg)^{\mc O^\eps}$ the completion of $U_q\mfg$ with respect to the category $\mc O^\eps$ (cf.~Section~\ref{ss:tannakian}).
Recall that, by construction, $(U_q\mfg)^{\mc O^\eps}$ is the algebra of operators defined on category $\mc O^\eps$ modules, which are natural with respect to $U_q\mfg$-intertwiners.
Similarly, we denote by $(U_q\mfg^{\ot n})^{\mc O^\eps}$ the algebra of operators defined on tensor products of $n$ modules in $\mc O^{\eps}$. 
For any $n \in \Z_{>0}$, $U_q\mfg^{\ot n}$ embeds in $(U_q\mfg^{\ot n})^{\mc O^\eps}$ and therefore $\mc O^\eps$ separates points \cite[Question~8.2]{Dr90b}.\\

Note that $(U_q\mfg)^{\mc O^\eps}$ contains as a subalgebra the completion $(U_q\mfn^\eps)^{\mc O^\eps}\coloneqq\prod_{\mu \in \eps\Qlat^+} (U_q\mfn^\eps)_\mu$ of $U_q\mfn^\eps$ with respect to its natural $\eps\Qlat^+$-grading. 
Indeed, every element in $(U_q\mfn^\eps)^{\mc O^\eps}$ is convergent on category $\mc O^{\eps}$ and commutes with every intertwiner. 
Similarly, we have that
\[
(U_q\mfn^\mp\ot U_q\mfn^\pm)^{\mc O^\eps}\coloneqq
\prod_{\mu\in\pm\Qlat^+}(U_q\mfn^\mp)_{-\mu} \ot (U_q\mfn^\pm)_\mu
\]
are subalgebras in $(U_q\mfg^{\ot 2})^{\mc O^\eps}$. 
Finally, note that $(U_q\mfn^\pm)_0 = \F$.


\subsection{Quasi-R-matrices}\label{sec:R-matrices}
We recall below the construction of the so-called quasi-R-matrix due to Lusztig \cite{Lus94}.
For any $\mu \in \Qlat^+$, let $( b^-_{\mu,r} )_r$ be an ordered basis for $(U_q\mfn^-)_{-\mu}$ and let $(b^+_{\mu,r})_r$ be the corresponding dual basis for $(U_q\mfn^+)_\mu$ with respect to the pairing $\langle \phantom{x},\phantom{x} \rangle$ defined by \eqref{pairing:rel1}-\eqref{pairing:rel2}.
Lusztig's \emph{quasi-R-matrix} is the element
\eq{ \label{quasiR:def}
	\begin{aligned}
		\Theta = \sum_{\mu \in \Qlat^+} \Theta_\mu \in (U_q\mfn^-\ot U_q\mfn^+)^{\mc O^\eps} 
		\quad\text{where} \quad \Theta_\mu = \sum_r b^-_{\mu,r} \ot b^+_{\mu,r} \in (U_q\mfn^-)_{-\mu} \ot (U_q\mfn^+)_\mu.
	\end{aligned}
}
The element $ \Theta_\mu$ is independent of the choice of basis.
The quasi-R-matrix is deeply related with the bar involution \cite[Thm.~4.1.2]{Lus94}.
Namely, $\Theta$ is the unique element of the form $\Theta = \sum_{\mu \in \Qlat^+} \Theta_\mu$ with $ \Theta_\mu \in (U_q\mfn^-)_{-\mu} \ot (U_q\mfn^+)_\mu$ such that 
$\Theta_0 = 1 \ot 1$ and 
\eq{ \label{quasiR:intw}
	\Theta \wb{\Del(u)} = \Del(\wb{u}) \Theta
}
for any $u\in U_q\mfg$.
Moreover, by a standard argument, the normalization $\Theta_0=1$ guarantees that 
$\Theta$ is invertible and moreover, by uniqueness, one has
\eq{ \label{quasiR:inverse}
	\Theta^{-1} = \wb{\Theta}.
}

It is clear that $\Theta$ is not the \emph{full} R-matrix of $U_q\mfg$. 
As $\Theta$ is the canonical element of the pairing between $U_q\mfn^+$ and $U_q\mfn^-$,
the missing factor is a weight zero operator, computed in terms of Cartan elements. 


\subsection{A completion of the quantum Cartan subalgebra} \label{sec:Uqhhat}

We shall describe several weight zero operators. For convenience, we think of such elements as belonging to the completion $(U_q\mfg)^{\mc W}$. In particular, they act on
category $\mc O^\eps$ modules.
Let $\Fun(A,B)$ denote the functions from a set $A$ to a set $B$.
Any $\be \in \Fun(\Plate, \F)$ induces an element of $(U_q\mfg)^{\mc W}$, also denoted $\be$, whose action on $M \in \mc W$ is given by
\[
\be \cdot m = \be(\mu) m
\]
for any $\mu\in\Plate$ and $m\in M_{\mu}$.
The subspace of $(U_q\mfg)^{\mc W}$ spanned by such $\be$ is a commutative subalgebra, which we denote by $(U_q\mfh)^{\mc W}$. 
Indeed, note that the Cartan subalgebra $U_q\mfh$ naturally embeds in $(U_q\mfh)^{\mc W}$. 
The group $(U_q\mfh)^{\mc W,\times}$ of invertible elements is given by $\Fun(\Plate, \F^\times)$. 
Similarly, any $\ka \in \Fun(\Plate \times \Plate ,\F)$ defines an element of $(U_q\mfg \ot U_q\mfg)^{\mc W}$, also denoted $\ka$, whose action on $M \ot N$, for $M,N \in \mc W$ is given by
\[
\ka \cdot m \ot n = \ka(\mu,\nu) m \ot n 
\]
for any $\mu,\nu\in\Plate$, $m\in M_{\mu}$, and $n\in N_{\nu}$.
In particular, for any $g \in \End_\Z(\Plate)$, we denote by $\ka_g$ the function from 
$\Plate \times \Plate \to \F^\times$ given by
\eq{ \label{ga:def}
	\ka_g(\mu,\nu)\coloneqq q^{(g(\mu),\nu)}.
}

\begin{remark}\label{rmk:theta-extend-to-completion}
Let $\psi_q$ be an algebra endomorphism of $U_q\mfg$, whose restriction functor, which we
refer to as the {\em pullback} of $\psi_q$, gives an endofunctor $\psi_q^*:\mc W\to \mc W$. 
Then, $\psi_q$ extends to an endomorphism of $(U_q\mfg)^{\mc W}$, given by 
$\psi_q(\varphi)|_M\coloneqq\varphi|_{\psi_q^*(M)}$ for any $\varphi\in(U_q\mfg)^{\mc W}$
and $M\in\mc W$. Suppose that $\psi_q$ is invertible and $\psi_q^*$ acts by permuting the 
weight spaces, \ie there exists $\psi \in \End_\Z(\Plate)$ such that, for any $M\in\mc W$ and $\mu\in\Plate$, $\psi_q^*(M)_{\mu}=M_{\psi(\mu)}$. Then,
for any $g \in \End_\Z(\Plate)$, we have $(\psi_q\ten\id)(\ka_g)=\ka_{g\circ\psi}$.
\hfill\rmkend
\end{remark}


\subsection{A distinguished subgroup of $(U_q\mfh)^{\mc W,\times}$}\label{ss:cartan-operators}
For any $\ze \in \End_\Z(\Plate)$ and $\la \in \Plate$, we define $G_{\ze,\la} \in \Fun(\Plate,\F^\times)$ by
\eq{ \label{G:def}
	G_{\ze,\la}(\mu) \coloneqq q^{(\ze(\mu),\mu)/2 + (\la,\mu)}
}
for any $\mu\in\Plate$. For example, for $\la \in \Qlat$, we have $G_{0,\la} = t_\la$ .
Note that the set of functions  $G_{\ze,\la}$ with $\ze \in \End_\Z(\Plate)$ and $\la \in \Plate$ 
form a subgroup of $(U_q\mfh)^{\mc W,\times}$, since 
$G_{\ze_1,\la_1}G_{\ze_2,\la_2} = G_{\ze_1+\ze_2,\la_1+\la_2}$.
We restrict to the case of $\ze \in \End_\Z(\Plate)$ which are self-adjoint with respect to  $(\phantom{x},\phantom{x})$.

\begin{lemma} \label{lem:G}
	Let $\ze \in \End_\Z(\Plate)$ be self-adjoint and $\la \in \Plate$.
	\begin{enumerate}\itemsep0.25cm
		\item The functional equation
		\eq{ \label{G:funceqn}
			G_{\ze,\la}(\mu+\nu) = G_{\ze,\la}(\mu) G_{\ze,\la}(\nu) q^{(\ze(\mu),\nu)}
		}
		holds for any $\mu,\nu\in\Plate$.
		\item In $(U_q\mfg \ot U_q\mfg)^{\mc W}$, it holds
		\eq{ \label{G:Delta}
			\Del(G_{\ze,\la}) = G_{\ze,\la} \ot G_{\ze,\la} \; \ka_\ze\,,
		}
		where $\Del$ denotes the coproduct $(U_q\mfg)^{\mc W}\to (U_q\mfg \ot U_q\mfg)^{\mc W}$.
		\item If $\ze(\Qlat) \subseteq \Qlat$, the automorphism $\Ad(G_{\ze,\la})$ 
		preserves $U_q\mfg$ in $(U_q\mfg)^{\mc W}$.
		More precisely, it holds
		\eq{ \label{G:Ad}
			\Ad(G_{\ze,\la})(u) = G_{\ze,\la}(\mu) u t_{\ze(\mu)} = G_{-\ze,\la}(\mu) t_{\ze(\mu)} u\,,
		}
		for any $\mu \in \Qlat$ and $u \in (U_q\mfg)_\mu$.
	\end{enumerate}
\end{lemma}

\begin{proof} \mbox{}
	\begin{enumerate}[leftmargin=2em]\itemsep0.25cm
		\item This follows immediately from the self-adjointness of $\ze$.
		\item Let $M,N \in \mc W$ and let $m \in M$, $n \in N$ such that $m \in M_\mu$ and $n \in N_\nu$ for some $\mu,\nu \in \Plate$.
		Then $m \ot n \in (M \ot N)_{\mu + \nu}$, so that 
		\[
		\Del(G_{\ze,\la})(m \ot n) = G_{\ze,\la}(\mu+\nu) m \ot n = G_{\ze,\la}(\mu) G_{\ze,\la}(\nu) \ka_\ze(\mu,\nu) m \ot n = \big( G_{\ze,\la} \ot G_{\ze,\la} \; \ka_\ze \big) (m \ot n),
		\]
		where the second equality follows from \eqref{G:funceqn}.
		\item Let $N \in \mc W$, $\nu \in \Plate$, $\mu \in \Qlat$ and $u \in (U_q\mfg)_\mu$.
		Then we have
		\[
		\Ad(G_{\ze,\la})(u)|_{N_\nu} = \frac{G_{\ze,\la}(\mu+\nu)}{G_{\ze,\la}(\nu)} u|_{N_\nu} 
		= G_{\ze,\la}(\mu) q^{(\ze(\mu),\nu)} u|_{N_\nu} = G_{\ze,\la}(\mu) u t_{\ze(\mu)}|_{N_\nu}
		\]
		and note that $t_{\ze(\mu)} u = q^{(\ze(\mu),\mu)} u t_{\ze(\mu)}$.
		It follows that $\Ad(G_{\ze,\la})(u)$ preserves $U_q\mfg$. \qedhere
	\end{enumerate}
\end{proof}

We shall consider also the subgroup of group homomorphisms $\Hom_\grp(\Plate,\F^\times)$, 
\ie $\be:\Plate\to\F^\times$ such that $\be(\mu+\nu) = \be(\mu)\be(\nu)$ for 
$\mu,\nu \in \Plate$. For any $\be \in \Hom_\grp(\Plate,\F^\times)$, we have 
$\Del(\be) = \be \ot \be$ and $\Ad(\be)(u) = \be(\mu)u$ for any $\mu \in \Qlat$ and
$u \in (U_q\mfg)_\mu$.


\subsection{The universal R-matrix} \label{sec:standardR}

We review the construction of the {\em full} universal R-matrix of $U_q\mfg$. 
For our choice of conventions, it is preferable to work with the operator
$\wtR \coloneqq \Theta^{-1} = \wb{\Theta}$,
which satisfies
\eq{ \label{tildeR:intw}
	\wtR\cdot \Del(u) = \wb{\Del(\wb{u})}\cdot\wtR
}
for any $u \in U_q\mfg$.
One has $\Ad(\ka_\id) \circ \wb{\Del(\wb{ \phantom{x} })} = \Del^\op$.
Therefore, by \eqref{quasiR:intw}, the universal R-matrix given by $R \coloneqq \ka_\id\cdot \wtR \in (U_q\mfg^{\ot 2})^{\mc O^\eps}$
(see e.g.\ \cite{Dr85,Ji86}) satisfies
\eq{ \label{R:intw}
	R \Del(u) = \Del^\op(u) R
}
for any $u\in U_q\mfg$. Moreover, the following coproduct identities hold:
\eq{ \label{R:Delta}
	(\Del \ot \id)(R) = R_{13} R_{23}\qq\mbox{and}\qq (\id \ot \Del)(R) = R_{13} R_{12}.
}
Finally, one has
\eq{ \label{R:symmetries}
	(\om \ot \om)(R) = {R}_{21} \qq\mbox{and}\qq (\tau \ot \tau)(R) = R
}
for any $\tau \in \Aut(A)$.


\subsection{Diagrammatic factorizations}
The factorization of $R$ has the following interpretation.
As a bialgebra, $U_q\mfb^-$ projects onto $U_q\mfh$. 
By a theorem of Radford \cite{Rad92}, the Borel bialgebra $U_q\mfb^-$ can be realized as the so-called Radford biproduct of $U_q\mfh$ and $U_q\mfn^-$, where the latter should be regarded as a \emph{Yetter-Drinfeld} $U_q\mfh$-module (see \eg \cite[Sec.~2.17]{ATL18}).
As the bialgebra structure on $U_q\mfb^-$ is recovered by those on $U_q\mfh$ and $U_q\mfn^-$, in the same way the R-matrix of (the quantum double of) $U_q\mfb^-$ is realized as the product of those of $U_q\mfh$ and $U_q\mfn^-$, represented respectively by $\ka_\id$ and $\wtR$. 
A similar phenomenon can be described more generally, when $U_q\mfh$ is replaced by a diagrammatic subalgebra $U_q\mfb^-_X$, $X\subseteq\indI$.\\

Let $X\subseteq\indI$ be arbitrary (in particular, not necessarily of finite type). 
There is a unique {\em diagrammatic} quasi-R-matrix 
\eq{
	\Theta_X \in \prod_{\la \in \Qlat^+_X} (U_q\mfn^-)_{-\la} \ot (U_q\mfn^+)_\la
}
such that $\Theta_{X,0}=1 \ot 1$ and 
\eq{ \label{quasiRX:intw}
	\Theta_X \wb{\Del(u)} = \Del(\wb{u}) \Theta_X
}
for any $u\in U_q\mfg_X$. Moreover, for any $\ka\in\Fun(\Plate\times\Plate,\F)$ such that
$\ka(\mu+\nu,\mu') = \ka(\mu,\mu'+\nu)$ for $\nu \in \Qlat_X$ and $\mu,\mu' \in \Plate$,
we have
\eq{ \label{quasiRX:Cartan0}
[\ka , \Theta_X ]=0
}
In particular, for any $\mu\in\Qlat$, we have
\begin{align} 
\label{quasiRX:Cartan} [t_\mu \ot t_\mu, \Theta_X] &=0 
\end{align}
and, for any $\be \in \Fun_X(\Plate,\F)$ where 
\eq{
	\Fun_X(\Plate,\F) := \{\be \in \Fun(\Plate,\F) \;|\;\be(\mu+\nu) = \be(\nu) \text{ for all } \mu \in \Qlat_X, \nu \in \Plate\}\,,
} 
we have
\begin{align}
\label{quasiRX:zeta1} [\be \ot 1,\Theta_X] &= 0\,.
\end{align}  
We shall need the following result.

\begin{proposition} \label{prop:quasiR:quotient}
We have
	\[
	\Theta \Theta_X^{-1} \in \prod_{\la \in \Qlat^+ \backslash\Qlat^+_X} (U_q\mfn^-)_{-\la} \ot (U_q\mfn^+)_\la.
	\]
\end{proposition}

\begin{proof}
Set $\wtR_X \coloneqq \Theta_X^{-1} = \wb{\Theta_X}$, so that $\wtR_X \Del(u) = \wb{\Del(\wb{u})}\cdot \wtR_X$ for all $u \in U_q\mfg_X$.
Note that $U_q\mfb^-$ projects, as a bialgebra, onto the quantum Levi subalgebra $U_q\mfn^-_XU_q\mfh$. 
Similarly, the latter projects onto $U_q\mfh$. 
Thus, by a double application of Radford's theorem and \cite[Prop.~4.8]{ATL18}, we obtain a refined factorization 
\eq{ \label{R:factorization}
R=\ka_{\id}\cdot \wtR_X\cdot\ol{R}
}
where $\ol{R}\in\prod_{\la \in \Qlat^+ \backslash\Qlat^+_X} (U_q\mfn^-)_{-\la} \ot (U_q\mfn^+)_\la$. 
The result follows.
\end{proof}

\begin{remark}\label{rmk:kr-factorization}
When $I$ is of finite type, the result is an obvious consequence of the so-called Kirillov--Reshetikhin factorization of the quasi-R-matrix \cite{KR90}.
\hfill\rmkend
\end{remark}


\section{Quantum Weyl groups and integrable representations} \label{s:qW-int}

We review the basic properties of the category of integrable $U_q\mfg$-modules, including
the action of the quantum Weyl group and the associated diagrammatic half-balances, which will
be used in Section \ref{s:coideal}.


\subsection{Integrable $U_q\mfg$-modules}\label{ss:int-mod-Uqg}

Recall that a $U_q \mfg$-module is \emph{integrable} if it has a locally finite $U_q\mfg_{\{ i \}}$-action for all $i \in \indI$. 
We denote\footnote{We used the same symbol for the analogue category of $\mfg$-modules. From now on we only consider $U_q\mfg$-modules.} by $\mc W_\int$ the full subcategory of integrable objects in $\mc W$, cf.\ Section~\ref{ss:quantum-cat-O}.
For $\eps\in\{\pm\}$, let $\mc O^\eps_\int\coloneqq\mc O^\eps\cap\mc W_\int$ denote the category of integrable category $\mc O^\eps$ modules. 
This is a semisimple category whose simple objects are given by highest-weight (if $\eps = +$) or lowest-weight (if $\eps = -$) modules.
Moreover, $\mc O^\eps_\int$ is a braided tensor subcategory of $\mc O^\eps$ with braiding induced by the action of the R-matrix.

\begin{remark} \label{rmk:omega:equivalence}
The pullback of the Chevalley involution defines a braided tensor equivalence $\mc O_\int^+\to\mc O_\int^{-, \op}$, where the latter category is endowed with the opposite tensor product and braiding, cf.\ Section~\ref{R:symmetries}. 
Note that, if $\mfg$ is of finite type, $\mc O^+_\int = \mc O^-_\int$ and $\omega^*$ is an autoequivalence at the level of abelian categories. 
However, if $\dim(\mfg)=\infty$, then every non-trivial module in $\mc O^\eps_\int$ is infinite-dimensional and $\mc O^+_\int\cap\mc O^-_\int$ consists only of trivial representations. \hfill \rmkend
\end{remark} 

Let $(U_q\mfg)^{\mc O^\eps_\int}$ be the completion of $U_q\mfg$ with respect to category $\mc O^\eps_\int$. 
By restriction, one gets a canonical morphism $(U_q\mfg)^{\mc O^\eps}\to(U_q\mfg)^{\mc O^\eps_\int}$. By \cite[Prop.~3.5.4]{Lus94}, the category ${\mc W}_\int$ separates points.
By \cite{ATL22, E22}, the same result holds for category $\mc O^\eps_\int$, 
thus $U_q\mfg$ embeds into 
$(U_q\mfg)^{\mc O^\eps_\int}$ through $(U_q\mfg)^{\mc O^\eps}$. 
In $(U_q\mfg)^{\mc O^+_\int}$, we shall consider elements of the form
\[
u = \sum_{\mu \in \eps \Qlat^+} c_\mu u_\mu
\]
with $u_\mu \in (U_q\mfn^\eps)_\mu$ and $c_\mu \in U_q\mfb^{-\eps}$ (cf.~\cite[Question~8.2]{Dr90b}).


\subsection{Quantum Weyl group operators}  \label{sec:internalbraidaction}

Let $M \in \mc W_\int$.
For $j \in \indI$, we denote by $\wt T_j$ Lusztig's operator $T''_{j,1}$ from \cite[5.2.1]{Lus94} (see also \cite{KR90, LS90}).
Recall that this is the element of $(U_q \mfg)^{\mc W_\int}$ defined on $M \in \mc W_\int$ by
\eqn{
\wt T_j|_{M_\mu} \coloneqq \sum_{a,b,c \in \Z_{\ge 0} \atop a-b+c = -\mu(h_j)} (-1)^b q_j^{b-ac} E_j^{(a)} F_j^{(b)} E_j^{(c)}|_{M_\mu}
}
for any $\mu \in \Plate$. We have $\wt T_j(M_\mu) \subseteq M_{\mu - (a-b+c) \al_j} = M_{s_j(\mu)}$ for any $\mu \in \Plate$. 
By \cite[5.2.3]{Lus94} $\wt T_j$ is invertible. It is well-known that the operators $\wt T_j$ satisfy the generalized braid relations \eqref{eq:gen-braid} and thus induce an action of $\Br{W}$ on any integrable representation \cite[39.4.3]{Lus94}. At $q=1$, this reduces to the action by triple exponentials described in Section \ref{ss:km-braid-int}.\\


For $j \in \indI$, let $\Ad(\wt T_j)$ 
be the algebra automorphism of $(U_q\mfg)^{\mc W_\int}$ given by conjugation by $\wt T_j$.
By \cite[37.1]{Lus94}, $\Ad(\wt T_j)$ preserves $U_q\mfg$ and $U_q\mfg'$. Thus, it restricts
to an automorphism of $U_q\mfg$ and satisfies $\wt T_j(x \cdot m) = \Ad(\wt T_j)(x) \cdot \wt T_j(m)$ for any $x \in U_q \mfg$ and $m \in M \in \mc W_\int$.
Also, by \cite[5.2.3 and 37.2.4]{Lus94} we have
\eq{ \label{Tj:bar}
\wb{\Ad(\wt T_j)(u)} = (-1)^{\mu(h_j)} q^{-(\al_j,\mu)} \Ad(\wt T_j^{-1})(\wb{u})\, ,
}
for any  $\mu \in \Qlat$ and $u \in (U_q\mfg)_\mu$.

\subsection{Diagrammatic operators}
We shall be interested in distinguished operators arising from finite type subdiagrams $X \subseteq I$. 
In this case, there is a well-defined element
\eq{ \label{TX:def}
	\wt T_X \coloneqq \wt T_{j_1} \cdots \wt T_{j_\ell} \in U_q \mfg^{\mc W_\int}
}
where $w_X = s_{j_1} \cdots s_{j_\ell}$ is any reduced expression of the longest element in $W_X$. 
Note that $\wt T_X$ is invertible and maps $M_\mu$ to $M_{w_X(\mu)}$, for all $M \in \mc W_\int$ and $\mu \in \Plate$.\\

By the explicit formulae in \cite[37.1.3]{Lus94}, it follows that
$\Ad(\wt T_X)(E_i) \in U_q\mfn^+$ for all $i \not\in X$.
Moreover, if $\tau \in \Aut_X(A)$, then, by the uniqueness of the longest element, 
we have $\tau\circ \Ad(\wt T_X) = \Ad(\wt T_X)\circ\tau$. Note also that
\eq{ \label{TX:rootspace}
\Ad(\wt T_X)(U_q\mfg_\mu) \subseteq U_q\mfg_{w_X(\mu)} 
}
for any $\mu\in\Qlat$. Define the algebra automorphism
\eq{ \label{tw:def}
\wt \om \coloneqq \om \circ \Ad(G_{\id,\rho}) = \Ad(G_{\id,-\rho}) \circ \om
}
where $\Ad(G_{\id,\rho})$ is given by \eqref{G:Ad}. 
Note that $\wt \om$ coincides with the automorphism $\tw$ from \cite[Sec.~7.1]{BK19}.
By \cite[Prop.~8.20]{Jan96}, one gets	
\eq{ \label{T:fixed}
\Ad(\wt T_X)|_{U_q\mfg_X} = \wt \om^{-1} \circ \oi_X|_{U_q\mfg_X}.
}
where $\oi_X$ is the opposition involution from \ref{sec:subdiagrams}.
By \cite[Lem.~7.1]{BK19}, $\wt \om$ commutes with $\Ad(\wt T_i)$ for any $i \in \indI$ and hence
\eq{
\label{tw:T:commute} \wt \om \circ \Ad(\wt T_X) = \Ad(\wt T_X) \circ \wt \om 
}
Recall that if $s_{j_1} \cdots s_{j_\ell} = w_X$ is a reduced decompostion, then the positive
roots in $\Phi_X$ are explicitly given by
\eq{
\Phi_X^+ = \{ \al_{j_1} ,s_{j_1}(\al_{j_2}), \ldots, (s_{j_1} \cdots s_{j_{\ell-1}})(\al_{j_\ell})\}.
}
Hence, \eqref{Tj:bar} implies
\eq{ \label{TX:bar}
\wb{\Ad(\wt T_X)(u)} = (-1)^{\mu(2\rho^\vee_X)} q^{-(2\rho_X,\mu)} \Ad(\wt T_X^{-1})(\wb{u}), \qq 
}
for any  $\mu \in \Qlat$ and $u \in (U_q\mfg)_\mu$.
Note that this is the key property used in \cite{BK19} to relate intertwiners for the subalgebra $U_q\mfk$ to the bar involution. 


\subsection{Pullback of integrable modules} 
As mentioned above, for any $j \in I$, $\Ad(\wt T_j)$ preserves $U_q\mfg$
and therefore it gives rise to a restriction functor on $U_q\mfg$-modules, which we refer
to as its pullback. We shall need the following. 
\begin{lemma} \label{lem:AdTj:Oint}
For all $j \in \indI$, the pullback $\Ad(\wt T_j)^*$ preserves integrable (resp. category $\mc O^\eps$ integrable) $U_q\mfg$-modules.
\end{lemma}

\begin{proof}
Let $M \in \mc W_\int$. 
For any $\mu\in\Plate$, we have $\Ad(\wt T_j)^*(M)_\mu=M_{s_j \mu}$,
therefore $\Ad(\wt T_j)^*(M)$ decomposes into weight spaces.
We shall prove that the action of $\Ad(\wt T_j)(U_q \mfg_{ \{ i \} })$ is locally finite for all $i \in \indI$.
Since the module $M$ is integrable, the subspace $U_q \mfg_{ \{ i \} } \cdot m$ is finite-dimensional for all $m \in M$. Therefore, $\wt T_j(U_q \mfg_{ \{ i \} } \cdot \wt T_j^{-1}(m))$ is finite-dimensional 
and $\Ad(\wt T_j)(U_q \mfg_{ \{ i \} }) = \wt T_j U_q \mfg_{ \{ i \} } \wt T_j^{-1}$ acts locally finitely on $\Ad(\wt T_j)^*(M)$. Similarly, since in this case the weight spaces are finite-dimensional, 
in order to show that $\Ad(\wt T_j)^*$ maps $\mc O^\eps_\int$ into $\mc O^\eps_\int$, it suffices to prove that $U_q\mfn^\eps$ act locally finitely on $\Ad(\wt T_j)^*(M)$ for all $M \in \mc O^\eps_\int$, which follows as before.
\end{proof}


\subsection{Diagrammatic half-balances}\label{ss:diag-half-bal}

As mentioned in Remark \ref{rmk:kr-factorization}, one of the main applications of the diagrammatic operators $\wt{T}_X$ is to provide an alternative description of the quasi-R-matrix $\Theta_X$. 
Indeed, by \cite[Thm.~3]{KR90} and \cite{LS90} (see also \cite[Lem.~3.8]{BK19}), we have the following coproduct formula:
\begin{align} 
\label{quasiRX:factorized:1}
\Theta_X &= \Del(\wt T_X^{-1}) \cdot (\wt T_X \ot \wt T_X). 
\end{align}
and therefore
\begin{align} 
\label{tildeRX:factorized:1}
\wtR_X &= (\wt T_X^{-1} \ot \wt T_X^{-1}) \cdot \Del(\wt T_X).
\end{align}
By \eqref{T:fixed}, the square of the operator $\wt{T}_X$ is almost central in $U_q\mfg_X$. 
Thus, for finite type subdiagrams, $\wt{T}_X$ is essentially a half-balance for $U_q\mfg_X$ up to a Cartan correction. This is particularly simple in the formal setting, \ie $q=e^{\hbar/2}\in\C[\negthinspace[\hbar]\negthinspace]$. In this case, the operator $q^{h_i(h_i+1)/2}\wt{T}_{s_i}$ is a half-balance for the quantum $\mathfrak{sl}_2$ subalgebra corresponding to $\alpha_i$. Note that this is a key property, which implies the rigidity of the representations of $\Br{W}$ given by the quantum Weyl group operators are essentially rigid cf.~\cite{ATL15}.
More generally, Kamnitzer and Tingley in \cite{KT09} (see also \cite[Definition\ 3.9]{ST09}) proved that the modified operator
\eq{ \label{hatT:def-KT}
	T_X \coloneqq G_{\id,\rho_X}\cdot \wt T_X = \wt T_X\cdot G_{\id,-\rho_X } 
}
satisfies 
\eq{\label{eq:RX-coprod-id}
{R}_X = (T_X^{-1} \ot T_X^{-1}) \cdot \Del(T_X)
}
where  $R_X$ is the universal R-matrix of $U_q\mfg_X$. Furthermore, $T_X^2$ is central and
therefore $T_X$ is a half-balance for the diagrammatic subalgebra $U_q\mfg_X$, cf.~ Section \ref{ss:bal-half-bal}.

In the next section, we shall discuss a modification of such diagrammatic half-balances, which
depends on the choice of a \emph{generalized Satake diagram} (cf.\ Section \ref{ss:mod-lusz-op}).


\section{Classical and quantum pseudo-fixed-point subalgebras} \label{s:coideal}

In this section we define the notions of (classical and quantum) \emph{pseudo}-fixed-point subalgebras of Kac-Moody type,  following \cite[Sections 2 and 3]{RV21}. 
Their combinatorial datum is given in terms of a \emph{generalized} Satake diagram, whose definition is a generalization of the finite type theory from \cite{RV20}. 
As a special case, one recovers the fixed-point subalgebras of involutions and their quantizations, whose theory for Kac-Moody type is developed in \cite{KW92, BBBR95} and \cite{Ko14}, for the classical and quantum case, respectively.\\

Along the way, we introduce a modification, depending upon a generalized Satake diagram, of the diagrammatic R-matrix and longest quantum Weyl group operator. 
These can be interpreted as \emph{modified diagrammatic half-balances}.


\subsection{Generalized Satake diagrams}
Just as Kac-Moody algebras $\mfg$ and their quantizations $U_q\mfg$ are defined in terms of the combinatorial datum $(I,A)$, we will define certain subalgebras of $\mfg$ and $U_q\mfg$ by adjoining some combinatorial datum which can be seen as a decoration of the Dynkin diagram.
We assume that $A$ is a symmetrizable indecomposable generalized Cartan matrix.

\begin{definition} \label{def:GSat}
Let $X \subseteq \indI$ be of finite type and $\tau \in \Aut_X(A)$ such that $\tau^2 = \id_I$ and $\tau|_X = \oi_X$.
We call a node $i\in \indI$ \emph{unsuitable for $(X,\tau)$} if $i \notin X$, $\tau(i)=i$ and the connected component of $X \cup \{ i \}$ containing $i$ is of type $\mathsf A_2$, or, equivalently, $\theta(\al_i) = -\al_i-\al_j$ and $a_{ji}=-1$ for some $j \in X$.
We call $(X,\tau)$ a \emph{generalized Satake diagram} if $\indI$ has no unsuitable nodes for $(X,\tau)$ and write $\GSat(A)$ for the set of generalized Satake diagrams. 
\hfill\defnend\end{definition}

This notion arose in \cite{He84} for $\mfg$ of finite type in the study of root system involutions and associated restricted Weyl groups.
In \cite{RV20} this generalization was found to describe coideal subalgebras of $U_q\mfg$ possessing a universal K-matrix for $\mfg$ of finite type (for suitable parameters).


\subsection{Pseudo-involutions} \label{sec:pseudoinvolutions}

From now on fix $(X,\tau) \in \GSat(A)$ and assume that the extended weight lattice $\Plate$ is $\tau$-compatible.
We consider the \emph{pseudo-involution}
\eq{ \label{theta:def}
\theta = \theta(X,\tau) \coloneqq \Ad(\wt w_X) \circ \om \circ \tau \in \Aut(\mfg)
}
cf.\ \cite[4.38-4.39]{KW92} and \cite[(2.26)]{RV21}.
From the fact that $\theta$ stabilizes $\mfh$ it follows that the dual map of $\theta$, $\theta^* \in \GL(\mfh^*)$ permutes root spaces: $\theta(\mfg_\al) = \mfg_{\theta^*(\al)}$. 
Therefore $\theta^*$ preserves $\Phi$ and hence also $\Qlat$. 
Since $\Plate$ is $\tau$-compatible, $\theta^*$ also preserves $\Plate$.\\

The terminology ``pseudo-involution'' is motivated by the fact that $\theta$ has properties similar to an honest Lie algebra involution.
Namely, $\theta^2(\mfg_\al) = \mfg_\al$ for all $\al \in \Phi$ and $\theta$ restricts to $\mfh$ and $\mfh'$ as an involution, namely $- w_X \circ \tau$.
The dual map $\theta^*$ is also given by the formula $-w_X \circ \tau$ and therefore we also denote it by $\theta$ henceforth.
Because the three automorphisms $\Ad(\wt w_X)$, $\om$ and $\tau$ commute, see \cite[Prop.~2.2 (3)]{Ko14}, it follows from \eqref{AdmX:square} that 
\eq{ \label{theta:squared}
\theta^2|_{\mfg_\al} = (-1)^{\al(2\rho^\vee_X)} \id_{\mfg_\al} \qq \text{for all } \al \in \Phi.
}

The above statements are satisfied for any pair $(X,\tau)$ with $X$ of finite type and $\tau \in \Aut_X(A)$ involutive.
For the following property of $\theta$ we also require the condition $\tau|_X = \oi_X$.
Combining \eqref{AdmX:X} and \eqref{theta:def} we obtain
\eq{ \label{theta:fixed}
\theta|_{\langle \mfg_X , \mfh^\theta \rangle} = \id_{\langle \mfg_X , \mfh^\theta \rangle}
}
and hence
\eq{ \label{theta:relation}
(\tau-\id) \circ (\theta-\id) = 0
}
as an identity in $\End_\C(\mfh)$ or $\End_\C(\mfh^*)$, see \cite[Equation\ (5.7)]{Ko14}.
From \eqref{theta:fixed}-\eqref{theta:relation} it follows that
\eq{ \label{htheta:decomposition}
(\mfh')^\theta = \mfh_X \oplus \bigoplus_{i\not\in X \atop i<\tau(i)} \C(h_i-h_{\tau(i)})
\qq\mbox{and}\qq
(\mfh')^{-\theta} = \bigoplus_{i\not\in X \atop i \le \tau(i)} \C(h_i + h_{\tau(i)}),
}
so that the number of $\tau$-orbits in $\indI \backslash X$ equals $\dim\big((\mfh')^{-\theta}\big)$,
which we refer to as \emph{restricted rank} of $(X,\tau)$. 
Note that $\mfh^\theta \subseteq \mfh'$ if $\cork(A) \leqslant 1$, in particular if $A$ is of finite or affine type.


\subsection{Comparison with standard Satake diagrams}\label{s:coideal:1}
One can always choose a group homomorphism $s \in \Hom_\grp(\Qlat,\{1,-1\})$ such that
\eq{ \label{s:condition}
\begin{split}
s(\al_i)&=1 \,\qq\qq\qq\qq\;\; \text{ if } i \in X \text{ or } \tau(i)=i \\ 
s(\al_{\tau(i)}) &= (-1)^{\al_i(2\rho^\vee_X)} s(\al_i) \qq \text{ otherwise}
\end{split}
}
(see \eg \cite[Eqns.\ (5.1) and (5.2)]{BK19}) and consider with the modified automorphism 
$\ol \theta = \Ad(s) \circ \theta$.
By \eqref{AdmX:square}, for any $i\in \indI$, one has
\[
\ol \theta^2|_{\mfg_{\al_i}} = s(\al_i)s(\theta(\al_i)) \Ad(\wt w_X)^2|_{\mfg_{\al_i}} = \frac{s(\al_i)}{s(\al_{\tau(i)})} (-1)^{\al_i(2\rho^\vee_X)} \id_{\mfg_{\al_i}}.
\]
It follows that $\ol \theta$ is an involution if $(X,\tau)$ is a \emph{Satake diagram},\ie if, in addition to the conditions in Definition \ref{def:GSat}, it holds $\al_i(\rho^\vee_X) \in \Z$ for any $i\not\in X$ such that $\tau(i)=i$ (cf.~ \cite[Definition\ 2.3]{Ko14}).
One checks that if $i\in \indI$ is an unsuitable node for $(X,\tau)$, then $i\not\in X$, $\tau(i)=i$, and $\al_i(\rho^\vee_X) = - \frac{1}{2}$. Thus, in the case of Satake diagrams, 
there are no unsuitable nodes and every Satake diagram is a generalized Satake diagram. Satake diagrams are known to describe and, up to a natural equivalence, classify involutive automorphisms of $\mfg$ of the second kind, see \cite[App.\ A]{Ko14} and cf.\ \cite[5.33]{KW92} and \cite[4.4-4.5]{BBBR95}.


\subsection{Pseudo-fixed-point subalgebras} \label{sec:pseudofixedpoint}
The introduction of unsuitable nodes in Definition \ref{def:GSat} is motivated by the following
generalization of fixed-point subalgebras with respect to an involutive Lie algebra automorphisms.
For $\bm i\in \indI^\ell$, $\ell \in \Z_{> 0}$, set
\[
f_{\bm i} \coloneqq \big( \ad(f_{i_1}) \circ \cdots \circ \ad(f_{i_{\ell-1}}) \big)(f_{i_\ell})
\]
and consider the vector space $\mfg_{\bm i,\theta} = \Sp_\C \big\{ f_{\bm i} , \theta(f_{\bm i}) \big\}$.
Choose a subset $\mc J \subset \cup_{\ell \in \Z_{> 0}} I^\ell$ such that the set $\{ f_{\bm i} \, | \, {\bm i} \in \mc J \}$ is a basis of $\mfn^-$.
Note that, for all $\bm i \in \mc J$, we have $\theta(f_{\bm i}) = f_{\bm i}$, if $\bm i \in X^\ell$ for some $\ell>0$, and $\theta(f_{\bm i}) \in \mfn^+$ otherwise.
Hence $\mfg_{\bm i,\theta}$ is one-dimensional if $\bm i \in X^\ell$ for some $\ell >0$ and two-dimensional otherwise.
We have the following decomposition of $\mfg$ as an $\ad(\mfh^\theta)$-module:
\eq{ \label{g:thetadecomposition}
\mfg = \mfn^+_X \oplus \mfh^\theta \oplus \bigoplus_{\bm i \in \mc J} \mfg_{\bm i, \theta} \oplus \mfh^{-\theta}.
}
Let $\bm \ga = (\ga_i)_{i\in \indI} \in (\C^\times)^{I}$ such that $\ga_i=1$ if $i \in X$.
We denote by the same symbol $\bm \ga$ the element of $\Hom_\grp(\Qlat,\C^\times)$ given by
$\bm \ga(\al_i) = \ga_i$.
Set
\eq{
\theta_{\bm \ga}  \coloneqq {\Ad(\bm \ga)} \circ \theta \in \Aut(\mfg).
}
Then $\theta_{\bm \ga}|_\mfh = \theta|_\mfh = -w_X \circ \tau$, $\theta_{\bm \ga}(\mfg_\al) = \mfg_{\theta(\al)}$ and $\theta_{\bm \ga}|_{\langle \mfg_X,\mfh^\theta \rangle} = \id_{\langle \mfg_X,\mfh^\theta \rangle}$ for all $\al \in \Phi$.
If $\theta_{\bm \ga}$ is indeed involutive, the decomposition of the fixed-point subalgebra $\mfg^{\theta_{\bm \ga}}$ will be supported in the first three components of \eqref{g:thetadecomposition} and its projection onto $\mfg_{\{ i \},\theta}$ is 
\[
\C (f_i + \theta_{\bm \ga}(f_i)) = \C (f_i + \ga_i \theta(f_i)).
\]
The following definition naturally generalizes fixed-point subalgebras of involutive algebra automorphisms of the second kind, cf.\ \cite[Definition\ 2 and Rmk.\ 3 (i)]{RV20}.

\begin{definition}
Suppose that $X \subseteq \indI$ is of finite type and that $\tau \in \Aut_X(A)$ is an involution satisfying $\tau|_X = \oi_X$; furthermore suppose that $\bm \ga \in (\C^\times)^{I}$ {such that $\ga_i=1$ for all $i \in X$}.
In terms of \eqref{theta:def}, we define the \emph{pseudo-fixed-point subalgebra}
\eqn{
\mfk = \mfk_{\bm \ga}(X,\tau) = \langle \mfg_X , \mfh^{\theta} , \{ b_i \, | \, i\not\in X \}\rangle = \langle \mfn^+_X, \mfh^\theta, \{ b_i \, | \, i\in \indI \} \rangle,
}
where 
\begin{flalign}
&& b_i = b_{i;\ga_i} \coloneqq \begin{cases} f_i & \text{if } i \in X, \\ f_i + \ga_i \theta(f_i) & \text{otherwise}. \end{cases} && \defnend
\end{flalign}
\end{definition}


The subalgebra $\mfk$ is closely related to the fixed-point subalgebra $\mfg^\theta$ as we
explain in the following. Note that $\mfk \cap \mfg^\theta$ contains $\langle \mfg_X ,\mfh^\theta \rangle$ and $b_i \in \mfg_{\{ i \},\theta}$.
It is natural to require $\mfk$ to be supported in the first three components of \eqref{g:thetadecomposition} and to be such that the projection on $\mfg_{\{ i \},\theta}$ is $\C b_i$.
Consider 

\begin{align}\label{Ieq:def} 
\begin{split}
\Ieq &\coloneqq \{ i\in \indI \, | \, i<\tau(i), \, (\theta(\al_i),\al_i)=0  \}\\ 
&\;= \{ i\not\in X \, | \, i<\tau(i), \, \forall j \in X \cup \{ \tau(i)\} \; a_{ij} = 0 \}, \\[0.5em]
\Ga &\coloneqq \{ \bm \ga \in (\C^\times)^{\indI} \; | \; \forall i \in X \; \ga_i = 1 \text{ and } \forall i \in \Ieq \; \ga_i = \ga_{\tau(i)}  \}.
\end{split}
\end{align}
The following key result motivates the definition of a generalized Satake diagram.
Set $b_{\bm i} \coloneqq \big( \ad(b_{i_1}) \circ \cdots \circ \ad(b_{i_{\ell-1}}) \big)(b_{i_\ell})$ for $\bm i\in \indI^\ell$.

\begin{theorem}[{\cite[Thm.~3.8]{RV21}}] \label{thm:GSat}
Let $X \subseteq \indI$ be of finite type, $\tau \in \Aut_X(A)$ an involution satisfying $\tau|_X = \oi_X$, and $\bm \ga \in (\C^\times)^{\indI}$ such that $\ga_i=1$ for all $i \in X$.
The following statements are equivalent:
\begin{enumerate}\itemsep0.25cm
\item $(X,\tau) \in \GSat(A)$ and $\bm \ga \in \Ga$;
\item $\mfk = \mfn^+_X \oplus \mfh^\theta \oplus \bigoplus_{\bm i \in \mc J} \C b_{\bm i}$ as $\ad(\mfh^\theta)$-modules;
\item $\mfk \cap \mfh = \mfh^\theta$.
\end{enumerate}
\end{theorem}

Note that from (ii) we deduce that $\mfk$ projects onto the first three summands in the decomposition \eqref{g:thetadecomposition}. Moreover, if $\cork(A)\leqslant 1$, then $\mfk \subseteq \mfg'$,
since in this case $\mfh^\theta \subseteq \mfh'$.\\

As in \cite[Cor.~2.9]{Ko14}, the generators of the universal enveloping algebra $U(\mfk_{\bm \ga}(X,\tau))$ corresponding to the $b_i$ can be further modified by adding a scalar term, 
yielding an additional tuple of parameters. A similar phenomenon occurs in the q-deformed 
case discussed in \cite[Rmk.~5.10]{Le99} and \cite{Ko14},  yielding however a non-trivial
deformation.


\subsection{Modified diagrammatic half-balances}\label{ss:mod-lusz-op}
As before, let $A$ be a symmetrizable indecomposable generalized Cartan matrix, $(X,\tau) \in \GSat(A)$, and $\Plate$ a $\tau$-compatible extended weight lattice.\\

Note that $\Ad(\wt w_X)$ and $\om$ commute as elements of $\Aut(\mfg)$ and $\Aut(U\mfg)$.
By  \eqref{tw:T:commute}, the quantum analogues $\Ad(\wt T_X)$ and $\wt \om$ also commute.
However,  $\wt \om$ is not an involution and the interaction of $\Ad(\wt T_X)$ and $\wt \om$ with the quasitriangular bialgebra structure of $U_q\mfg$ is not optimal.
On the other hand, the involution $\om \in \Aut_\alg(U_q\mfg)$ is a coalgebra antiautomorphism 
and does interact nicely with the quasitriangular bialgebra structure.
Similarly, one can use instead the element $ T_X$ as in \eqref{hatT:def-KT}, 
which commutes with the undeformed $\om$ and resolves the diagrammatic universal R-matrix \eqref{eq:RX-coprod-id}.\\

Modifying the definition of $T_X$, we introduce another correction of the element $\wt T_X$ 
which enjoys similar properties and is explicitly tailored around $U_q\mfk_{\bm \ga,\bm \si}$.

\begin{definition}\label{defn:mod-lusz-op}
The \emph{modified diagrammatic half-balance} associated to the generalized Satake diagram 
$(X,\tau)$ is the operator on integrable $U_q\mfg_X$-modules
\begin{flalign} \label{hatT:def}
&& T_{X,\tau} \coloneqq G_{\theta(X,\tau),\rho_X} \wt T_X = \wt T_X G_{\theta(X,\tau),-\rho_X } 
\in (U_q\mfg_X)^{\mc W_\int}\,, && 
\end{flalign}
where $\theta(X,\tau)$ is the pseudo-involution associated to $(X,\tau)$ and 
$G_{\theta(X,\tau),\rho_X} \in \Fun(\Plate , \F^\times)$ is defined in \eqref{G:def}.
\hfill\defnend
\end{definition}
Note that $T_{X,\tau}$ is an invertible element in $(U_q\mfg)^{\mc W_\int}$ and $T_{X,\tau}(M_\la) \subseteq M_{w_X(\la)}$ for any $M \in \mc W_\int$ and $\la\in\Plate$.

\begin{remark}
If $I$ is of finite type and $X=\indI$, this coincides with the diagrammatic half-balance \eqref{hatT:def-KT} from \cite{KT09}
(cf. Section \ref{ss:diag-half-bal}). Similar modifications of
$\wt T_X$ also appeared in \cite[Sec.~4.4]{Ko14} and \cite[Sec.~4.2 and App.\ A]{CM18}.
\hfill\rmkend
\end{remark}

The algebra automorphism $\Ad(T_{X,\tau}) \in \Aut_\alg(U_q\mfg)$ satisfies several
useful properties, which motivate the definition of $T_{X,\tau}$.

\begin{lemma}
\hfill
\begin{enumerate}\itemsep0.25cm
	\item The analogue of \eqref{TX:rootspace} holds:
	\begin{align}
		\label{hatT:rootspace}
		\Ad(T_{X,\tau})(U_q\mfg_\la) &\subseteq U_q\mfg_{w_X(\la)}\,,
	\end{align}
	for any $\la\in\Qlat$\,.
	\item The analogue of \eqref{T:fixed} holds:
	\begin{align}
		\label{hatT:fixed}
		\Ad(T_{X,\tau})|_{U_q\mfg_X} &= \Ad(G_{\theta-\id,\rho_X-\rho}) \circ \om \circ \oi_X|_{U_q\mfg_X} = \om \circ \oi_X|_{U_q\mfg_X}\,.
	\end{align}
	\item The analogue of \eqref{tw:T:commute} holds:
	\begin{align}
		\label{om:hatT:commute} 
		\om \circ \Ad(T_{X,\tau}) &= \Ad(G_{\theta-\id,\rho-\rho_X}\wt T_X) \circ \wt \om = \Ad(T_{X,\tau}) \circ \om\,.
	\end{align}
	\item The analogue of  \eqref{TX:bar}  holds:
	\begin{align}
	\label{hatT:bar} 
	\wb{\Ad(T_{X,\tau})(u)} &= (-1)^{\la(2\rho^\vee_X)} \Ad(T_{X,\tau}^{-1})(\wb{u})\,,
	\end{align}
	for any $\la \in \Qlat$ and $u\in (U_q\mfg)_\la$.
\end{enumerate}
\end{lemma}

\begin{proof}
We note that \eqref{hatT:rootspace} and \eqref{hatT:fixed} follow, respectively, from \eqref{TX:rootspace} and \eqref{T:fixed}. 
The properties \eqref{wX:rhominusrhoX} and \eqref{tw:T:commute} imply that $\om$ and $\Ad(T_{X,\tau})$ commute, yielding \eqref{om:hatT:commute}. 
Finally, \eqref{hatT:bar} follows from \eqref{TX:bar}.  
\end{proof}


\subsection{Modified diagrammatic R-matrices} \label{sec:RX}
In analogy with Section \ref{ss:diag-half-bal}, the element $T_{X,\tau}$ can be thought of as an 
half-balance resolving a {\em modified} diagrammatic universal R-matrix (cf.\ \ref{ss:bal-half-bal}). Note that, since $\Plate$ is $\tau$-compatible, the map $\theta= \theta(X,\tau) = -w_X \circ \tau \in \End(\mfh^*)$ restricts to $\Plate$ and is self-adjoint.

\begin{definition}
The \emph{modified diagrammatic R-matrix} corresponding to $(X,\tau)$ is the operator
\begin{flalign} 
\label{RX:def}
&&	R_{X,\tau} \coloneqq \ka_{\theta(X,\tau)} \wtR_X && 
\end{flalign}
where $\ka_{\theta(X,\tau)}$ is defined as in \ref{ga:def}. \hfill \defnend
\end{definition}

The definition is motivated by the following result.

\begin{lemma}
\hfill
\begin{enumerate}\itemsep0.25cm
		\item  The intertwining identity holds:
	\begin{align}
		\label{RX:intw}
		R_{X,\tau} \Del(u) &= \Del^\op(u) R_{X,\tau}
	\end{align}
		for any $u \in U_q\mfg_X$.
	\item  The following coproduct identities holds:
	\begin{align}
		\label{RX:factorized}
		R_{X,\tau} &= (T_{X,\tau}^{-1} \ot T_{X,\tau}^{-1}) \cdot \Del(T_{X,\tau})\,,
	\end{align}
	and
	\begin{align}
		\label{RX:factorized2}
		(R_{X,\tau})_{21} &= \Del(T_{X,\tau}) \cdot (T_{X,\tau}^{-1} \ot T_{X,\tau}^{-1})\,.
	\end{align}
\end{enumerate}
\end{lemma}

\begin{proof}
We first prove \eqref{RX:factorized}. Since $\theta(X,\tau)$ and $w_X$ commute, as a consequence of \eqref{G:Delta} and \eqref{tildeRX:factorized:1} we have, as required,
\begin{flalign*}
&& R_{X,\tau} &= \ka_{\theta(X,\tau)} \cdot (\wt T_X^{-1} \ot \wt T_X^{-1}) \cdot \Del(\wt T_X) \\
&& &= (\wt T_X^{-1} \ot \wt T_X^{-1}) \cdot \ka_{\theta(X,\tau)} \cdot \Del(\wt T_X) \\
&& & = \big( (G_{{\theta(X,\tau)},\rho_X} \wt T_X)^{-1} \ot (G_{{\theta(X,\tau)},\rho_X} \wt T_X)^{-1} \big) \cdot \Del(G_{{\theta(X,\tau)},\rho_X} \wt T_X) \\
&& &= (T_{X,\tau}^{-1} \ot T_{X,\tau}^{-1}) \cdot \Del(T_{X,\tau}). && 
\end{flalign*}
Observing that $\theta(X,\tau)$ fixes $\mfh^*_X$ pointwise and recalling that $\theta(X,\tau)$ preserves roots, we obtain $\theta(X,\tau)|_{\Qlat_X} = \id_{\Qlat_X}$, so that
\[
\ka_{\theta(X,\tau)} \wb{\Del(\wb{u})} = \Del^\op(u) \ka_{\theta(X,\tau)}
\]
for all $u\in U_q\mfg_X$.
Finally, \eqref{RX:intw} follows from \eqref{tildeR:intw}.
Combining \eqref{RX:factorized} and \eqref{RX:intw}, we obtain \eqref{RX:factorized2}.
\end{proof}


\subsection{Quantum pseudo-involutions} \label{sec:q-pseudo-inv}
We have two triples of commuting algebra automorphisms, \ie $(\Ad(\wt T_X), \wt \om,\tau)$ and $( \Ad(T_{X,\tau}), \om, \tau)$, both of which specialize to the triple $(\Ad(\wt w_X),\om,\tau)$ of algebra automorphisms of $U\mfg$ providing a factorization of the pseudo-involution $\theta$.
Therefore, we obtain two distinct quantum analogues of $\theta$ (see also Remark \ref{rmk:no-sign}):
\eq{ \label{thetaq:def} 
\wt \theta_q = \wt \theta_q(X,\tau) \coloneqq \Ad(\wt T_X) \circ \wt \om \circ \tau \qq\mbox{and}\qq  \theta_q =  \theta_q(X,\tau) \coloneqq \Ad(T_{X,\tau}) \circ \om \circ \tau.
}
By \cite[Thm.~4.4 (1-2)]{Ko14}, the map $\wt \theta_q$ preserves certain nice properties of $\theta$, listed in Lemma \ref{lem:theta-theta} below. 
The same properties remain true for $\theta_q$, thanks to the relation
\eq{ \label{hattheta:relation}
\wt \theta_q = \Ad(G_{\id-\theta,\rho_X-\rho}) \circ  \theta_q =  \theta_q \circ \Ad(G_{\id-\theta,\rho-\rho_X}).
}
Note that \eqref{hattheta:relation} follows from \eqref{om:hatT:commute} and the observation that $\Ad(G_{\id-\theta,\rho_X-\rho})$ fixes $U_q\mfg_X$ pointwise (a consequence of the invariance of $G_{\id-\theta,\rho_X-\rho}$ under the translation action of $\Qlat_X$ on $\Plate$).
Therefore, we have the following
\begin{lemma}\label{lem:theta-theta}
\hfill
\begin{enumerate}\itemsep0.25cm
\item For any $\la\in\Qlat$, $\wt \theta_q((U_q\mfg)_\la) = \theta_q((U_q\mfg)_\la) = (U_q\mfg)_{\theta(\la)}$.
\item For any $h\in\Qlat^\vee_\ext$, $\wt \theta_q(t_h)=\theta_q(t_h)=t_{\theta(h)}$. 
\item For any $\la\in\Qlat$, $\wt \theta_q(t_\la) =  \theta_q(t_\la) = t_{\theta(\la)}$\,.
\end{enumerate}
In particular, $U_q\mfh$ and $U_q\mfg'$ are both $\wt \theta_q$-stable and $ \theta_q$-stable 
and
\eq{
	\label{thetaq:fixed} 
	\wt \theta_q(u)=u=\theta_q(u)\,.
}
for any $u\in U_q\mfg_X U_q\mfh^\theta$.
\end{lemma}

Note that \eqref{thetaq:fixed} follows immediately from \eqref{T:fixed} and \eqref{hatT:fixed},
since $\tau|_X = \oi_X$.

\begin{remark}
In summary, we have introduced in parallel two series of data 
\[
\big( \wt \om, \wt T_X, \Xi_X, \wt \theta_q(X,\tau) \big) \qq \longleftrightarrow \qq \big( \om,T_{X,\tau},R_{X,\tau},\theta_q(X,\tau) \big).
\]
The data on the left are predominant in \cite{BK19} and are crucial in the construction of the quasi-K-matrix.
However, once the quasi-K-matrix has been constructed, we will mainly work with their counterparts on the right, which enjoy more convenient relations with the coproduct structure. 
\rmkend
\end{remark}


\subsection{Quantum pseudo-fixed-point subalgebras}\label{sec:q-coideal}
Together with the subset $\Ieq$ of $\indI \backslash X$ defined in \eqref{Ieq:def}, we also consider
\eq{
\label{Ins:def}
\Ins \coloneqq \{ i\in \indI \, | \, \theta(\al_i) = -\al_i \} = \{ i\not\in X \, | \, \tau(i)=i, \, \forall j \in X \; a_{ij}=0 \}.
}
Accordingly, we consider the following sets of parameter tuples:
\begin{align}
\Ga_q &= \Ga_q(X,\tau) \coloneqq \{ \bm \ga \in (\F^\times)^{\indI} \; | \; \forall i \in X \; \ga_i = 1 \text{ and } \forall i \in \Ieq \; \ga_i = \ga_{\tau(i)} \} \\
\Si_q &= \Si_q(X,\tau) \coloneqq \big\{ \bm \si \in \F^{\indI} \, \big| \, \forall i\in \indI \backslash \indI_{\rm ns} \; \si_i = 0 \text{ and } \forall (i,j)\in \indI_{\rm ns}^2 \; a_{ij} \in 2\Z \text{ or } \si_j = 0 \big\}. \hspace{-8pt}  
\end{align}

\begin{definition} \label{Uqk:def}
Let $(X,\tau) \in \GSat(A)$ and $(\bm \ga,\bm \si) \in \Ga_q\times \Si_q$. 
For $i \in I$ we set
\eq{
\label{Bi:def}
B_i = B_{i;\ga_i,\si_i}(X,\tau) \coloneqq \begin{cases} 
F_i & \text{if } i \in X, \\
F_i + \ga_i  \theta_q(X,\tau)(F_i) + \si_i t_i^{-1} & \text{otherwise}. \end{cases}
}
The \emph{quantum pseudo-fixed-point subalgebra} corresponding to $(X,\tau,\bm \ga,\bm \si)$ is the subalgebra
of $U_q\mfg$ given by 
\begin{flalign}
&& U_q\mfk = U_q\mfk_{\bm \ga,\bm \si}(X,\tau) = \langle U_q\mfg_X, U_q\mfh^\theta, \{ B_i \, | \, i\not\in X \} \rangle = \langle U_q\mfn^+_X, U_q\mfh^\theta, \{ B_i \, | \, i\in \indI \} \rangle.
&& \hfill\defnend
\end{flalign}
\end{definition}

The main result of the present paper is that, up to completion, $U_q\mfk$ is equipped with a universal K-matrix. The following Lemma compares our expression of the generators with that used in \cite{Ko14} and \cite{BK19}.

\begin{lemma} \label{lem:Bi:theta}
Let $i\not\in X$.
Then
\eq{
\label{Bi:theta} B_i = F_i + G_{\theta,-\rho_X}(\al_i) \ga_i \wt \theta_q(X,\tau)(F_i t_i) t_i^{-1} + \si_i t_i^{-1}.
}
\end{lemma}

\begin{proof}
In the proof we write $\theta$ instead of $\theta(X,\tau)$ for simplicity, and similarly for $\wt \theta_q$ and $ \theta_q$.
We have
\begin{align*}
\theta_q(F_i) &= -\Ad(G_{\theta,\rho_X} \wt T_X)(E_{\tau(i)}) \\
&= -G_{\theta,\rho_X}(w_X(\al_{\tau(i)})) \Ad(\wt T_X)(E_{\tau(i)}) t_{\theta(w_X(\al_{\tau(i)}))} \\
&= G_{\theta,-\rho_X}(\al_i) \wt \theta_q(F_i t_i) t_i^{-1}
\end{align*}
by virtue of Lemma \ref{lem:G} (iii) and the definitions \eqref{thetaq:def}.
Now \eqref{Bi:theta} follows as an immediate consequence.
\end{proof}

\begin{remark}\label{rmk:no-sign}
Note that the map $\theta(X,\tau)$ as defined in \cite{Ko14,BK19}  
differs from ours by having an extra factor $\Ad(s)$. This guarantees that $\theta(X,\tau)$ is an involutive automorphism of $\mfg$. Nevertheless, its quantum analogue $ \theta_q(X,\tau)$ 
is not an involutive automorphisms of $U_q\mfg$. 
Moreover, it follows from \eqref{Bi:theta} that, if $(X,\tau)$ is a Satake diagram, our expression 
for $B_i$ corresponds precisely to that used in \cite[(5.8)]{BK19} upon identifying, for any $i\in I\backslash X$,
\eq{ \label{gamma:relationwithc}
\ga_i = s(\al_{\tau(i)}) G_{-\theta(X,\tau),\rho_X}(\al_i) c_i 
}
where $s \in \Hom_\grp(\Qlat,\{1,-1\})$ is constrained by \eqref{s:condition}. \hfill \rmkend
\end{remark}


\subsection{Structure of $U_q\mfk$} \label{sec:Uqk:structure}

By \cite[Prop.~5.2]{Ko14}, $U_q\mfk$ is a right coideal of $U_q\mfg$; the proof of this statement requires $\tau|_X = \oi_X$ but not the condition that $(X,\tau)$ has no unsuitable nodes, or anything stronger.
More precisely, by \cite[Eq.~(5.5)]{Ko14}, we have
\eq{ \label{Del:Bi}
\Del(B_i) - B_i \ot t_i^{-1} \in U_q\mfg_X U_q\mfh^\theta \ot U_q\mfg.
}
Furthermore, it follows from \cite[Prop.~6.2]{Ko14} that
\eq{ \label{intersection}
U_q\mfk \cap U_q\mfh = U_q\mfh^\theta,
} 
providing a quantum analogue of the identity $\mfk \cap \mfh = \mfh^\theta$.
If $\cork(A) \leqslant 1$ (in particular if $A$ is of affine type), then $U_q\mfh^\theta \subseteq U_q\mfh'$. Combining with \eqref{intersection}, we obtain $U_q\mfk \subseteq U_q\mfg'$.
These statements follow from the analysis of the expressions ${\sf Ser}_{ij}(B_i,B_j)$ in \cite[Sec.~5.3]{Ko14} for $(\bm \ga,\bm \si) \in \Ga_q \times \Si_q$.
This analysis remains valid for any generalized Satake diagram. 


\section{The quasi-K-matrix} \label{s:quasiK} 

In this section, we review the construction of the so-called \emph{quasi}-K-matrix, which is the essential ingredient in the construction of universal K-matrices for quantum groups. 
We present a more general and simpler construction as we explain below.\\

In \cite{ES18} Ehrig and Stroppel studied the categorification of certain coideal subalgebras of $U_q\mfgl_N$ and the associated skew Howe duality.
For the same quantum symmetric pairs, Bao and Wang in \cite{BW18} developed a coideal version of Lusztig's theory of canonical bases.
Central to both papers is the notion the {\em internal} bar involution,\ie an algebra automorphism of $U_q\mfk$ which is a suitable analogue of Lusztig's bar involution of $U_q\mfg$. 
Note indeed that the latter does not preserve $U_q\mfk$ and therefore the two involutions
do not coincide on $U_q\mfk$.
In \cite{BW18} this leads to the first definition of the quasi-K-matrix, as a canonical $U_q\mfk$-intertwiner between the restriction of Lusztig's bar involution  and the internal bar involution. 
Balagovi\'{c} and Kolb generalized the existence of the internal bar involution and the associated quasi-K-matrix in \cite{BK15,BK19} to all quantum symmetric pairs of finite type considered by G. Letzter \cite{Le99,Le02,Le03} and to a large class of quantum symmetric Kac-Moody pairs.
Finally, and more importantly for us, it was shown in \cite{BK19} that the quasi-K-matrix is a key factor of the universal K-matrix of quantum symmetric pairs of finite type.\\

In the above works, the existence of the internal bar involution relies on certain constraints on the parameters $\bm \ga$ and $\bm \si$ (\eg \cite[(5.16)-(5.17)]{BK19}). The latter apply also to the quasi-K-matrix, since its construction requires the internal bar involution.
In this section we provide a generalization of the latter result, valid for quantum symmetric pairs of Kac-Moody type, which does not rely on the existence of the internal bar involution and does not require special constraints on the parameters (Theorem~\ref{thm:sum-quasi-k}). 
In fact, as shown by Kolb in \cite{Ko21}, this in turn can be used to \emph{define} the internal bar involution for $U_q\mfk$.


\subsection{Locally inner automorphisms} 
We recalled in Section \ref{s:QG} that the construction of the quasi-R-matrix due to Lusztig essentially amounts to produce a solution to the following problem: find an element $\mathsf{X}$ in $U_q\mfg^{\ot 2}$ (or rather in a suitable completion) such that the subalgebras $\wb{\Del\big(\wb{U_q\mfg}\big)},\, \Del(U_q\mfg)\subset U_q\mfg^{\ot 2}$ are pointwise related by conjugation by $\mathsf{X}$.
Note that the two subalgebras do not coincide.
Therefore, if the solution $\mathsf{X}$ exists, it is certainly non-trivial. 
This is indeed the defining intertwining equation satisfied by $\mathsf{X}=\wtR$ \eqref{tildeR:intw}.\\

This suggests the notion of {\em locally inner automorphisms}, that is, global automorphisms of algebras, which become inner (or rather topologically inner) when restricted to a distinguished subalgebra. We shall consider the following situation.
Let $A \subseteq B \subseteq C$ be a tower of (unital associative) algebras over a field $\F$.
Let $A_0$ be a given generating set for $A$. 
Suppose we have a function $f:A_0 \to B$ and an element $c \in C^\times$ such that $f(a_0) = \Ad(c)(a_0)$ for all $a_0 \in A_0$.
We extend $f$ to an algebra embedding $A \to C$ by setting $f(a) = \Ad(c)(a)$.
Since every element of $A$ can be written as a linear combination of products of the elements of $A_0$, it follows that $f$ maps $A$ into $B$.
Finally, restricting the codomain of $f$ to the subalgebra $f(A)$, we obtain an algebra isomorphism between $A$ and $f(A)$. 
Furthermore the set $f(A_0)=\{\Ad(c)(a_0) \, | \, a_0 \in A_0\}$ is clearly a generating set of $f(A)$.


\subsection{The case of the bar involution}\label{ss:loc-inner-bar}
Proving that the bar involution is a locally inner automorphism is the problem at 
the origin of the quasi-K-matrix. We shall consider the situation described above with $A = U_q\mfk_{\bm \ga,\bm \si}(X,\tau)$, $B = U_q\mfg$ and $C = (U_q\mfg)^{\mc O^+}$ for $(X,\tau) \in \GSat(A)$ and $(\bm \ga , \bm \si) \in \Ga_q \times \Si_q$.
The set $A_0$ will simply be the canonical set of generators $U_q\mfg_X U_q\mfh^\theta \cup \{ B_{i;\ga_i,\si_i} \, | \, i \not\in X \}$ of $U_q\mfk_{\bm \ga,\bm \si}$.
We choose $f|_{A_0}$ as follows:
\eq{
\begin{aligned}
f|_{U_q\mfg_X U_q\mfh^\theta} = \id_{U_q\mfg_X U_q\mfh^\theta}, 
\qq\mbox{and}\qq
f(B_{i;\ga_i,\si_i}) = \wb{B_{i;\ga'_i,\si'_i}} \qq (i \not\in X)
\end{aligned}
}
where in addition to the parameter tuples $\bm \ga = (\ga_i)_{i \in \indI} \in \Ga_q$, $\bm \si = (\si_i)_{i \in \indI} \in \Si_q$, we have chosen alternative tuples $\bm \ga' = (\ga'_i)_{i \in \indI} \in \Ga_q$, $\bm \si' = (\si'_i)_{i \in \indI} \in \Si_q$, to be specified later, see \eqref{prime:def}.
The element $c$ will be given by an element
\[
\wtKM = \wtKM_{\bm \ga,\bm \si}(X,\tau) \in (U_q\mfn^+)^{\mc O^+} \subset (U_q\mfg)^{\mc O^+}
\]
which we will construct in the following.

The condition $\Ad(\wtKM)(u) = f(u)$ for all $u \in U_q\mfg_X U_q\mfh^\theta \cup \{ B_{i;\ga_i,\si_i} \, | \, i \not\in X \}$ is equivalent to requiring
\begin{align}
	\label{tildek:Bi} \wtKM \, B_{i;\ga_i,\si_i} = \wb{B_{i;\ga'_i,\si'_i}}\, \wtKM
	\qq\mbox{and}\qq \wtKM \, u = u \, \wtKM 
\end{align}
for $i\in\indI \backslash X$ and $u \in U_q\mfg_XU_q\mfh^\theta$.
By the above discussion,  
we obtain that $f(U_q\mfk_{\bm \ga,\bm \si})$ is an algebra isomorphic to $U_q\mfk_{\bm \ga,\bm \si}$, generated by $U_q\mfg_XU_q\mfh^\theta$ and $\{ \wb{B_{i;\ga'_i,\si'_i}} \, | \, i\not\in X\}$.
Since the bar involution is an algebra automorphism of $U_q\mfg$ preserving $U_q\mfg_X U_q\mfh^\theta$, we have $f(U_q\mfk_{\bm \ga,\bm \si}) = \wb{U_q\mfk_{\bm \ga',\bm \si'}}$. 

\begin{remark}
Similarly to the case of the quasi-R-matrix, we shall see in the following sections that $f(U_q\mfk) = \Ad(\wtKM)(U_q\mfk)$ and $U_q\mfk$ are different subalgebras of $U_q\mfg$, for generic values of $(\bm \ga,\bm \si)$.
Define the \emph{internal} bar involution of $U_q\mfk_{\bm \ga,\bm \si}$ to be the composition
\eq{ \label{barB:def}
\wb{\phantom{i} \cdot \phantom{i}}^B \coloneqq \wb{\phantom{i} \cdot \phantom{i}} \circ f.
}
On the subbialgebra $U_q\mfg_X U_q\mfh^\theta$ this coincides with the usual bar involution (note that this map preserves $U_q\mfg_X U_q\mfh^\theta$).
The condition $(\bm \ga',\bm \si') = (\bm \ga,\bm \si)$ imposed in \cite{BK15,BK19} implies that the internal bar involution fixes $B_{i;\ga_i,\si_i}$ for all $i \not\in X$.
Note that {\em ibid.} it is furthermore assumed that $\wb{\; \cdot \;}^B$ preserves $U_q\mfk_{\bm \ga,\bm \si}$.
Detailed information on the presentation of $U_q\mfk_{\bm \ga,\bm \si}$, in particular the quantum Serre relations satisfied by the generators $B_i$, is then necessary to deduce that $\wb{\; \cdot \;}^B$ is indeed an algebra automorphism of $U_q\mfk_{\bm \ga,\bm \si}$, cf.\ \cite[Thm.~5.6 (1)]{BK19}.\\

Our approach allows instead to completely avoid the use of the internal bar involution, 
and hence does not require detailed results on the presentation of $U_q\mfk_{\bm \ga,\bm \si}$.
Indeed, we show below that the proofs in \cite[Sections 6 and 9.2]{BK19} are independent of the condition $(\bm \ga',\bm \si') = (\bm \ga,\bm \si)$. Thus, we obtain a quasi-K-matrix in a more general setting. 
In \cite{Ko21} this observation is further exploited to prove the existence {\em a fortiori} of the internal bar involution in the case $(\bm \ga',\bm \si') = (\bm \ga,\bm \si)$.
\hfill\rmkend
\end{remark}


\subsection{The involution on the set of parameters}

Define a map $'$ on $\F^{I}$ via
\eq{ \label{prime:def}
(\bm x')_i = x'_i \coloneqq (-1)^{\al_i(2 \rho^\vee_X)} \wb{x_{\tau(i)}} 
}
for $\bm x \in \F^{I}$ and $i \in \indI$.

\begin{lemma}
The map $'$ defined by \eqref{prime:def} is an involution which preserves $\Ga_q$ and $\Si_q$.
Moreover, for the latter, it maps $\bm \si$ to $\wb{\bm \si}$.
\end{lemma}

\begin{proof}
Note that $'$ restricts to $(\F^\times)^{I \backslash X}$.
From $\tau \in \Aut_X(A)$ it follows that $(\bm \ga'')_i = \ga_i$ and that $\ga'_i = \ga'_{\tau(i)}$ if and only if $\ga_i = \ga_{\tau(i)}$, for all $i \not\in X$, which proves the claim for $\Ga_q$. 
The claim for $\Si_q$ follows immediately from the fact that $\si_i=0$ if $\tau(i) \ne i$ or $\al_i(\rho^\vee_X) \ne 0$.
\end{proof} 


\subsection{Parameters as  elements in $(U_q\mfh)^{\mc O^+_\int}$} \label{sec:tuples:Uqhhat}
By a mild abuse of notation, given $\bm x \in (\F^\times)^I$, we shall denote by the same symbol  the corresponding character of the root lattice $\bm x \in \Hom_\grp(\Qlat,\F^\times)$ given by
$\bm x(\al_i) = x_i$ for $i \in \indI$.
By further abuse of notation, by the same symbol we shall denote an arbitrary extension to a group homomorphism of $\Plate$.
Note that such extensions exist since $\Qlat \subset \Plate$ is an embedding of abelian groups and $\F^\times$ is a divisible abelian group.
Finally, we will denote by $\bm x$ also the corresponding element of $(U_q\mfh)^{\mc O^+_\int}$
defined as in Section~\ref{sec:Uqhhat}. In the following, we shall consider the tuples
$\bm \ga'$, $\bm \ga$, and $\bm \ga^{-1} = (\ga_i^{-1})_{i \in \indI}$ as elements in 
$(U_q\mfh)^{\mc O^+_\int}$.


\subsection{The quasi-K-matrix}
We state the main result of this section.
It is convenient to rewrite the system \eqref{tildek:Bi}. 
In order to be able to apply Lusztig's theory of skew derivations directly, we write 
\eq{
\wtKM_{\bm \ga,\bm \si} = \wb{\mfX_{\bm \ga,\bm \si}}
}
for some $\mfX = \mfX_{\bm \ga,\bm \si} \in (U_q\mfn^+)^{\mc O^+}$. 
We note that the system \eqref{tildek:Bi} is equivalent to
\begin{align}
	\label{quasiK:Bi} \mfX \, \wb{B_{i;\ga_i,\si_i}} = B_{i;\ga'_i,\si'_i}\, \mfX \qq\mbox{and}\qq
	\mfX \, u = u \, \mfX 
\end{align}
for $i \not\in X$ and $u \in U_q\mfg_XU_q\mfh^\theta$
(note that the bar involution preserves $U_q\mfg_XU_q\mfh^\theta$).\\

The rest of the section is devoted to the construction of the quasi-K-matrix $\mf X$ based on its intertwining properties and the computation of its coproduct. 

\begin{theorem}\label{thm:sum-quasi-k}
	For all $(\bm \ga,\bm \si) \in \Ga_q \times \Si_q$ we have the following two results:
	\begin{enumerate}\itemsep0.25cm
		\item
There is a unique operator $\mfX = \mfX_{\bm \ga,\bm \si}\in(U_q\mfn^+)^{\mc O^+}$ of the form
\[\mfX = \sum_{\la \in (\Qlat^+)^{-\theta}} \mfX_{\bm \ga,\bm \si;\la}\] 
such that $\mfX_{\bm \ga,\bm \si;0} = 1$, $\mfX_{\bm \ga,\bm \si;\la} \in (U_q\mfn^+)_\la$ and the system \eqref{quasiK:Bi} is satisfied.
		\item 
		The following coproduct identiy holds:
		\eqn{
			\Del(\mfX) = (\mfX \ot 1) \cdot (\Ad(\bm \ga') \circ \theta_q^{-1} \ot \id)(\Theta) \cdot \Ad(\ka_\id)(1 \ot \mfX) \cdot \Theta_X^{-1}.
		}
	\end{enumerate}
\end{theorem}

This is a generalization of analogue results from \cite{BK19}.
The proof of (i) (Theorem \ref{thm:quasiK:2}) is carried out in Section \ref{ss:intertwining-bis}-\ref{sec:generalsigma}. 
The proof of (ii) (Theorem \ref{thm:quasiK:Delta}) is carried out in Section \ref{ss:coprod-formula}-\ref{ss:coprod-formula-pf}.

\begin{remark}
We follow the same approach used in \cite[Sec.~6]{BK19}. The results for arbitrary values of $\bm \si \in \Sigma_q$ are obtained from the special case $\bm \si = \bm 0$ (Theorem \ref{thm:quasiK}).
This relies on the arguments given in \cite[Sec.~3.5]{DK19} and simplifies the computations significantly. \hfill\rmkend
\end{remark}


\subsection{The intertwining property}\label{ss:intertwining-bis}
The key ingredient to prove the existence of the intertwiner is the use of the so-called Lusztig 
skew derivations. 
To this end, we shall first find an equivalent formulation of the intertwining system. 

\begin{lemma}
	For all $i \not\in X$, 
	\begin{align}
	\label{AdTX:bar2} \wb{\Ad(T_{X,\tau})(E_{\tau(i)})} &= (-1)^{\al_i(2 \rho^\vee_X)} \Ad(T_{X,\tau}^{-1})(E_{\tau(i)}), \\
	\nonumber \Ad(t_i T_{X,\tau})(E_{\tau(i)}) &= q^{-(\theta(\al_i),\al_i)} \Ad(T_{X,\tau})(E_{\tau(i)}).
	\end{align}
	Moreover,
	\begin{align}
	\nonumber 
	\wb{B_{i;\ga_i,\si_i}} &= F_i - \big( \ze_i \, \Ad(\wt T_X^{-1})(E_{\tau(i)}) - \wb{\si_i} \big) t_i \\
	\label{Bi:prime} B_{i;\ga'_i,\si'_i} &= F_i - t_i^{-1} \big( \ze_{\tau(i)} \, \Ad(\wt T_X)(E_{\tau(i)}) - \wb{\si_i} \big)
	\end{align}
	where
	\eq{ \label{phi:def} 
		\ze_i \coloneqq \begin{cases} (-1)^{\al_i(2\rho^\vee_X)} G_{-\theta,-\rho_X}(\al_i) \wb{\ga_i} = G_{-\theta,-\rho_X}(\al_i) \ga'_{\tau(i)} & \text{if } i \not\in X \\ 0 & \text{if } i \in X. \end{cases}
	}
\end{lemma}

\begin{proof}
	It is enough to observe that \eqref{AdTX:bar2} follows from 
	\eqref{hatT:bar} and the defining relations of $U_q\mfg$.
	Then, the explicit formula \eqref{Bi:def} implies \eqref{Bi:prime}.
\end{proof}

It follows that \eqref{quasiK:Bi} together with the condition $\mfX F_i = F_i \mfX$ 
($i \in X$) is equivalent to the condition
	\eq{ \label{quasiK:Bi:2}
		\big[ \mfX, F_i \big] = \mfX \big( \ze_i \, \Ad(\wt T_X^{-1})(E_{\tau(i)}) - \wb{\si_i} \big) t_i - t_i^{-1} \big( \ze_{\tau(i)} \, \Ad(\wt T_X)(E_{\tau(i)}) - \wb{\si_i} \big) \mfX.
	}
for all $i \in I$.


\subsection{Skew derivations}

We recall some basic facts from \cite{Lus94} and \cite{Jan96}.
Let $i \in \indI$ and note that $\Ad(t_i)$ is an algebra automorphism of $U_q\mfn^+$.
Following \cite[1.2.13]{Lus94}, let $D_i^{(\ell)}, D_i^{(r)}\in \End_\F( U_q\mfn^+)$  be the unique linear maps (denoted $r_i$ and ${}_ir$ in \emph{ibid.}) such that $D_i^{(\ell)}(E_j) = \del_{ij}=D_i^{(r)}(E_j)$ for any $j \in \indI$ and
\eq{ \label{Di:def}
D_i^{(r)}(uu') = D_i^{(r)}(u) \Ad(t_i)(u') + u D_i^{(r)}(u'), \qu D_i^{(\ell)}(uu') = D_i^{(\ell)}(u)u' + \Ad(t_i)(u)D_i^{(\ell)}(u')
}
for any $u,u' \in U_q\mfn^+$.
They satisfy $D_i^{(\ell)}((U_q\mfn^+)_\la) \subseteq (U_q\mfn^+)_{\la - \al_i}\supseteq D_i^{(r)}((U_q\mfn^+)_\la)$ for all $\la \in \Qlat^+$ and
\eq{ 
\label{Di:op} \op \circ D_i^{(r)} = D_i^{(\ell)} \circ \op,
}
where $\op$ is the unique algebra antiautomorphism of $U_q\mfg$ which fixes each $E_i$ and $F_i$ ($i \in \indI$) and inverts each $t_h$ ($h \in \Qlat^\vee_\ext$). 
Recall that the following  properties hold.
\begin{enumerate}\itemsep0.25cm
	\item By \cite[Prop.~3.1.6]{Lus94},
	\eq{ \label{Di:commutatorFi}
		[u, F_i ] = \frac{D_i^{(r)}(u)t_i - t_i^{-1} D_i^{(\ell)}(u)}{q_i-q_i^{-1}}.
	}
	for any $u \in U_q\mfn^+$.
	\item By \cite[Lem.~1.2.15 (a)]{Lus94},
	\eq{ \label{Di:zero}
		u = 0 \qq \Leftrightarrow \qq \forall i \in \indI \; D_i^{(r)}(u) = 0 \qq \Leftrightarrow \qq \forall i \in \indI \; D_i^{(\ell)}(u) = 0 
	}
	for any $u \in U_q\mfn^+_\la$ with $\la \in Q^+ \backslash \{0\}$.
	\item By \cite[Lem.~10.1]{Jan96}, 
	\eq{ \label{DiDj:commute}
		D_i^{(r)} \circ D_j^{(\ell)} = D_j^{(\ell)} \circ D_i^{(r)}
	}
	for any $i,j\in\indI$.
	\item By \cite[1.2.13]{Lus94},
	 \eq{
	 	\label{Di:pairing} \langle F_i v,u \rangle = \frac{1}{q_i^{-1}-q_i} \langle v, D_i^{(\ell)}(u) \rangle
	 	\qq \mbox{and}\qq
	 	\langle v F_i,u \rangle = \frac{1}{q_i^{-1}-q_i} \langle v, D_i^{(r)}(u) \rangle
	 }
 	for any $u\in U_q\mfn^+$ and $v \in U_q\mfn^-$.
\end{enumerate}

Note that the maps $D_i^{(\ell)}, D_i^{(r)}$ naturally extend to $(U_q\mfn^+)^{\mc O^+}$ (roughly, the latter consists of formal series in $U_q\mfn^+$ converging on category ${\mc O}^+$ modules, cf.~Section~\ref{ss:quantum-cat-O}).

\subsection{The intertwining property in terms of skew derivations}
We shall use the skew derivations to provide an equivalent description of the system \eqref{quasiK:Bi}.
By \eqref{Di:commutatorFi} and the linear independence of $t_i$ and $t_i^{-1}$ over $(U_q\mfn^+)^{\mc O^+}$, 
$\mfX$ is a solution of \eqref{quasiK:Bi} if and only if it is a solution of the following system in $(U_q\mfn^+)^{\mc O^+}$:
\begin{align}
\label{quasiK:Bi:3a} D_i^{(r)}(\mfX) &= (q_i-q_i^{-1}) \mfX \big( \ze_i  \, \Ad(\wt T_X^{-1})(E_{\tau(i)}) - \wb{\si_i} \big), \\
\label{quasiK:Bi:3b} D_i^{(\ell)}(\mfX) & = (q_i-q_i^{-1}) \big( \ze_{\tau(i)} \, \Ad(\wt T_X)(E_{\tau(i)}) - \wb{\si_i} \big) \mfX, 
\end{align}
for any $i \in \indI$.
For $\la \in \Qlat$, let $\mfX_\la$ be the projection of $\mfX$ on the root space $(U_q\mfn^+)_\la$ with respect to the root space decomposition $U_q\mfn^+ = \bigoplus_{\la \in \Qlat} (U_q\mfn^+)_\la$.
Then, $\mfX_\la = 0$ if $\la \not\in\Qlat^+$ and we get the following result.

\begin{lemma}
Let $\mfX \in (U_q\mfn^+)^{\mc O^+}$ be an invertible element.
Then, $\mfX$ is a solution of \eqref{quasiK:Bi} if and only if
\begin{align}
\label{quasiK:Bi:cpts:a}
D_i^{(r)}(\mfX_\la) &= (q_i-q_i^{-1}) \big(  \ze_i \, \mfX_{\la-\al_i+\theta(\al_i)} \Ad(\wt T_X^{-1})(E_{\tau(i)}) - \wb{\si_i} \mfX_{\la-\al_i} \big), \\
\label{quasiK:Bi:cpts:b}
D_i^{(\ell)}(\mfX_\la) & = (q_i-q_i^{-1}) \big( \ze_{\tau(i)} \, \Ad(\wt T_X)(E_{\tau(i)}) \mfX_{\la-\al_i+\theta(\al_i)} - \wb{\si_i} \, \mfX_{\la-\al_i} \big)
\end{align}
for any $\la \in \Qlat^+$ and $i \in \indI$.
\end{lemma}
We normalize $\mfX$ by setting $\mfX_0 = 1$.
In order to show that the system \eqref{quasiK:Bi:cpts:a}-\eqref{quasiK:Bi:cpts:b} 
has a solution, we rely on \cite[Prop.~6.3]{BK19}.

\begin{proposition} \label{prop:A:eqns}
Let $\mu \in \Qlat^+$ be a positive weight of height $\geqslant 2$ and 
$A^{(r)}_i,A^{(\ell)}_i \in (U_q\mfn^+)_{\mu - \al_i}$, for $i\in\indI$, a collection of given elements.
\begin{enumerate}\itemsep0.25cm
\item  There exists $u \in (U_q\mfn^+)_\mu$ such that, for any $i\in\indI$, 
\eq{
\label{Di:eqns}
D_i^{(r)}(u) = A^{(r)}_i \qq\mbox{and}\qq D_i^{(\ell)}(u) = A^{(\ell)}_i 
}
if and only if, for any $i,j \in \indI$, we have
\eq{ \label{A:compatibility}
D_i^{(r)}(A^{(\ell)}_j) = D_j^{(\ell)}(A^{(r)}_i) 
}
and, for $i \ne j$,
\eq{ \label{A:eqns}
\frac{1}{q_i-q_i^{-1}} \sum_{s=1}^{1-a_{ij}} (-1)^s \binom{1-a_{ij}}{s}_{q_i} \big\langle F_i^{1-a_{ij}-s} F_j F_i^{s-1},A^{(r)}_i \big\rangle = \frac{1}{q_j^{-1}-q_j} \big\langle F_i^{1-a_{ij}} , A^{(r)}_j \big\rangle.
}
\item If the system \eqref{Di:eqns} has a solution, it is unique.
\end{enumerate}
\end{proposition}

\begin{remark}
If $\mfX = \sum_{\la \in \Qlat} \mfX_\la$ with $\mfX_0=1$ is a solution of the system \eqref{quasiK:Bi:cpts:a}-\eqref{quasiK:Bi:cpts:b}, then so is $\op(\mfX)|_{\ga_i \leftrightarrow \ga_{\tau(i)}}$. 
Therefore, by uniqueness, $\op(\mfX)|_{\ga_i \leftrightarrow \ga_{\tau(i)}} = \mfX$.\hfill \rmkend
\end{remark}

Finally, proceeding exactly as in \cite[case $(3) \Rightarrow (4)$ of the proof of Prop. 6.1]{BK19}, we get the following result.
\begin{proposition} \label{prop:quasiK:zerocomponents}\label{prop:quasiK:Et}
	Let $\mfX \in (U_q\mfn^+)^{\mc O^+}$ be an invertible solution of \eqref{quasiK:Bi:cpts:a}. 
	Then, $\mfX_\la = 0$ unless $\la \in (\Qlat^+)^{-\theta}$, \ie
	$\mfX$ has the form
	\[\mfX = \sum_{\la \in (\Qlat^+)^{-\theta}} \mfX_{\bm \ga,\bm \si;\la}\] 
	with $\mfX_{\bm \ga,\bm \si;\la} \in (U_q\mfn^+)_\la$. Moreover, 
	$[\mfX,u]=0$ for any $u \in U_q\mfn^+_XU_q\mfh^\theta$.
\end{proposition}


\subsection{The case $\bm \si = \bm 0$}

We prove the result in the case $\bm \si = \bm 0$. Namely, we have the following.

\begin{theorem} \label{thm:quasiK}
	For any $\bm \ga \in \Ga_q$ and $\bm \si = \bm 0$, there exists a unique solution $\mfX$ of the
	system \eqref{quasiK:Bi} of the form 
	$\mfX = \sum_{\mu \in (\Qlat^+)^{-\theta}} \mfX_\mu$ with $\mfX_0 = 1$ and $\mfX_\mu \in (U_q\mfn^+)_\mu$.
\end{theorem}

Note that, by uniqueness, $\wb{\mfX} = \mfX^{-1}|_{\bm \ga \mapsto \bm \ga'}$. The proof is carried out in 
Sections \ref{thm:intertwiner-0-part-1}-\ref{thm:intertwiner-0-part-2}.

\subsubsection{}\label{thm:intertwiner-0-part-1}
The recursive construction of $\mfX$ through Proposition \ref{prop:A:eqns} (i) relies on the following technical result.
\begin{lemma}\label{lem:fundlemma}\hfill
	\begin{enumerate}\itemsep0.25cm
		\item For any $i\not\in X$, $\big( D_i^{(r)} \circ \Ad(\wt T_X)\big)(E_i)$ is fixed by $\op \circ \tau$.
		\item Let $i,j \in \indI$ such that $i \ne j$ and consider $\la_{ij}\coloneqq(1-a_{ij})\al_i + \al_j \in \Qlat^+$.
		If $\theta(\la_{ij}) = -\la_{ij}$, then either $\tau(j)=i \in \Ieq \cup \tau(\Ieq)$ or $i,j \in \Ins$.
		\item 
Let $\mu \in \Qlat^+$ and $j \in \indI \backslash X$.
If for all $\la \le \mu$ we have elements $\mfX_\la \in (U_q\mfn^+)_\la$ satisfying \eqref{quasiK:Bi:cpts:a} for all $i \in \indI$ and $\mfX_\mu \ne 0$ then $\mu \in \Z_{\ge 0} (\al_j - \theta(\al_j)) + \Qlat^+_{I \backslash \{ j \}}$.
	\end{enumerate}
\end{lemma}

\begin{proof}
	(i) The statement that it is fixed by $\op \circ \tau$ is \cite[Thm.~4.1]{BW21}, the proof of which does not use the condition $\al_i(\rho^\vee_X) \in \Z$, so that it holds for all pairs $(X,\tau)$ such that $X \subseteq I$ of finite type, $\tau \in \Aut_X(A)$ is involutive and $\tau|_X = \oi_X$.\\
	
	\noindent
	(ii) This is \cite[Lem.~6.4]{BK19}. 
	Note that in the proof nothing stronger than the defining condition of $\GSat(A)$ is used, namely that there exists no pair $(i,j) \not\in X \times X$ such that $\tau(i)=i$, the connected component of $X$ neighbouring $i$ is $\{j\}$ and $a_{ij}=-1=a_{ji}$.\\
	
	\noindent
	(iii) This is \cite[Lem.~6.5]{BK19}. Recall that we assumed $\si_j=0$ at the start of this section.
\end{proof}

\subsubsection{}\label{thm:intertwiner-0-part-2}

We fix $\mu \in \Qlat^+$ and assume that for all $\la < \mu$ we have constructed elements $\mfX_\la \in (U_q\mfn^+)_\la$ satisfying $\mfX_0=1$ and, for all $i \in \indI$, \eqref{quasiK:Bi:cpts:a}-\eqref{quasiK:Bi:cpts:b}.
Define, for all $i \in \indI$, the following elements in $(U_q\mfn^+)_{\mu-\al_i}$:
\begin{align} 
\label{Ai:def} A^{(r)}_i &\coloneqq (q_i-q_i^{-1}) \ze_i \, \mfX_{\mu-\al_i+\theta(\al_i)} \Ad(\wt T_X^{-1})(E_{\tau(i)}) , \\
\label{Aprimei:def} A^{(\ell)}_i &\coloneqq (q_i-q_i^{-1}) \ze_{\tau(i)} \, \Ad(\wt T_X)(E_{\tau(i)}) \mfX_{\mu-\al_i+\theta(\al_i)}.
\end{align}
We will now prove \eqref{A:compatibility}-\eqref{A:eqns} for the above choices of $A^{(\ell)}_i, A^{(r)}_i$.

\begin{proposition} \label{prop:A:compatibility}
With $A^{(\ell)}_i, A^{(r)}_i$ as in \eqref{Ai:def}-\eqref{Aprimei:def}, the condition \eqref{A:compatibility} is satisfied.
\end{proposition}

\begin{proof}
We follow the proof of \cite[Lem.~6.7]{BK19}.
The crucial observation is that, for all $i,j \in \indI$, $D_j^{(r)}(\Ad(\wt T_X)(E_{\tau(i)})) = 0$ unless $j=i$, see \cite[Equation (5.10)]{BK19}, which goes back to \cite[Lem.~7.2]{Ko14}.
As a consequence, by the defining property of $D_i^{(r)}$ and the induction hypothesis most terms in $D_i^{(r)}(A^{(\ell)}_j)$ and $D_j^{(\ell)}(A^{(r)}_i)$ match pairwise.
We have $D_i^{(r)}(A^{(\ell)}_j) - D_j^{(\ell)}(A^{(r)}_i) = 0$ if $i$ or $j$ lies in $X$ or if $\tau(i) \ne j$.
Without loss of generality we may assume $j=\tau(i) \not\in X$.
In this case
\begin{align*}
D_i^{(r)}(A^{(\ell)}_j) - D_j^{(\ell)}(A^{(r)}_i) &= (q_j-q_j^{-1}) q^{(\mu-\al_j+\theta(\al_j),\al_i)} \,  \ze_i \, \big( D_i^{(r)} \circ \Ad(\wt T_X) \big)(E_{\tau(j)})\mfX_{\mu-\al_j+\theta(\al_j)} + \\
& \qq -  (q_i-q_i^{-1}) q^{(\al_j,\mu-\al_i+\theta(\al_i))} \, \ze_i \, \mfX_{\mu-\al_i+\theta(\al_i)} \big( D_j^{(\ell)} \circ \Ad(\wt T_X^{-1}) \big) (E_{\tau(i)}).
\end{align*}
In this case $q_i=q_j$ and $q^{(-\al_j+\theta(\al_j),\al_i)} = q^{(\al_j,-\al_i+\theta(\al_i))}$ so that  
\begin{align*}
D_i^{(r)}(A^{(\ell)}_j) - D_j^{(\ell)}(A^{(r)}_i) &= (q_j-q_j^{-1}) q^{(\al_j,\theta(\al_i)-\al_i+\mu)} \, \ze_i \, \Big( q^{(\mu,\al_i-\al_j)} \big( D_i^{(r)} \circ \Ad(\wt T_X) \big)(E_i)\mfX_{\mu-\al_j+\theta(\al_j)} + \\
& \hspace{60mm} - \mfX_{\mu-\al_i+\theta(\al_i)} \big( D_j^{(\ell)} \circ \Ad(\wt T_X^{-1}) \big) (E_j) \Big).
\end{align*}
Recall \eqref{theta:relation}.
Also, note that $\big( D_j^{(\ell)} \circ \Ad(\wt T_X^{-1}) \big) (E_j)$ lies in $U_q\mfn^+_X$ and hence commutes with $\mfX_{\mu-\al_i+\theta(\al_i)}$.
We obtain
\begin{align*}
D_i^{(r)}(A^{(\ell)}_j) - D_j^{(\ell)}(A^{(r)}_i) &= (q_j-q_j^{-1}) q^{(\al_j,\theta(\al_i)-\al_i+\mu)} \, \ze_i \, \mfX_{\mu-\al_i+\theta(\al_i)}  \cdot \\
& \qq \cdot \Big( q^{(\mu,\al_i-\al_j)} \big( D_i^{(r)} \circ \Ad(\wt T_X) \big)(E_i) - \big( D_j^{(\ell)} \circ \Ad(\wt T_X^{-1}) \big) (E_j) \Big).
\end{align*}
If $\mfX_{\mu-\al_i+\theta(\al_i)}=0$ we obtain the desired statement; hence we may assume that it is nonzero. 
By Proposition \ref{prop:quasiK:zerocomponents} we have $\theta(\mu)=-\mu$.
Applying \eqref{theta:relation} again, we have $(\mu,\al_i-\al_j)=0$.
By \cite[37.2.4]{Lus94} we have $\op \circ \Ad(\wt T_X) \circ \op = \Ad(\wt T_X^{-1})$.
Recalling \eqref{Di:op}, we obtain
\begin{align*}
D_i^{(r)}(A^{(\ell)}_j) - D_j^{(\ell)}(A^{(r)}_i) &= (q_j-q_j^{-1}) q^{(\al_j,\theta(\al_i)-\al_i+\mu)} \, \ze_i \, \mfX_{\mu-\al_i+\theta(\al_i)} \cdot \\
& \qq \cdot \Big( \big( D_i^{(r)} \circ \Ad(\wt T_X) \big)(E_i) - \big( \op \circ D_j^{(r)} \circ \Ad(\wt T_X) \big) (E_j) \Big).
\end{align*}
Finally, Lemma \ref{lem:fundlemma} (i) implies $D_i^{(r)}(A^{(\ell)}_j) - D_j^{(\ell)}(A^{(r)}_i) =0$, as required.
\end{proof}

\begin{proposition}
Let $A^{(r)}_i$ be given by \eqref{Ai:def}. 
Then \eqref{A:eqns} is satisfied for all $i,j \in \indI$ such that $i \ne j$.
\end{proposition}

\begin{proof}
We may follow the proof of \cite[Lem.~6.8]{BK19}.
Note that $F_i^{1-a_{ij}-s} F_j F_i^{s-1} \in (U_q\mfn^-)_{\la_{ij}-\al_i}$ and $F_i^{1-a_{ij}} \in (U_q\mfn^-)_{\la_{ij}-\al_j}$.
By the non--degeneracy of the bilinear pairing, we only need to consider the case that $\mu = \la_{ij}$.
By Proposition \ref{prop:quasiK:zerocomponents}, we may assume $\theta(\mu)=-\mu$ and by Lemma \ref{lem:fundlemma} (ii) we are in one of two possible cases: $\tau(j)=i \in \Ieq \cap \tau(\Ieq)$ or $i,j \in \Ins$.
In the former case, we have $\mu = \al_i + \al_j$, $\ga_i=\ga_j$ and $q_i=q_j$.
It follows that $A^{(r)}_i =  (q_i-q_i^{-1}) \wb{\ga_i} E_j$ and $A^{(r)}_j =  (q_i-q_i^{-1}) \wb{\ga_i} E_i$, so that \eqref{A:eqns} is an immediate consequence of $\langle F_j, E_j \rangle =\langle F_i , E_i \rangle$.

It remains to consider the case $i,j \in \Ins$, for which we can now follow the first part of \cite[Proof of Lemma 6.8, Case 2]{BK19}.
Namely, we invoke Lemma \ref{lem:fundlemma} (iii) and deduce that 
\[
(1-a_{ij}) \al_i + \al_j \in \Z_{\ge 0} (\al_j - \theta(\al_j)) + \Qlat^+_{I \backslash \{j \}} = 2\Z_{\ge 0} \al_j + \Qlat^+_{I \backslash \{ j \}}, 
\]
which is a contradiction. 
Hence this case does not occur and there is nothing left to prove.
\end{proof}

Finally, relying on the previous results, the proof of \cite[Thm.~6.10]{BK19} applies and we deduce Theorem \ref{thm:quasiK}.


\subsection{The intertwining property of $\mfX$ for general $\bm \si$} \label{sec:generalsigma}

Theorem \ref{thm:quasiK} generalizes as follows.

\begin{theorem} \label{thm:quasiK:2}
	For any $(\bm \ga,\bm \si) \in \Ga_q \times \Si_q$, there is a unique $\mfX = \mfX_{\bm \ga,\bm \si} = \sum_{\la \in (\Qlat^+)^{-\theta}} \mfX_{\bm \ga,\bm \si;\la} \in (U_q\mfn^+)^{\mc O^+}$ such that $\mfX_{\bm \ga,\bm \si;0} = 1$, $\mfX_{\bm \ga,\bm \si;\la} \in (U_q\mfn^+)_\la$ and the system \eqref{quasiK:Bi} is satisfied.
\end{theorem}

The proof relies on a generalization of the arguments made in \cite[Sec.~3.5]{DK19} to the case $(\bm \ga',\bm \si') \ne (\bm \ga,\bm \si)$ and  is carried out in \ref{thm:quasiK-2-1}-\ref{thm:quasiK-2-2}.

\subsubsection{}\label{thm:quasiK-2-1}
By \cite[Thm.~7.1]{Ko14}, the algebra $U_q\mfk_{\bm \ga,\bm \si}$ has a presentation in terms of generators and relations, which are independent of $\bm \si$. That is, the assignments 
\[
\phi_{\bm \si}(B_{i;\ga_i,0}) = B_{i;\ga_i,\si_i}\qq\mbox{and}\qq\phi_{\bm \si}(u) = u
\]
for $i \not\in X$ and $u\in U_q\mfg_XU_q\mfh^\theta$,
define an algebra isomorphism
\[
\phi_{\bm \si}: U_q\mfk_{\bm \ga,\bm 0} \to U_q\mfk_{\bm \ga,\bm \si} .
\]
Hence, $\chi_{\bm \si}\coloneqq\eps \circ \phi_{\bm \si}: U_q\mfk_{\bm \ga,0} \to \F$
is a one-dimensional representation.
Note that $\chi_{\bm \si}(B_{i;\ga_i,0}) = \si_i$ for $i \not\in X$.

\begin{lemma}
We have the following identities of morphisms  of algebras $U_q\mfk_{\bm \ga,\bm 0}\to U_q\mfg$:
\begin{align}
\label{phi:factorization1} \phi_{\bm \si} &= (\chi_{\bm \si} \ot \id) \circ \Del, \\
\label{phi:factorization2} \phi_{\bm \si} &= \wb{\phantom{i} \cdot \phantom{i}} \circ (\chi_{\bm \si'} \ot \id) \circ (\wb{\phantom{i} \cdot \phantom{i}}^B \ot \wb{\phantom{i} \cdot \phantom{i}}) \circ \Del.
\end{align}
where the map $\wb{\phantom{i} \cdot \phantom{i}}^B: U_q\mfk_{\bm \ga,\bm 0} \to U_q\mfk_{\bm \ga',\bm 0}$ is defined by \eqref{barB:def}.
\end{lemma}

\begin{proof}
The relations \eqref{phi:factorization1}-\eqref{phi:factorization2} can be verified by checking on generators. 
Applying both sides of \eqref{phi:factorization1} to $u \in U_q\mfg_XU_q\mfh^\theta$ we obtain $u$ for either side.
Furthermore, applying $\phi_{\bm \si} \ot \id$ to \eqref{Del:Bi} implies
\[
\big( ( \phi_{\bm \si} \ot \id) \circ \Del \big)(B_{i;\ga_i,0}) - B_{i;\ga_i,\si_i} \ot t_i^{-1} = \Del(B_{i;\ga_i,0}) - B_{i;\ga_i,0} \ot t_i^{-1}.
\]
To this we apply $\eps \ot \id$ and deduce, using $\eps(B_{i;\ga_i,\si_i}) = \si_i$, that $(\chi_{\bm \si} \ot \id) \big(\Del(B_{i;\ga_i,0}) \big) = B_{i;\ga_i,\si_i}$.
This completes the proof of \eqref{phi:factorization1}.\\

As for \eqref{phi:factorization2}, one can check that the right-hand side fixes $E_i$ and $F_i$ for $i \in X$ and $t_h$ for $h \in (\Qlat^\vee_\ext)^\theta$ pointwise, so that \eqref{phi:factorization2} is true when restricted to $U_q\mfg_X U_q\mfh^\theta$.
It remains to prove that 
\eq{ \label{phi:factorization2:Bi}
\wb{\big( \chi_{\bm \si'} \circ \wb{\phantom{i} \cdot \phantom{i}}^B \ot \wb{\phantom{i} \cdot \phantom{i}} \big)(\Del(B_{i;\ga_i,0}))} = B_{i;\ga_i,\si_i}
}
for $i \not\in X$. If $i \notin \Ins$, then $\si_i=\si'_i=0$. Thus, from the identity
\[
\big( \chi_{\bm \si'} \circ \wb{\phantom{i} \cdot \phantom{i}}^B \ot \wb{\phantom{i} \cdot \phantom{i}} \big)(B_{i;\ga_i,0} \ot t_i^{-1}) = \chi_{\bm \si'}(B_{i;\ga_i,0}) \ot t_i = 0\,,
\] 
\eqref{Del:Bi}, and $\eps(\wb{u}) = \eps(u)|_{q \to q^{-1}}$, we deduce that 
\[
\wb{\big( \chi_{\bm \si'} \circ \wb{\phantom{i} \cdot \phantom{i}}^B \ot \wb{\phantom{i} \cdot \phantom{i}} \big)(\Del(B_{i;\ga_i,0}))} = 
\wb{\big( \eps \circ \wb{\phantom{i} \cdot \phantom{i}} \ot \wb{\phantom{i} \cdot \phantom{i}} \big)(\Del(B_{i;\ga_i,0}))}
=
\big( \eps \ot \id \big)(\Del(B_{i;\ga_i,0})) = B_{i;\ga_i,0}\,.
\]
Therefore, \eqref{phi:factorization2:Bi} is satisfied in this case.
On the other hand, if $i \in \Ins$, then by \eqref{Bi:theta} we have
\[
B_{i;\ga_i,\si_i} = F_i - q_i^{-1} \ga_i E_i t_i^{-1} + \si_i t_i^{-1}
\qq\mbox{and}\qq
\Del(B_{i;\ga_i,0}) = B_{i;\ga_i,0} \ot t_i^{-1} + 1 \ot B_{i;\ga_i,0}.
\]
Therefore, 
\[
\wb{\big( \chi_{\bm \si'} \circ \wb{\phantom{i} \cdot \phantom{i}}^B \ot \wb{\phantom{i} \cdot \phantom{i}} \big)(\Del(B_{i;\ga_i,0}))} = 
\wb{\si_i' t_i + \wb{B_{i;\ga_i,0}}} = B_{i;\ga_i,0} + \si_i t_i^{-1} = B_{i;\ga_i,\si_i}\,.
\]
The result follows.
\end{proof}

\subsubsection{}\label{thm:quasiK-2-2}
Following \cite[3.1]{BW18} and \cite[Sec.~3.3]{Ko20}, we consider the \emph{2-tensor quasi-K-matrix for $U_q\mfk_{\bm \ga,\bm \si}(X,\tau)$}, \ie
the operator in $(U_q\mfg^{\ot 2})^{\mc O^+}$ given by
\eq{ \label{quasiRB:def}
\Theta_{\bm \ga,\bm \si} \coloneqq \Del(\mfX_{\bm \ga,\bm \si}) \cdot \Theta \cdot \mfX_{\bm \ga,\bm \si}^{-1} \ot 1\,.
}

By \cite[Prop.~3.2]{BW18} and \cite[Prop.~3.9]{Ko20}, it satisfies
\eq{ \label{quasiRB:intw}
\Theta_{\bm \ga,\bm \si}  \; \Big(\big(\wb{\phantom{i} \cdot \phantom{i}}^B \ot \wb{\phantom{i} \cdot \phantom{i}}\big) \circ \Del\Big)(b) = \Del\Big(\wb{b}^B\Big) \; \Theta_{\bm \ga,\bm \si}
}
for $b \in U_q\mfk_{\bm \ga,\bm \si}$.
By \cite[Prop.~3.10, cf.~Rmk.~3.11]{Ko20} (see also \cite[Prop.~3.5]{BW18}), 
the operator $\Theta_{\bm \ga,\bm \si}$ is given by a series
\eq{ \label{quasiRB:decomposition}
\Theta_{\bm \ga,\bm \si}  =  \sum_{\la \in \Qlat^+} \Theta_{\bm \ga,\bm \si;\la} \qq
\mbox{where}\qq \Theta_{\bm \ga,\bm \si;\la} \in U_q\mfk_{\bm \ga,\bm \si} \ot U_q\mfn^+_\la.
}
We then obtain the following generalization of \cite[Prop.~3.26]{DK19}.
\begin{proposition}
For any $(\bm \ga , \bm \si) \in \Ga_q \times \Si_q$, the operator in $(U_q\mfn^+)^{\mc O^+}$
given by
\[
\mfX'_{\bm \ga,\bm \si} \coloneqq (\chi_{\bm \si'} \ot \id)(\Theta_{\bm \ga,\bm 0})
\]
satisfies the system \eqref{quasiK:Bi}.
\end{proposition}

\begin{proof}
Applying $\chi_{\bm \si'} \ot \id$ to \eqref{quasiRB:intw} in the special case $\bm \si = \bm 0$, we deduce
\[
\mfX'_{\bm \ga,\bm \si} \; \Big(\big(\chi_{\bm \si'} \ot \id \big) \circ \big( \wb{\phantom{i} \cdot \phantom{i}}^B \ot \wb{\phantom{i} \cdot \phantom{i}}\big) \circ \Del\Big)(b) = \big( \chi_{\bm \si'}  \ot \id \big) \Del\big(\wb{b}^B\big) \; \mfX'_{\bm \ga,\bm \si} 
\]
for $b \in U_q\mfk_{\bm \ga,\bm 0}$.
By \eqref{phi:factorization1}-\eqref{phi:factorization2}, we obtain
\[
\mfX'_{\bm \ga,\bm \si} \; \wb{\phi_{\bm \si}(b)} = \phi_{\bm \si'}\big(\wb{b}^B\big) \; \mfX'_{\bm \ga,\bm \si}
\]
for any $b \in U_q\mfk_{\bm \ga,\bm 0}$. The result follows.
\end{proof}

By \eqref{quasiRB:decomposition}, we have
\[
\mfX'_{\bm \ga,\bm \si} = \sum_{\la \in \Qlat^+} \mfX'_{\bm \ga,\bm \si;\la} \qq \mbox{where}\qq \mfX'_{\bm \ga,\bm \si;\la} \in (U_q\mfn^+)_\la\,.
\]
with $\mfX'_{\bm \ga,\bm \si;0}=1$.
By Proposition \ref{prop:A:eqns} (ii), we deduce that $\mfX_{\bm \ga,\bm \si} = \mfX'_{\bm \ga,\bm \si}$. Finally, Theorem~\ref{thm:quasiK:2} follows. This concludes the proof of part (i) of Theorem ~\ref{thm:sum-quasi-k}.

\subsection{The coproduct formula for $\mfX$}\label{ss:coprod-formula}
We now address the proof of part (ii) of Theorem ~\ref{thm:sum-quasi-k}.
In order to establish a factorization of the coproduct of $\mfX$, we use the bilinear pairing again. 
We consider the subalgebra 
\eq{ \label{coproductsubalg}
	(U_q\mfb^+ \ot U_q\mfn^+)^{\mc O^+}_\Del \coloneqq \prod_{\la \in \Qlat^+} U_q\mfn^+ t_\la \ot (U_q\mfn^+)_\la \subset (U_q\mfb^+ \ot U_q\mfn^+)^{\mc O^+}.
}
Note that $\Del((U_q\mfn^+)^{\mc O^+}) \subset (U_q\mfb^+ \ot U_q\mfn^+)^{\mc O^+}_\Del$.
By \cite[Lem.~2.4]{BK19} we have, for all $X \in (U_q\mfb^+ \ot U_q\mfn^+)^{\mc O^+}_\Del$,
\eq{ \label{pairing:equality}
	\forall y,z \in U_q\mfn^- \; \langle y \ot z, X \rangle = 0 \qq \implies \qq X = 0.
}
The following result is the direct generalization of \cite[Theorem~9.4]{BK19} to the case of unrestricted parameters.

\begin{theorem} \label{thm:quasiK:Delta}
	We have
	\eq{ \label{quasiK:Delta}
		\Del(\mfX) = (\mfX \ot 1) \cdot (\Ad(\bm \ga') \circ \theta_q^{-1} \ot \id)(\Theta) \cdot \Ad(\ka_\id)(1 \ot \mfX) \cdot \Theta_X^{-1}.
	}
\end{theorem}

The proof, given in Section~\ref{ss:coprod-formula-pf}, relies on the properties of the 
auxiliary element $\Psi$, which we discuss in Sections~\ref{sec:Psi}-\ref{ss:properties-psi}.


\subsection{The auxiliary element $\Psi$} \label{sec:Psi}

In \cite[Secs. 8 and 9]{BK19} the coproduct of $\mfX$ is computed for the special case that $\indI$ is of finite type, $\bm \ga = \bm \ga'$ and $\bm \si = \bm \si'$.
We will now generalize this.

For $X \subseteq I$, recall Lusztig's quasi-R-matrix $\Theta_X \in (U_q\mfn^-_X \ot U_q\mfn^+_X)^{\mc O^+}$.
The key ingredient for the coproduct of $\mfX$ is the element
\eq{ \label{Thetadiagram:def}
\Psi \coloneqq \big( \Ad(\bm \ga') \circ \theta_q^{-1} \ot \id \big)(\Theta \Theta_X^{-1}) \in (U_q\mfg \ot U_q\mfg)^{\mc O^+}.
}

Note that we proved in Proposition \ref{prop:quasiR:quotient} that the element 
$\Theta\Theta_X^{-1}$ is supported on $\Qlat^+\setminus\Qlat^+_X$. More precisely,
\[
\Theta\Theta_X^{-1}\in \prod_{\la \in \Qlat^+ \backslash \Qlat_X^+} (U_q\mfn^-)_\la \ot (U_q\mfn^+)_\la
\]
We shall use this result in Lemma \ref{lem:Thetadiagram:properties}.

\begin{remark}
Let $\indI$ be of finite type and assume $\ga'_i = \ga_i$ for all $i\not\in X$; note that for $i \in X$ we automatically have $\ga'_i = \ga_i (=1)$.
We observe that the function $\xi: \Plate \to \F^\times$ defined by \cite[Equation~(8.1)]{BK19} is of the form $\xi(\mu) = G_{-\id-\theta,0}(\mu) \xi_\grp(\mu)$ with $\xi_\grp \in \Hom_\grp(\Plate,\F^\times)$ such that $\xi_\grp = G_{0,\rho-\rho_X} \ga'$ for some extension $\ga' \in \Hom_\grp(\Plate,\F^\times)$ of the group homomorphism: $\Qlat \to \F^\times$ defined by $\al_i \mapsto \ga'_i$. 
It follows that $\xi = G_{-\id-\theta,\rho-\rho_X} \ga'$.
Also, we have $\Ad(\wt T_I^{-1}) = \wt \om \circ \oi_I$.
By inspecting the list of generalized Satake diagrams \cite[Table I]{He84} we see that $\oi_I$ preserves $X$, so commutes with $\Ad(\wt T_X^{-1})$, and commutes with $\tau$.
Hence 
\[
\Ad(\wt T_I^{-1} \wt T_X^{-1}) \circ \tau \circ \oi_I = \wt \om \circ \oi_I \circ \Ad(\wt T_X^{-1}) \circ \tau \circ \oi_I = \wt \om^2 \circ \wt \theta_q^{-1} = \Ad(G_{\id + \theta, \rho_X - \rho}) \circ \theta_q^{-1}
\] 
so that $\Ad(\xi \wt T_I^{-1} \wt T_X^{-1}) \circ \tau \circ \oi_I  =  \Ad(\bm \ga') \circ \theta_q^{-1}$ and hence $\Psi$ coincides with the element defined by \cite[Equation (9.1)]{BK19} \hfill \rmkend
\end{remark}


\subsection{Properties of $\Psi$}\label{ss:properties-psi}

The element $\Psi$ satisfies the following properties, which generalize \cite[Lem. 9.1-9.2-9.3]{BK19}. 

\begin{lemma} \label{lem:Thetadiagram:properties}
We have
\begin{align}
\label{Thetadiagram:element}
\Psi &\in \prod_{\la \in \Qlat^+ \backslash \Qlat_X^+} (U_q\mfn^+)_{-\theta(\la)} t_\la \ot (U_q\mfn^+)_\la \\
\label{Thetadiagram:Di}
(\id \ot D_i^{(r)})(\Psi) &= (q_i-q_i^{-1}) \, \ze_i \, \Psi \, \Ad\big(\Theta_X \, \wt T_X^{-1} \ot 1 \big) (E_{\tau(i)} \ot 1) \cdot (t_i \ot 1).
\end{align}
\end{lemma}

\begin{proof}
We have
\[
\theta_q^{-1}((U_q\mfn^-)_{-\la}) = \Ad(G_{-\theta,\rho_X} \wt T_X^{-1})((U_q\mfn^+)_{\tau(\la)}) = \Ad(G_{-\theta,\rho_X})((U_q\mfn^+)_{-\theta(\la)}) = (U_q\mfn^+)_{-\theta(\la)} t_\la.
\]
Therefore, \eqref{Thetadiagram:element} follows from Proposition~\ref{prop:quasiR:quotient}.
From \eqref{quasiR:intw} and \eqref{quasiRX:intw}, for $i\in X$, $\Del(F_i)$ commutes with $\Theta \Theta_X^{-1}$.
Combined with \eqref{Di:commutatorFi} and the linear independence of $t_i$ and $t_i^{-1}$ over $U_q \mfn^+$, this yields \eqref{Thetadiagram:Di} for $i \in X$.
We need to prove the case $i\not\in X$. As before, we obtain
\eq{ \label{Thetadiagram:Di:1}
(\id \ot D_i^{(r)})(\Theta) = (q_i^{-1}-q_i) \, \Theta \, F_i \ot 1
} 
for any $i\in \indI$.
Since for $i \not\in X$, we have $(\id \ot D_i^{(r)})(\Theta_X) = 0$, the identities \eqref{Di:def} and \eqref{Thetadiagram:Di:1} yield
\begin{align} \label{Thetadiagram:Di:2}
	\begin{split}
(\id \ot D_i^{(r)})(\Theta \Theta_X^{-1}) &= (\id \ot D_i^{(r)})(\Theta) \; \Ad(1 \ot t_i)\big(\Theta_X^{-1}\big)\\  &= (q_i^{-1}-q_i) \; \Theta \; F_i \ot t_i \; \Theta_X^{-1} \; 1 \ot t_i^{-1}.
	\end{split}
\end{align}
Since $\Theta_X \in (U_q\mfn_X^- \ot U_q\mfn_X^+)^{\mc O^+}$, it is fixed by both 
$\Ad(\bm \ga') \ot 1$ and $\theta_q^{-1}\ot 1$. Moreover,
\begin{align*}
\big( \Ad(\bm \ga') \circ \theta_q^{-1} \big)(F_i) &= - \Ad(\bm \ga' G_{-\theta,\rho_X} \wt T_X^{-1})(E_{\tau(i)}) \\
&= - G_{-\theta,\rho_X}(-\theta(\al_i)) \ga'_{\tau(i)} \Ad(\wt T_X^{-1})(E_{\tau(i)}) t_i \\
&= - \ze_i \Ad(\wt T_X^{-1})(E_{\tau(i)}) t_i.
\end{align*}
Applying $\Ad(\bm \ga') \circ \theta_q^{-1} \ot \id$ to \eqref{Thetadiagram:Di:2}, we get
\eq{ \label{Thetadiagram:Di:3}
\begin{aligned}
(\id \ot D_i^{(r)})(\Psi) &= (q_i-q_i^{-1}) \ga_i \; (\Ad(\bm \ga') \circ \theta_q^{-1} \ot \id)(\Theta) \; \Ad(\wt T_X^{-1})(E_{\tau(i)}) t_i \ot t_i \; \Theta_X^{-1} \; 1 \ot t_i^{-1} \\
&= (q_i-q_i^{-1}) \ze_i \; \Psi \; \Ad(\Theta_X)\big(\Ad(\wt T_X^{-1})(E_{\tau(i)}) t_i \ot t_i\big) \; 1 \ot t_i^{-1}.
\end{aligned}
}
Thus, by \eqref{quasiRX:Cartan}, we obtain \eqref{Thetadiagram:Di} in the case $i \not\in X$.
\end{proof}

Note that
\[
\Ad(\ka_\id)(1 \ot \mfX_\la)|_{M_\mu \ot N_\nu} = q^{(\la,\mu)}1 \ot \mfX_\la|_{M_\mu \ot N_\nu} = t_\la \ot \mfX_\la|_{M_\mu \ot N_\nu}
\]
for $\la \in \Qlat^+$, $\mu,\nu \in \Plate$ and $M, N \in \mc O^+$.
Hence,
\eq{ \label{quasiKmod:formula}
\Ad(\ka_\id)(1 \ot \mfX) = \sum_{\la \in \Qlat^+} t_\la \ot \mfX_\la
}
and $\Ad(\ka_\id)(1 \ot \mfX) $  coincides with the element defined by \cite[Equation~(9.6)]{BK19}.

\begin{lemma} \label{lem:quasiKmod:properties}
We have
\begin{align}\label{quasiKmod:Di}
\begin{split}
 (\id \ot D_i^{(r)})&\big(\Ad(\ka_\id)(1 \ot \mfX)\big) =\\ = & (q_i-q_i^{-1}) \, \Ad(\ka_\id)(1 \ot \mfX) \, \Big(\ze_i \,  t_{-\theta(\al_i)} \ot \Ad(\wt T_X^{-1})(E_{\tau(i)}) -  \wb{\si_i} \, 1 \ot 1 \Big) t_i \ot 1
\end{split}
\end{align}
Moreover,
\begin{align}
\label{quasiKmod:ThetaX}
[\Ad(\ka_\id)(1 \ot \mfX),\Theta_X] = 0=[\Ad(\ka_\id)(1 \ot \mfX),\Ad(\wt T_X^{-1})(E_{\tau(i)}) t_i \ot t_i]\,.
\end{align}
\end{lemma}

\begin{proof}
Note that \eqref{quasiKmod:ThetaX} follows immediately from  \eqref{quasiKmod:formula} and Proposition~\ref{prop:quasiK:zerocomponents}.
Instead, by \eqref{quasiKmod:formula}, we have
\[
(\id \ot D_i^{(r)})\big(\Ad(\ka_\id)(1 \ot \mfX)\big) = \sum_{\la \in \Qlat^+} t_\la \ot D_i^{(r)}(\mfX_\la)
\]
Thus, by \eqref{quasiK:Bi:cpts:a}, it follows
\begin{align*}
& (\id \ot D_i^{(r)})\big(\Ad(\ka_\id)(1 \ot \mfX)\big) = \\
&= (q_i-q_i^{-1}) \sum_{\la \in \Qlat^+} t_\la \ot \big( \ze_i \, \mfX_{\la-\al_i+\theta(\al_i)} \, \Ad(\wt T_X^{-1})(E_{\tau(i)}) - \wb{\si_i} \, \mfX_{\la-\al_i} \big) \\
&= (q_i-q_i^{-1}) \Bigg( \ze_i \bigg( \sum_{\la \in \Qlat^+} t_{\la-\al_i+\theta(\al_i)} \ot \mfX_{\la-\al_i+\theta(\al_i)} \bigg) t_{\al_i-\theta(\al_i)} \ot \Ad(\wt T_X^{-1})(E_{\tau(i)}) + \\
& \hspace{25mm} - \wb{\si_i} \bigg( \sum_{\la \in \Qlat^+} t_{\la-\al_i} \ot \mfX_{\la-\al_i} \bigg) t_i \ot 1 \Bigg) \\
&= (q_i-q_i^{-1}) \Ad(\ka_\id)(1 \ot \mfX) \Big( \ze_i \, t_{\al_i-\theta(\al_i)} \ot \Ad(\wt T_X^{-1})(E_{\tau(i)}) -  \wb{\si_i} \, t_i \ot 1 \Big)
\end{align*}
The result follows.
\end{proof}


\subsection{Proof of Theorem~\ref{thm:quasiK:Delta}} \label{ss:coprod-formula-pf}
By \eqref{quasiRX:zeta1} and \eqref{quasiKmod:ThetaX}, the coproduct identity 
\eqref{quasiK:Delta} is equivalent to 
\eq{ \label{quasiK:Delta:2}
\Del(\mfX) = (\mfX \ot 1) \cdot \Psi \cdot \Ad(\ka_\id)(1 \ot \mfX).
}
By \eqref{Thetadiagram:element} and \eqref{quasiKmod:formula}, the right-hand side of \eqref{quasiK:Delta:2} belongs to $(U_q\mfb^+ \ot U_q\mfn^+)^{\mc O^+}_\Del$.
Since also $\Del(\mfX) \in (U_q\mfb^+ \ot U_q\mfn^+)^{\mc O^+}_\Del$, by \eqref{pairing:equality}, the coproduct identity \eqref{quasiK:Delta:2} is equivalent to 
\eq{ \label{quasiK:Delta:3}
\langle y \ot z, \Del(\mfX) \rangle = \langle y \ot z, (\mfX \ot 1) \cdot \Psi \cdot \Ad(\ka_\id)(1 \ot \mfX) \rangle
}
for $y,z \in U_q\mfn^-$.
By linearity it suffices to consider the case $z = F_{i_1} F_{i_2} \cdots F_{i_\ell}$ for all $(i_1,\ldots,i_\ell) \in \indI^\ell$, $\ell \in \Z_{\geqslant 0}$.
We do this by induction on $\ell$.\\ 

Consider the case $\ell=0$. Denote by $P^+_0$ the projection from $(U_q\mfb^+ \ot U_q\mfn^-)^{\mc O^+}_\Del$ to the direct summand $U_q\mfn^+ \ot \F$.
By \eqref{Thetadiagram:element}, $\Ad(\bm \ga') \circ \theta_q^{-1} \ot \id $ maps $\sum_{\la \in \Qlat, \mu \in \Qlat_X \atop \la + \mu = \nu} \Theta_{I,\la}\wb{\Theta_{X,\mu}}$ into $(U_q\mfn^+)_{-\theta(\nu)} t_\nu \ot (U_q\mfn^+)_\nu$ for $\nu \in \Qlat^+$, so that $P^+_0(\Psi) = 1 \ot 1$.
Also, by \eqref{quasiKmod:formula} we have
\[
P^+_0\big(\Ad(\ka_\id)(1 \ot \mfX)\big) = \sum_{\la \in \Qlat^+} P^+_0\big( t_\la \ot \mfX_\la \big) = 1 \ot 1.
\]
Therefore, we obtain
\[
\langle y \ot 1, (\mfX \ot 1) \cdot \Psi \cdot \Ad(\ka_\id)(1 \ot \mfX) \rangle = \langle y \ot 1,\mfX \ot 1 \rangle = \langle y,\mfX \rangle = \langle y \ot 1, \Del(\mfX) \rangle\, .
\]
Assume \eqref{quasiK:Delta:3} is satisfied for all $y \in U_q\mfn^+$ and all monomials 
$z = F_{i_1} F_{i_2} \cdots F_{i_\ell}$ with $\ell \geqslant 0$.
It remains to prove that
\eq{ \label{quasiK:Delta:4}
\langle y \ot zF_i, \Del(\mfX) \rangle = \langle y \ot z F_i, (\mfX \ot 1) \cdot \Psi \cdot \Ad(\ka_\id)(1 \ot \mfX) \rangle
}
for any $i\in \indI$.
By \eqref{pairing:rel1} and \eqref{Di:pairing}, we have
\[
\langle y \ot z F_i,\Del(\mfX)\rangle = \langle y z F_i, \mfX \rangle = (q_i^{-1}-q_i)^{-1} \langle y z, D_i^{(r)}(\mfX) \rangle.
\]
Thus, by \eqref{quasiK:Bi:3a}, we get
\begin{align*}
\langle y \ot z F_i,\Del(\mfX)\rangle 
&= \Big\langle y z, \mfX\big(\wb{\si_i} - \ze_i \Ad(\wt T_X^{-1})(E_{\tau(i)}) \big) \Big\rangle \\
&= \Big\langle y \ot z, \Del(\mfX)\big(\wb{\si_i} 1 \ot 1 - \ze_i \Del\big(\Ad(\wt T_X^{-1})(E_{\tau(i)})\big) \big) \Big\rangle.
\end{align*}
By induction, the LHS of \eqref{quasiK:Delta:4} gives
\[
\langle y \ot z F_i,\Del(\mfX)\rangle = \Big\langle y \ot z, (\mfX \ot 1) \cdot \Psi \cdot \Ad(\ka_{\id})(1 \ot \mfX) \big(\wb{\si_i} 1 \ot 1 - \ze_i \Del\big(\Ad(\wt T_X^{-1})(E_{\tau(i)})\big) \big) \Big\rangle.
\]
while the RHS of \eqref{quasiK:Delta:4}, by \eqref{pairing:rel3}, \eqref{Di:pairing}, and \eqref{Di:def}, gives
\begin{align*}
& \langle y \ot z F_i, (\mfX \ot 1) \cdot \Psi \cdot \Ad(\ka_\id)(1 \ot \mfX) \rangle = \\
& \qq = \langle y \ot z F_i, (\mfX \ot 1) \cdot \Psi \cdot \Ad(\ka_\id)(1 \ot \mfX) t_i^{-1} \ot 1 \rangle = \\
& \qq = (q_i^{-1}-q_i)^{-1} \Big\langle y \ot z, (\id \ot D_i^{(r)})\big(\mfX \ot 1 \; \Psi \; \Ad(\ka_\id)(1 \ot \mfX) \big) t_i^{-1} \ot 1 \Big\rangle \\
& \qq = (q_i^{-1}-q_i)^{-1} \Big\langle y \ot z, (\mfX \ot 1) \; (\id \ot D_i^{(r)})\big( \Psi \; \Ad(\ka_\id)(1 \ot \mfX) \big) t_i^{-1} \ot 1 \Big\rangle.
\end{align*}
Therefore, the desired identity \eqref{quasiK:Delta:4} reduces to
\eq{ \label{quasiK:Delta:5}
\begin{aligned}
& \Psi \; \Ad(\ka_{\id})(1 \ot \mfX) \; \big(\wb{\si_i} 1 \ot 1 - \ze_i \Del\big(\Ad(\wt T_X^{-1})(E_{\tau(i)})\big) = \\
& \qq = (q_i^{-1}-q_i)^{-1} (\id \ot D_i^{(r)})\big( \Psi \; \Ad(\ka_\id)(1 \ot \mfX) \big)  t_i^{-1} \ot 1.
\end{aligned}
}
By \eqref{Thetadiagram:Di} and \eqref{quasiKmod:Di},
if $i \in X$, then \eqref{quasiK:Delta:5} is satisfied since $\ze_i = 0=\si_i$.
If $i \not\in X$, then by \eqref{quasiKmod:ThetaX}, we have
\begin{align*}
& (q_i^{-1}-q_i)^{-1} (\id \ot D_i^{(r)})(\Psi) \Ad((1 \ot t_i)\ka_\id)(1 \ot \mfX) \, t_i^{-1} \ot 1 = \\
&\qq = - \ze_i \Psi \; \Ad\big(\Theta_X(\wt T_X^{-1} \ot 1)\big) (E_{\tau(i)} \ot 1) \Ad((t_i \ot t_i)\ka_\id)(1 \ot \mfX)  \\
&\qq = - \ze_i \Psi \; \Theta_X \; \Ad(\wt T_X^{-1})(E_{\tau(i)}) \ot 1 \; \Ad((t_i \ot t_i)\ka_\id)(1 \ot \mfX) \Theta_X^{-1} \\
&\qq = - \ze_i \Psi \; \Theta_X \Ad(\ka_\id)(1 \ot \mfX) \; \big( \Ad(\wt T_X^{-1})(E_{\tau(i)}) \ot 1 \big) \; \Theta_X^{-1} \\
&\qq = - \ze_i \Psi \; \Ad(\ka_\id)(1 \ot \mfX) \Ad\big(\Theta_X(\wt T_X^{-1} \ot 1)\big)(E_{\tau(i)} \ot 1).
\end{align*}
Hence, \eqref{Di:def} implies
\begin{align*}
&  (q_i-q_i^{-1})^{-1} (\id \ot D_i^{(r)})\big( \Psi \; \Ad(\ka_\id)(1 \ot \mfX) \big)  t_i^{-1} \ot 1 = \\
& = \Psi \; \Ad(\ka_\id)(1 \ot \mfX)  \Big(\ze_i \Ad\big(\Theta_X(\wt T_X^{-1} \ot 1)\big)(E_{\tau(i)} \ot 1) + \ze_i t_{-\theta(\al_i)} \ot \Ad(\wt T_X^{-1})(E_{\tau(i)}) - \wb{\si_i} 1 \ot 1 \Big) .
\end{align*}
Therefore, \eqref{quasiK:Delta:5} further reduces to
\eq{ \label{quasiK:Delta:6}
\Del\big(\Ad(\wt T_X^{-1})(E_{\tau(i)}) = \Ad\big(\Theta_X(\wt T_X^{-1} \ot 1)\big)(E_{\tau(i)} \ot 1)  +  t_{-\theta(\al_i)} \ot \Ad(\wt T_X^{-1})(E_{\tau(i)}).
}
Let $j \in X$.
Applying $\Ad(\wt T_X^{-1}) \circ \tau$ to $E_i F_j t_j = q^{-(\al_i,\al_j)} F_j t_j E_i$, we obtain $\Ad(\wt T_X^{-1})(E_{\tau(i)}) E_j = q^{-(\al_i,\al_j)} E_j \Ad(\wt T_X^{-1})(E_{\tau(i)})$ and
\begin{align*}
(t_{-\theta(\al_i)} \ot \Ad(\wt T_X^{-1})(E_{\tau(i)})) \; (F_j \ot E_j) 
&= q^{(\theta(\al_i)-\al_i,\al_j)} (F_j \ot E_j) \; (t_{-\theta(\al_i)} \ot \Ad(\wt T_X^{-1})(E_{\tau(i)})) \\
&= (F_j \ot E_j) \; (t_{-\theta(\al_i)} \ot \Ad(\wt T_X^{-1})(E_{\tau(i)})).
\end{align*}
Hence, $t_{-\theta(\al_i)} \ot \Ad(\wt T_X^{-1})(E_{\tau(i)}) $ commutes with $\Theta_X$ and, by \eqref{quasiRX:factorized:1}, we conclude that
\begin{align*}
\Del\big(\Ad(\wt T_X^{-1})(E_{\tau(i)})\big) &= \Ad(\Theta_X(\wt T_X^{-1} \ot \wt T_X^{-1}))(E_{\tau(i)} \ot 1 + t_{\tau(i)} \ot E_{\tau(i)}) \\
&= \Ad(\Theta_X)\big( \Ad(\wt T_X^{-1})(E_{\tau(i)}) \ot 1 + t_{-\theta(\al_i)} \ot \Ad(\wt T_X^{-1})(E_{\tau(i)}) \big) \\
&= \Ad\big(\Theta_X(\wt T_X^{-1} \ot 1)\big) (E_{\tau(i)} \ot 1) + t_{-\theta(\al_i)} \ot \Ad(\wt T_X^{-1})(E_{\tau(i)})\,.
\end{align*}
The result follows.


\section{Universal K-matrices} \label{s:universalk} 

In this section, we introduce the {\em standard} universal K-matrix and derive its key properties.
A further modification in terms of a multiplicative difference of two modified diagrammatic half-balances, corresponding to a pair of generalized Satake diagrams, yields a rich theory of new {\em modified} universal K-matrices. 
Among those, in special cases, certain choices are more convenient or natural than others.
In particular, when the two diagrams coincide, this yields a natural
interpretation of the quasi-K-matrix as a universal K-matrix.
For quantum groups of finite type, this recovers the Balagovi\'c-Kolb universal K-matrix and their formalism (cf.~Section \ref{sec:GSatmin:finite}). 
In Section \ref{s:affine}, we shall outline the applications of this approach in the theory of
quantum affine algebras.\\ 

Throughout the section, we fix $(X,\tau) \in \GSat(A)$, $(\bm \ga,\bm \si) \in \Ga_q \times \Si_q$, we assume that $\Plate$ is $\tau$-compatible, and we consider the associated quantum pseudo-fixed-point subalgebra $U_q\mfk_{\bm \ga,\bm \si}\subset U_q\mfg$.


\subsection{The inverse of the quasi-K-matrix}
It is convenient for us to work with the {\em inverse} of the quasi-K-matrix constructed in Section \ref{s:quasiK} (cf.\ Theorem \eqref{thm:quasiK:2}). Thus, we set
\eq{ \label{tildek:def}
\wtKM = \wtKM_{\bm \ga,\bm \si}\coloneqq \wb{\mfX_{\bm \ga,\bm \si}} = \mfX_{\bm \ga',\bm \si'}^{-1},
}
where the parameters $(\bm \ga',\bm \si')$ are defined in \eqref{prime:def}.
Recall the injective algebra homomorphism $f:  U_q\mfk_{\bm \ga,\bm \si}\to U_q\mfg$ defined by $f(u)=u$ if $u \in U_q\mfg_X U_q\mfh^\theta$ and $f(B_{i;\ga_i,\si_i}) = \wb{B_{i;\ga'_i,\si'_i}}$ for all $i \not\in X$.
We have the following
\begin{lemma} \label{prop:tildek}
The operator $\wtKM\in(U_q\mfn^+)^{\mc O^+}$ is the unique element 
with $\wtKM_0=1$ satisfying the intertwining equation
\eq{ \label{tildek:intw}
\wtKM u = f(u) \wtKM
}
for any $u\in U_q\mfk_{\bm \ga,\bm \si}$.
Moreover, it satisfies the coproduct identity
\eq{ \label{tildek:Delta}
\Del(\wtKM) = {R}_{X,\tau}^{-1} \cdot (1 \ot \wtKM) \cdot (\Ad(\bm \ga) \circ \theta_q^{-1} \ot \id)({R}) \cdot (\wtKM \ot 1).
}
\end{lemma}

\begin{proof}
By Theorem \ref{thm:quasiK:2}, it follows that $\wtKM$ satisfies \eqref{tildek:intw}.
Then, from \eqref{quasiK:Delta}, we have
\begin{align*}
\Del(\wtKM) &= \wtR_X^{-1} \; \Ad(\ka_\id)(1 \ot \wtKM) \; (\Ad(\bm \ga) \circ  \theta_q^{-1} \ot \id)(\wtR) \; (\wtKM \ot 1) \\
&= \wtR_X^{-1} \; \Ad(\ka_{-\theta})(1 \ot \wtKM) \; (\Ad(\bm \ga) \circ  \theta_q^{-1} \ot \id)(\wtR) \; (\wtKM \ot 1) \\
&= {R}_{X,\tau}^{-1} \; (1 \ot \wtKM) \; (\Ad(\bm \ga) \circ  \theta_q^{-1} \ot \id)({R}) \; (\wtKM \ot 1).
\end{align*}
where the first equality follows from the identity $\Ad(\ka_{\theta+\id})(1\ot x) = 1 \ot x$ for any $x \in (U_q\mfn^+)_\la$ and $\la \in \Qlat^{-\theta}$, while the third equality follows from $( \theta_q^{-1} \ot \id)(\ka_\id) = \ka_\theta$ (cf.~Remark~\ref{rmk:theta-extend-to-completion}).
\end{proof}

\begin{remark}
The choice of $\ka_{\theta}$ in the definition of $R_{X,\tau}$ (cf.~Definition \ref{RX:def}) is 
therefore instrumental to obtain \eqref{tildek:Delta}, in that it absorbs the Cartan corrections that
naturally arise in the coproduct identity of the quasi-K-matrix. \hfill \rmkend
\end{remark}


\subsection{Quantum pseudo-involutions on $ U_q\mfk_{\bm \ga,\bm \si}$}
In the following, we shall use the quantum pseudo-involution $\theta_q$ as defining a new
intertwining equation. To this end, we need to describe its action on
$U_q\mfk_{\bm \ga,\bm \si}$. Recall that we regard $\bm \ga$ as an element in
$(U_q\mfh)^{\mc O^+_\int}$.  We have the following

\begin{proposition} \label{prop:ups:intw}
	For any $u\in U_q\mfk_{\bm \ga,\bm \si}$, it holds $\theta_q^{-1}(u)=\Ad(\bm \ga^{-1})(f(u))$.
\end{proposition}

\begin{proof}
	For $u \in U_q\mfg_X U_q\mfh^\theta$, we have $\Ad(\bm \ga^{-1})(u) = f(u) = u =  \theta_q^{-1}(u)$. It remains to prove that, for $i \not\in X$,
	\eq{
		\label{ups:intw:Bi} \big(  \theta_q \circ \Ad(\bm \ga^{-1})\big)(\wb{B_{i;\ga'_i,\si'_i}}) = B_{i;\ga_i,\si_i}.
	}
	Note that
	\eq{ \label{ups:intw:Fi}
		\big( \theta_q \circ \Ad(\bm \ga^{-1})\big)(F_i) = \ga_i  \theta_q(F_i).
	}
	Moreover, by \eqref{prime:def} and \eqref{hatT:bar}, we have
	\[
	\wb{\ga'_i  \theta_q(F_i)} = - (-1)^{\al_i(2\rho^\vee_X)} \ga_{\tau(i)} \wb{\Ad(T_{X,\tau})(E_{\tau(i)})} 
	= - \ga_{\tau(i)} \Ad(T_{X,\tau}^{-1})(E_{\tau(i)}) = \ga_{\tau(i)}  \theta_q^{-1}(F_i)\,.
	\]
	Therefore,
	\eq{ \label{ups:intw:thetaqFi}
		\big( \theta_q \circ \Ad(\bm \ga^{-1})\big) \big( \wb{ \ga'_i \theta_q(F_i)} \big) = \ga_{\tau(i)}\big( \theta_q \circ \Ad(\bm \ga^{-1})\big)( \theta_q^{-1}(F_i)) = F_i.
	}
	Finally, 
	\eq{ \label{ups:intw:Cartan}
		\big( \theta_q \circ \Ad(\bm \ga^{-1})\big)(\si_i t_i) = \si_i  \theta_q(t_i) = \si_i t_{\theta(\al_i)} = \si_i t_i^{-1}
	}
	since $\si_i =0$ if $\theta(\al_i) \ne -\al_i$.
	Combining \eqref{ups:intw:Fi}-\eqref{ups:intw:Cartan}, we obtain \eqref{ups:intw:Bi}.
\end{proof}


\subsection{The standard universal K-matrix $\KM{X,\tau}$}\label{sec:basic-k}
We introduce a subtle correction of the operator $\wtKM$, which reveals crucial in the
following.

\begin{definition}\label{ss:basic-k}
The \emph{standard} universal K-matrix is the operator in $(U_q\mfb^+)^{\mc O^+}$ given by
\begin{align}
\KM{X,\tau} \coloneqq \bm \ga^{-1} \wtKM 
\end{align}
where $\wtKM$ is defined in \eqref{tildek:def}.
The \emph{standard} twisting operator is the algebra automorphism of $U_q\mfg$ given by
$\psi_{X,\tau}\coloneqq  \theta_q^{-1}$.
\hfill\defnend 
\end{definition}
We prove the first main result of the paper.
\begin{theorem} \label{thm:kX}
The standard universal K-matrix $ \KM{X,\tau}$
satisfies the intertwining equation 
\begin{align}
\label{kXtau:intw} \KM{X,\tau} u &= \psi_{X,\tau}(u) \KM{X,\tau}
\end{align}
for any $u\in U_q\mfk_{\bm \ga,\bm \si}$ and the coproduct identity
\begin{align}
\label{kXtau:Delta} \Del(\KM{X,\tau}) &= {R}_{X,\tau}^{-1} \cdot (1 \ot \KM{X,\tau}) \cdot \big( \psi_{X,\tau} \ot \id \big)({R}) \cdot (\KM{X,\tau} \ot 1)\,.
\end{align}
\end{theorem}

\begin{proof}[Proof of Theorem \ref{thm:kX}.]
Combining \eqref{tildek:Bi} with Proposition~\ref{prop:ups:intw}, we immediately obtain \eqref{kXtau:intw}.
To prove \eqref{kXtau:Delta}, note that $\Del(\bm \ga) = \bm \ga \ot \bm \ga$ and \eqref{tildek:Delta} imply
\[
\Del(\KM{X,\tau}) = (\bm \ga^{-1} \ot \bm \ga^{-1}) \; {R}_{X,\tau}^{-1} \; (1 \ot \wtKM) \; \big(\Ad(\bm \ga) \circ  \theta_q(X,\tau)^{-1} \ot \id \big)({R}) \; (\wtKM \ot 1).
\]
By \eqref{quasiRX:Cartan0}, $\bm \ga \ot \bm \ga$ commutes with $R_{X,\tau} = \ka_{\theta(X,\tau)} \Theta_X^{-1}$, which completes the proof.
\end{proof}


\subsection{The universal K-matrices $\KM{Y,\eta}$} \label{sec:parabolick}

The twisting operator $\psi_{X,\tau}=\theta_q^{-1}$ is in general quite a complicated automorphism,
whose pullback functor is not easily described. It is therefore convenient to introduce a further
modification of the pair $(\KM{X,\tau}, \psi_{X,\tau})$ in terms of an {\em auxiliary} generalized Satake diagram $(Y,\eta)$ which yields a  simpler twisting operator. This however requires to restrict to integrable category $\mc O^+$ modules.

\begin{definition}\label{defn:mod-k}
For any $(Y,\eta) \in \GSat(A)$ such that $\Plate$ is $\eta$-compatible, we consider the operator in $(U_q\mfg)^{\mc O^+_\int}$ given by
\begin{align}\label{kYeta:def}
	\KM{Y,\eta} \coloneqq  (T_{Y,\eta}^{-1} \, T_{X,\tau})\cdot\KM{X,\tau} = (T_{Y,\eta}^{-1} \, T_{X,\tau})\cdot \bm \ga^{-1}\cdot \wtKM
\end{align}
and the algebra automorphism of $U_q\mfg$ given by
\begin{align}\label{psiYeta:def}
\psi_{Y,\eta}\coloneqq  \Ad( T_{Y,\eta}^{-1} T_{X,\tau} ) \circ \psi_{X,\tau}=\theta_q(Y,\eta)^{-1}
\circ\eta\circ\tau =\Ad(T_{Y,\eta})^{-1}\circ\omega\circ\tau\,,
\end{align}
where $\theta_q(Y,\eta)$ denotes the quantum pseudo-involution associated to $(Y,\eta)$.
\hfill\defnend
\end{definition}

Note that, in the case $(Y,\eta)=(X,\tau)$, Definitions \ref{ss:basic-k} and \ref{defn:mod-k} 
yield the same operators and there is no clash of notations.\\

Our next main result is that the element $\KM{Y,\eta}$ is a universal K-matrix for $U_q\mfk_{\bm \ga,\bm \si}$ with respect to the twisting operator $\psi_{Y,\eta}$.
More precisely, following Definition \ref{def:annular-bialg-twist-pair}, we shall prove 
that $(\psi_{Y,\eta}, R_{Y, \eta})$ is a twist pair, $(U_q\mfg, R, \psi_{Y,\eta}, R_{Y, \eta}, \KM{Y,\eta})$ is a cylindrical bialgebra, and $U_q\mfk_{\bm \ga,\bm \si}$ is a cylindrically invariant coideal subalgebra. 
These results will be proved in Proposition~\ref{thm:twistpair} and Theorem~\ref{thm:kYeta:int-Delta}, respectively.


\subsection{The twist pair $(\psi_{Y,\eta},R_{Y,\eta})$}

We first prove that $(\psi_{Y,\eta},R_{Y,\eta})$ is a twist pair for the quasitriangular 
bialgebra $(U_q\mfg,\Del,\eps,R)$. This amounts to proving that $\psi_{Y,\eta}$ is an isomorphism of quasitriangular bialgebras
\eq{ \label{psi:iso:braidedbialgebras}
	\psi_{Y,\eta}\colon (U_q\mfg , \Del^\op, \eps, R_{21})\; {\longrightarrow} \; (U_q\mfg, \Ad(R_{Y,\eta}) \circ \Del, \eps, (R_{Y,\eta})_{21}\cdot R\cdot R_{Y,\eta}^{-1})\,.
}

\begin{proposition} \label{thm:twistpair}
	The following relations hold: 
	\begin{align}\label{psi:Del}
		(\psi_{Y,\eta} \ot \psi_{Y,\eta}) \, \circ \, \Del^\op= \Ad(R_{Y,\eta}) \circ \Del \circ \psi_{Y,\eta}\,,
		\qq \eps \circ \psi_{Y,\eta} = \eps\,,
	\end{align}
	and
	\begin{align}\label{psi:R}
	 (\psi_{Y,\eta} \ot \psi_{Y,\eta})(R_{21})= ({R}_{Y,\eta})_{21} \cdot  R \cdot  {R}_{Y,\eta}^{-1}. 
	\end{align}
\end{proposition}

\begin{proof}
	Since the diagram automorphisms are bialgebra automorphisms and $\om$ is a coalgebra antiautomorphism satisfying \eqref{R:symmetries}, the identities \eqref{psi:Del}-\eqref{psi:R}
	reduce to
	\begin{align}
	\label{T:Del} \Ad(T_{Y,\eta}^{-1} \ot T_{Y,\eta}^{-1}) \, \circ \, \Del &= \Ad(R_{Y,\eta}) \circ \Del \circ \Ad(T_{Y,\eta}^{-1}), \\
	\label{T:eps} \eps \circ \Ad(T_{Y,\eta}) &= \eps, \\
	\label{T:R} \Ad(T_{Y,\eta}^{-1} \ot T_{Y,\eta}^{-1})(R) &= ({R}_{Y,\eta})_{21} \cdot  R \cdot  {R}_{Y,\eta}^{-1}.
	\end{align}
The identity \eqref{T:eps} follows from \eqref{hatT:rootspace} 
and the fact that $\Ad(T_{Y,\eta})(t_h) = t_{w_Y(h)}$ for $h \in \Qlat^\vee_\ext$.
The identities \eqref{T:Del} and \eqref{T:R} follow from \eqref{RX:factorized}.
\end{proof}


\subsection{Properties of $\KM{Y,\eta}$} \label{sec:intw}
We prove that the operator $\KM{Y,\eta}$ is indeed a universal K-matrix with respect
to the twist pair $(\psi_{Y,\eta}, R_{Y,\eta})$.

\begin{theorem} \label{thm:kYeta:int-Delta}
The operator $\KM{Y,\eta}$ satisfies the intertwining equation
\eq{ 
\label{kYeta:intw}{\KM{Y,\eta}}\cdot u = \psi_{Y,\eta}(u)\cdot {\KM{Y,\eta}}\,,
}
for any $u\in U_q\mfk_{\bm \ga,\bm \si}$, and the coproduct identity
\eq{
	\label{kYeta:Delta} \Del(\KM{Y,\eta}) = {R}_{Y,\eta}^{-1} \; (1 \ot{\KM{Y,\eta}}) \; (\psi_{Y,\eta} \ot \id)({R}) \; (\KM{Y,\eta} \ot 1).
}
\end{theorem}

\begin{proof}
From \eqref{kXtau:intw}, one has
\[
\KM{Y,\eta}\cdot u = T_{Y,\eta}^{-1} \, T_{X,\tau} \,  \theta_q^{-1}(u) \, \KM{X,\tau} = \big( \Ad(T_{Y,\eta}^{-1}) \circ \om \circ \tau \big)(u) \, T_{Y,\eta}^{-1} \, T_{X,\tau} \, \KM{X,\tau} = \psi_{Y,\eta}(u) \,{\KM{Y,\eta}}\,,
\]
for any $u \in U_q\mfk_{\bm \ga,\bm \si}$. Then, from \eqref{kXtau:Delta}, one has
	\begin{align*}
	\Del(\KM{Y,\eta}) &= \Del\big( T_{Y,\eta}^{-1}T_{X,\tau}^{} \big) {R}_{X,\tau}^{-1} \; (1 \ot \KM{X,\tau}) \; \big(  \theta_q(X,\tau)^{-1} \ot \id \big)({R}) \; (\KM{X,\tau} \ot 1) \\
	&= \Del(T_{Y,\eta})^{-1} \; (T_{X,\tau} \ot T_{X,\tau}) \; (1 \ot \KM{X,\tau}) \; \big(  \theta_q(X,\tau)^{-1} \ot \id \big)({R}) \; (\KM{X,\tau} \ot 1) \\
	&= \Del(T_{Y,\eta})^{-1} \; (1 \ot  T_{X,\tau} \KM{X,\tau}) \; \big( \Ad( T_{X,\tau} ) \circ  \theta_q(X,\tau)^{-1} \ot \id \big)({R}) \; (T_{X,\tau} \KM{X,\tau} \ot 1) \\
	&= R_{Y,\eta}^{-1} \; (T_{Y,\eta}^{-1} \ot T_{Y,\eta}^{-1}) \; (1 \ot  T_{X,\tau} \KM{X,\tau}) \; \big( \Ad( T_{X,\tau} ) \circ  \theta_q(X,\tau)^{-1} \ot \id \big)({R}) \;  (T_{X,\tau} \KM{X,\tau} \ot 1) \\
	&= R_{Y,\eta}^{-1} \; (1 \ot{\KM{Y,\eta}}) \; \big( \psi_{Y,\eta} \ot \id \big)({R}) \; (\KM{Y,\eta} \ot 1),
	\end{align*}
	where the second and fourth equalities follows from \eqref{RX:factorized}.
\end{proof}


\subsection{Generalized reflection equation}

The following result is the analogue of Proposition~\ref{prop:twistedRE}.

\begin{theorem} \label{thm:kYeta:RE}
The operator $\KM{Y,\eta}$ satisfies the generalized reflection equation \eqref{eq:psi-Q-re}  
with respect to the twisting operator $\psi_{Y,\eta}$, \ie
\begin{align}\label{kYeta:RE}
\begin{split}
(\psi_{Y,\eta} \ot \psi_{Y,\eta})(R_{21}) \cdot (1 \ot{\KM{Y,\eta}}) \cdot&(\psi_{Y,\eta} \ot \id)(R) \cdot (\KM{Y,\eta} \ot 1)=\\=&
(\KM{Y,\eta} \ot 1) \cdot (\id \ot \psi_{Y,\eta})(R_{21}) \cdot (1 \ot{\KM{Y,\eta}}) \cdot R\,.
\end{split}
\end{align}
\end{theorem}

\begin{proof}
From the coproduct formula \eqref{kYeta:Delta}, 
one has
\begin{align*}
\Del^\op(\KM{Y,\eta}) =({R}_{Y,\eta})_{21}^{-1} \; (\KM{Y,\eta} \ot 1) \; (\id \ot \psi_{Y,\eta})({R}_{21}) \; (1 \ot{\KM{Y,\eta}})
\end{align*}
and, from \eqref{R:intw},
\begin{align*}
\Del^\op(\KM{Y,\eta}) =  R\cdot\Del(\KM{Y,\eta})\cdot R^{-1}
= R \;  {R}_{Y,\eta}^{-1} \; (1 \ot{\KM{Y,\eta}}) \; (\psi_{Y,\eta} \ot \id)({R}) \; (\KM{Y,\eta} \ot 1) \; R^{-1}.
\end{align*}
Therefore,
\begin{align}
(\KM{Y,\eta} \ot 1) \cdot (\id \ot \psi_{Y,\eta})&({R}_{21}) \cdot (1 \ot{\KM{Y,\eta}}) \cdot R =\\&= ({R}_{Y,\eta})_{21} \;  R \cdot  {R}_{Y,\eta}^{-1} \cdot (1 \ot{\KM{Y,\eta}}) \cdot (\psi_{Y,\eta} \ot \id)({R}) \; (\KM{Y,\eta} \ot 1).
\end{align}
Thus, the result follows from \eqref{psi:R}.
\end{proof}


\subsection{Cartan corrections of universal K-matrices}
\label{sec:K:Cartanmod}
We describe a further modification of the standard K-matrix $\KM{X,\tau}$
associated to the elements $g \in (U_q\mfh)^{\mc W,\times}$ such that $\Ad(g)$
preserves $U_q\mfg\subset(U_q\mfg)^{\mc W_\int}$. Namely, we set
\begin{align}
	\label{kG:def}
	\gKM{Y,\eta}{g} \coloneqq g \cdot{\KM{Y,\eta}} 
	\qq\mbox{and}\qq
	\psi_{Y,\eta}^{g} \coloneqq \Ad(g) \circ \psi_{Y,\eta}\,. 
\end{align}

The operator $\gKM{Y,\eta}{g}$ remains a universal K-matrix.

\begin{theorem} \label{thm:kG}
	The operator $\gKM{Y,\eta}{g}$ satisfies the intertwining equation
	\begin{align}
		\label{kG:intw}
		\gKM{Y,\eta}{g} u &= \psi_{Y,\eta}^{g}(u){\KM{Y,\eta}}^{g}\,, 
	\end{align}
	for any $u \in U_q\mfk_{\bm \ga,\bm \si}$, the coproduct identity
	\begin{align}
		\label{kG:Delta}
		\Del(\gKM{Y,\eta}{g}) &= (R_{Y,\eta}^{g})^{-1}\cdot (1 \ot{\KM{Y,\eta}}^{g}) \cdot (\psi_{Y,\eta}^{g} \ot \id)(R) \cdot (\gKM{Y,\eta}{g} \ot 1)\,,
	\end{align}
	where $(R_{Y,\eta}^{g})^{-1}\coloneqq  \Del(g) \, R_{Y,\eta}^{-1} \, (g^{-1} \ot g^{-1})$,
	and the generalized reflection equation
	\eq{
		\label{kG:RE} 
		\begin{aligned}
			& (\gKM{Y,\eta}{g} \ot 1) \cdot (\id \ot \psi_{Y,\eta}^{g})({R}_{21}) \cdot (1 \ot{\KM{Y,\eta}}^{g}) \cdot {R} = \\
			&\qq \qq = (\psi_{Y,\eta}^{g} \ot \psi_{Y,\eta}^{g})({R}_{21}) \cdot (1 \ot{\KM{Y,\eta}}^{g}) \cdot (\psi_{Y,\eta}^{g} \ot \id)({R}) \cdot (\gKM{Y,\eta}{g} \ot 1).
		\end{aligned}
	}
\end{theorem}

\begin{proof}
	The identities \eqref{kG:intw}, \eqref{kG:Delta} and \eqref{kG:RE} follow by multiplying \eqref{kYeta:intw}, \eqref{kYeta:Delta} and \eqref{kYeta:RE} by $g$, $\Del(g)$ and $g \ot g$, respectively.
\end{proof}

\begin{remark}\label{rmk:gauge}
	Clearly, $g$ can be thought of as a gauge transformation {\em acting} on the cylindrical structure
	$(\psi_{Y,\eta}, R_{Y,\eta},\KM{Y,\eta})$, \ie we have $g\star(\psi_{Y,\eta}, R_{Y,\eta},\KM{Y,\eta})=(\psi_{Y,\eta}^g, R_{Y,\eta}^g,\KM{Y,\eta}^g)$. 
	In fact, the same result applies for any $g\in(U_q\mfg)^{\mc O^+_\int}$ such that $\Ad(g)$ preserves $U_q\mfg$ (cf.~\cite[Sec.~3]{AV22}). In particular, we have
	$(T_{Y,\eta}^{-1} T_{X,\tau})\star(\psi_{X,\tau}, R_{X,\tau},\KM{X,\tau})
	=(\psi_{Y,\eta}, R_{Y,\eta},\KM{Y,\eta})$.
\hfill \rmkend
\end{remark}


\subsection{An alternative choice} \label{sec:alternative}
From Remark~\ref{rmk:gauge}, it is clear that the universal K-matrix $\KM{X,\tau}$ and the twisting operator $\psi_{X,\tau}$ give rise to large family of universal K-matrices depending upon the choice
of a gauge transformation with some mild restrictions. In particular, we can recover immediately the
slightly different setup used in \cite{BK19}. Namely, instead of considering $\KM{Y,\eta}$ and $\psi_{Y,\eta}$ defined by \eqref{kYeta:def}-\eqref{psiYeta:def}, we can set 
\eq{ \label{kpsi:alt}
\gKM{Y,\eta}{\prime} \coloneqq (T_{Y,\eta} T_{X,\tau})\cdot \KM{X,\tau} \qq
\mbox{and}\qq 
\psi_{Y,\eta}^\prime \coloneqq \Ad(T_{Y,\eta} T_{X,\tau}) \circ \psi_{X,\tau}.
}
Proceeding as before, one verifies that the operator $\gKM{Y,\eta}{\prime}$
satisfies the intertwining equation
\[
\gKM{Y,\eta}{\prime} u = \psi_{Y,\eta}^\prime(u) \gKM{Y,\eta}{\prime} 
\]
for  any $u \in U_q\mfk_{\bm \ga,\bm \si}$, the coproduct identity
\eq{ \label{k:Del:alt}
\Del(\gKM{Y,\eta}{\prime}) = (R_{Y,\eta})_{21} \; (1 \ot \gKM{Y,\eta}{\prime}) \; (\psi_{Y,\eta}^\prime \ot \id)(R) \; (\gKM{Y,\eta}{\prime} \ot 1)\,,
}
and the generalized reflection equation
\begin{align}
(\gKM{Y,\eta}{\prime} \ot 1) \cdot (\id \ot \psi_{Y,\eta}^\prime)&(R_{21}) \cdot (1 \ot \gKM{Y,\eta}{\prime}) \cdot R =\\=& (\psi_{Y,\eta}^\prime \ot \psi_{Y,\eta}^\prime)(R_{21}) \cdot (1 \ot \gKM{Y,\eta}{\prime}) \cdot (\psi_{Y,\eta}^\prime \ot \id)(R) \cdot (\gKM{Y,\eta}{\prime} \ot 1)\,.
\end{align}


\subsection{Distinguished K-matrices} \label{sec:maxminGSat}\label{sec:GSatmin:finite}
By Sections~\ref{sec:parabolick} and \ref{sec:intw}, as $(Y,\eta)$ ranges through $\GSat(A)$, we obtain various universal K-matrices $\KM{Y,\eta}$. In the case $(Y,\eta)=(X,\tau)$, 
we recover the standard universal K-matrix $ \KM{X,\tau}$. From a representation theoretic point of view,
this choice is somewhat preferable, since  $\KM{X,\tau}\in(U_q\mfg)^{\mc O^+}$,\ie it acts on any  category $\mc O^+$ $U_q\mfg$-module. Note however that this is no longer true if $Y\neq X$.\\

In the case $(Y,\eta)=(X,\tau)$, the twisting operator $\psi_{X,\tau}$ is in 
general a complicated automorphism, since it coincides with $\theta_q^{-1}$. From the point of view of integrability theory, it is convenient to look for choices of $(Y,\eta)$, 
yielding a \emph{simple} form of the generalized reflection equation.
As a measure, since $\psi_{Y,\eta}=\theta_q(Y,\eta)^{-1}\circ\eta\circ\tau$, we consider the dimension of the subspace of fixed points in $\mfh'$.
Namely, we define a strict linear order on $\GSat(A)$ given by
\begin{align*}
(Y,\eta) < (Y',\eta') \qu 
&\Longleftrightarrow \qu \dim((\mfh')^{-\theta(Y,\eta)})<\dim((\mfh')^{-\theta(Y',\eta')}) .
\end{align*}
Recall that, since $\eta|_Y = \oi_Y$, the dimension of $(\mfh')^{-\theta(Y,\eta)}$ equals 
the restricted rank of $(Y,\eta)$, \ie the number of $\eta$-orbits in $I \backslash Y$.
There are two extreme cases.
\begin{enumerate}\itemsep0.25cm
\item The restricted rank is maximal. 
This corresponds to the choice $(Y,\eta) = (\emptyset,\id)$.
In this case, we have $\theta(\emptyset,\id)=\om$,
\[
\KM{\emptyset,\id}\coloneqq T_{\emptyset,\id}^{-1}\cdot T_{X,\tau}\cdot\KM{X,\tau}
\qq\mbox{and}\qq
\psi_{\emptyset,\id}=\Ad(T_{\emptyset,\id})^{-1}\circ\om\circ\tau
\]
Since $T_{\emptyset,\id}\in(U_q\h)^{\mc W,\times}$, by setting $g=T_{\emptyset,\id}$ in Theorem~\ref{thm:kG}, we obtain a distinguished universal K-matrix
\begin{equation}\label{eq:semistandard-KM}
\KM{\omega}\coloneqq T_{X,\tau}\cdot\KM{X,\tau}\qq\mbox{and}\qq \psi_{\omega}\coloneqq \om\circ\tau\,,
\end{equation}
which we refer to as the \emph{semi-standard} universal K-matrix. Note that in this case
the twisting operator $\psi_{\omega}$ is an involution on $U_q\mfg$ and we obtain the 
generalized reflection equation
	\eq{
	\begin{aligned}
		 (\KM{\om} \ot 1) \cdot (\id \ot \psi_{\om})({R}_{21}) \cdot (1 \ot\KM{\om}) \cdot {R} 
		= R \cdot (1 \ot\KM{\om}) \cdot (\psi_{\om}\ot \id)({R}) \cdot (\KM{\om} \ot 1).
	\end{aligned}
}

\item The restricted rank is minimal, \ie it equals the number of $\tau$-orbits in the 
complement of the largest $\tau$-stable subset of $I$ of finite type. If $A$ is of infinite type, 
there are in general several $(Y,\eta) \in \GSat(A)$ whose restricted rank is minimal.
On the other hand, if $A$ is of finite or affine type, the minimal restricted rank is 0 or 1, respectively,
and the choice is canonical as we describe below.

\item

If $\mfg$ is of finite type, $\GSat(A)$ has a unique minimal element given by $(I,\oi_I)$.
In this case, we get
\[
\KM{\scsop{fin}}\coloneqq T_{I,\oi_I}^{-1}\cdot T_{X,\tau}\cdot\KM{X,\tau}
\qq\mbox{and}\qquad
\psi_{\scsop{fin}} = \oi_I \circ \tau\,,
\]
since $\theta_q(I,\oi_I)=\id$. Note that the twisting operator is an involution on $U_q\mfg$ 
(cf.~\cite[Rmk 7.2]{BK19}).
Finally, we obtain the generalized reflection equation
\eq{ \label{kfinite:RE}
(\KM{\scsop{fin}} \ot 1) \cdot (\id \ot \psi_{\scsop{fin}})(R_{21}) \cdot (1 \ot \KM{\scsop{fin}}) \cdot R = 
R_{21} \cdot (1 \ot \KM{\scsop{fin}}) \cdot (\psi_{\scsop{fin}}\ot \id)(R) \cdot (\KM{\scsop{fin}} \ot 1).
}
\end{enumerate}

\begin{remark}
The case (iii) is considered in \cite[Corollary 7.7]{BK19}, where the opposition involution
$\oi_I$ is denoted $\tau_0$. More precisely, relying on the alternative formalism from Section ~\ref{sec:alternative}, the operator $\mc K$ constructed in \cite[Corollary 7.7]{BK19} is related to 
$\gKM{I,\oi_I}{\prime}$ by 
\[
\mc K ^{-1}= (\oi_I \circ \tau)\big( \gKM{I,\oi_I}{\prime} \big)
\]
under the additional constraints $\bm \ga = \bm \ga'$, $\bm \si = \bm \si'$, $\ga_{(\oi_I \circ \tau)(i)} = \ga_i$, and $\si_{(\oi_I \circ \tau)(i)} = \si_i$, cf.~\cite[(7.4)]{BK19}.
This is a consequence of the fact that
\[
\phi\big( U_q\mfk_{\bm \ga,\bm \si}(X,\tau) \big) = U_q\mfk_{\phi(\bm \ga),\phi(\bm \si)}(X,\tau)
\]
for involutive $\phi \in \Aut_X(A)$ commuting with $\tau$.
\hfill\rmkend
\end{remark}


\section{Spectral K-matrices for quantum affine $\mathfrak{sl}_2$} \label{s:affine}

In this section, we motivate our construction by discussing its application in the finite-dimensional representation theory of the quantum loop algebra $\UqLsl{2}$. 
The general case is treated in detail in \cite{AV22}. 
We show that the inversion of the spectral parameter on finite-dimensional representations can be realized in terms of a suitable choice of the twisting operators $\psi_{Y,\eta}$. 
The specialization of the corresponding universal K-matrix yields a formal solution of the generalized reflection equation with a spectral parameter. 
Finally, we prove that this construction gives rise to matrix solutions of the standard reflection equation which are formal series in the spectral parameter. 

\subsection{The quantum loop algebra $\UqLsl{2}$}
We set $\indI\coloneqq\{0,1\}$ and we consider the symmetric generalized Cartan matrix $A$
with $a_{01}=-2$. 
We denote the corresponding Kac-Moody algebra by $\wt{\mathfrak{sl}}_2$ and
its derived subalgebra by $\asl{2}$ (see \eg \cite{FR92}). 
We shall consider $0$ as the \emph{affine} node \cite[Ch.~6]{Ka90}.
The Lie algebra $\asl{2}$ is an extension of the loop algebra $L\mathfrak{sl}_2\coloneqq\mathfrak{sl}_2 \ot\C[t,t^{-1}]$ by a central element $c$. 
Similarly, the element $t_c$ is central in $U_q\asl{2}$ and the quotient $\UqLsl{2}\coloneqq U_q\asl{2}/(t_c-1)$ is known as the quantum \emph{loop} algebra of $\mathfrak{sl}_2$. 
As in the classical case, $\UqLsl{2}$ is endowed with a family of algebra homomorphisms $\ev{a}:\UqLsl{2}\to\Uqsl{2}$
($a\in\F^\times$) called \emph{evaluation homomorphisms} (see \eg \cite[Prop.~4.1]{CP91}) defined as follows:
\begin{align}
	\ev{a}(E_1) &= E, & \ev{a}(F_1) &= F & \ev{a}(t_1) &= t, \\
	\ev{a}(E_0) &= q^{-1}aF, & \ev{a}(F_0) &= qa^{-1}E, & \ev{a}(t_0) &= t^{-1}\,.
\end{align}  
Here we have denoted the Chevalley-Serre generators of $\Uqsl{2}$ by $E$, $F$, $t^{\pm 1}$, suppressing the subscript 1.
Note that $\UqLsl{2}$ is also endowed with a \emph{grading shift} automorphism (cf.~\cite{Dr86})
\begin{equation}\label{eq:grading-shift}
	\shift{z}: \UqLsl{2}[z,z^{-1}]\to\UqLsl{2}[z,z^{-1}]\,,
\end{equation}
where $\UqLsl{2}[z,z^{-1}]\coloneqq \UqLsl{2}\ten\F[z,z^{-1}]$,
given by $\shift{z}(t_{h})\coloneqq t_{h}$, $\shift{z}(E_{i})\coloneqq z^{\delta_{0i}} E_{i}$,
and $\shift{z}(F_{i})\coloneqq z^{-\delta_{0i}}F_{i}$. By specializing $z$
in $\F^\times$, we obtain a one-parameter family of automorphism of $\UqLsl{2}$ satisfying $\ev{a}=\ev{1}\circ\shift{a}$.

\subsection{Evaluation representations}

By pullback through $\ev{a}$, every irreducible finite-dimensional $\Uqsl{2}$-module is acted upon by $\UqLsl{2}$. 
Specifically, let $V_n$ be the $(n+1)$-dimensional irreducible $\Uqsl{2}$-module. 
For any $a\in\F^\times$, we obtain an irreducible (type $\bm 1$) $\UqLsl{2}$-module $V_n(a)\coloneqq\ev{a}^*(V_n)$, which is referred to as an {\em evaluation} representation.
Let $\Repfd(\UqLsl{2})$ be the category of finite-dimensional (type $\bm 1$) $\UqLsl{2}$-modules, which clearly contains every evaluation representation.
By \cite[Thm.~4.11]{CP91}, every irreducible module in $\Repfd(\UqLsl{2})$ arises as a tensor product of evaluation representation. 
Note that, while category $\mc O^+$ integrable $U_q\wt{\mathfrak{sl}}_2$-modules form a semisimple and braided category, $\Repfd(\UqLsl{2})$ is not semisimple nor braided (see \eg \cite[Ch.~12]{CP95}). 
However, as we briefly recall below, it is \emph{functionally} braided, since the universal R-matrix of $U_q\wt{\mathfrak{sl}}_2$ gives rise to a parameter-dependent operator on finite-dimensional $\UqLsl{2}$-modules (cf.\ \cite{FR92, KhT92, KS95, EM03}). 

\subsection{Spectral R-matrices}
The universal R-matrix does not immediately act on finite-dimensional $\UqLsl{2}$-modules, since the series $\wtR$ determined by $R = \ka_\id\cdot \wtR$ (cf.~Section \ref{sec:standardR}) does not necessarily converge. 
Relying on the grading shift, for any $V\in\Repfd(\UqLsl{2})$ with action $\pi_V:\UqLsl{2}\to\End(V)$, we consider the infinite-dimensional representation $\shrep{V}{z}\coloneqq V\ten\Lfml{\F}{z}$, with action given by $\pi_{V,z}\coloneqq\pi_V\circ\shift{z}$. 
Then, for any $V,W\in\Repfd(\UqLsl{2})$, we obtain an operator
\[
R_{VW}(z,w)\coloneqq\pi_{V,z}\ten\pi_{W,w}(R)\in\fml{\End(V\ten W)}{z^{-1},w}
\]
By the explicit description of $\wtR$ and \eqref{quasiR:def}, it follows that $R_{VW}(z,w)$
is a formal series in $w/z$, which we denote by $R_{VW}(w/z)$. 
For any $V_1,V_2,V_3\in\Repfd(\UqLsl{2})$, the specialization of the universal R-matrix on the tensor product $\shrep{V_1}{z^{-1}}\ten V_2\ten \shrep{V_3}{w}$ yields
a formal solution of the Yang-Baxter equation with a spectral parameter:
\begin{equation}\label{eq:z-YBE}
	\sRM{12}{z} \sRM{13}{zw} \sRM{23}{w} = \sRM{23}{w} \sRM{13}{zw} \sRM{12}{z}\, .
\end{equation}

Relying on a similar strategy, the universal K-matrices constructed in Section~\ref{s:km}
produce formal solutions of generalized reflection equations with a spectral parameter.

\subsection{Quantum pseudo-fixed-point subalgebras for $\UqLsl{2}$}
We shall consider the quantum pseudo-fixed-point subalgebras $U_q\mfk=U_q\mfk_{{\bm\ga},{\bm\si}}(X,\tau)\subset U_q\wt{\mathfrak{sl}}_2$, where ${\bm\ga}\in\Ga_q$ is such that $\ga_0\ga_1=1$, ${\bm\si}\in\Sigma_q$, and 
$(X,\tau)$ is one of the Satake diagrams $(\emptyset,\id)$, $(\{1\},\id)$, $(\emptyset,(0\,1))$, 
with $(0\,1)$ being the permutation of two nodes of the affine Dynkin diagram. 
Following \cite{BB17},
we refer to the corresponding 
subalgebras as the \emph{q-Onsager algebra}, \emph{invariant q-Onsager algebra} and \emph{augmented q-Onsager algebra}, respectively.
Note that $U_q\mfk(X,\tau)$ identifies with a 
coideal subalgebra in $\UqLsl{2}$, since $\theta(X,\tau)(c)=-c$.

\subsection{Spectral K-matrices for the q-Onsager algebra}\label{ss:ons}
We consider the Satake diagram $(X,\tau)=(\emptyset, \id)$. 
The corresponding coideal subalgebra $\Ons 
\subseteq\UqLsl{2}$ is generated by 
\[
B_0 \coloneqq F_0 - q^{-1} \ga_0 E_0 t_0^{-1} + \si_0 t_0^{-1}, 
\qq\mbox{and}\qq B_1 \coloneqq F_1 - q^{-1} \ga_1 E_1 t_1^{-1} + \si_1 t_1^{-1}.
\]
Following Section~\ref{sec:parabolick}, we choose the auxiliary Satake diagram $(Y,\eta)=(\{1\},\id)$ and we consider the universal K-matrix $K$ and the twisting operator $\psi$ given by the formulae \eqref{kYeta:def}
and \eqref{psiYeta:def}. 
Note that $\psi$ descends to an automorphism of $\UqLsl{2}$ and
satisfies $\shift{z}\circ\psi=\psi\circ\shift{1/z}$.
We get the following special case of \cite[Thm.~4.2.1]{AV22}.

\begin{theorem}\label{thm:km-ons}
\hfill
\begin{enumerate}\itemsep0.25cm
\item
For any $V\in\Repfd(\UqLsl{2})$, the universal K-matrix $K$ descends to a formal series
$\sKM{V}{z}\in\fml{\End(V)}{z}$ satisfying 
\begin{equation}\label{eq:spectral-intw}
	\sKM{V}{z}\pi_{V}(\shift{z}(u))=\pi_{V}(\psi(\shift{1/z}(u)))K_V(z)
\end{equation}
for any $u\in\Ons$, \ie it yields a formal $\Ons$-intertwiner $\sKM{V}{z}\colon V(z)\to V^\psi(1/z)$, where
$V^\psi\coloneqq\psi^*(V)$.
\item 
For any $V,W\in\Repfd(\UqLsl{2})$, the generalized reflection equation with a spectral parameter
holds:
\eq{ \label{eq:spectral-RE}
	\begin{aligned}
		\sKM{V}{z} \ten \id &\cdot \sRM{W^{\psi}V}{zw}_{21} \cdot \id \ten \sKM{W}{w} \cdot \sRM{VW}{\tfrac{w}{z}} =\\
		&= \sRM{W^{\psi}V^{\psi}}{\tfrac{w}{z}}_{21}  \cdot \id \ten\sKM{W}{w} \cdot \sRM{V^{\psi}W}{zw} \cdot \sKM{V}{z} \ten \id 
	\end{aligned}
}
\end{enumerate}
\end{theorem}

\begin{proof}
We shall prove that $\gKM{}{} = (T_{Y,\eta}^{-1}T_{X,\tau}^{})\cdot{\bm \ga}^{-1}\cdot\wtKM$ acts on $\shrep{V}{z}$. 
Any finite-dimensional $\UqLsl{2}$-module has a weight decomposition over 
$\Plat/(\Plat\cap\Q\delta)$ and is integrable as a 
$\Uqsl{2}$-module.
Thus, $(T_{Y,\eta}^{-1}T_{X,\tau}^{})\cdot{\bm \ga}^{-1}$
naturally descends to an element in $\End(V)$, since $\ga_0\ga_1=1$ and  $T_{Y,\eta}^{-1}T_{X,\tau}^{}$ is supported only on $\Uqsl{2}$.
Moreover, it is invariant under $\shift{z}$. Finally, as in the case of $\wtR$ \cite{Dr86}, 
the shifted quasi-K-matrix $\shift{z}(\wtKM)$ gives a formal series in $\fml{\End(V)}{z}$.
Therefore, we get
\[
\sKM{V}{z}\coloneqq\pi_{V,z}(K)\in\fml{\End(V)}{z}.
\]
Note that \eqref{eq:spectral-intw} follows from \eqref{kYeta:intw} and the identity $\shift{z}\circ\psi=\psi\circ\shift{1/z}$.
Similarly, \eqref{eq:spectral-RE} follows directly from \eqref{kYeta:RE}.
\end{proof}

A direct computation shows that $\psi$ is the identity on $\Uqsl{2}=\langle E_1, F_1, K_1\rangle$ and, for any $a\in\F^\times$, it satisfies $\ev{a}\circ\psi=\ev{q^2a^{-1}}$. 
Thus, we get the following special case of \cite[Thm.~7.2.1]{AV22}.

\begin{corollary}\label{cor:km-ons}
	Let $V,W\in\Repfd(\UqLsl{2})$ be evaluation representations at $a=q$. Then, 
	the standard reflection equation with a spectral parameter
	holds:
	\eq{ \label{eq:standard-spectral-RE}
		\begin{aligned}
			\sKM{V}{z} \ten \id & \cdot \sRM{WV}{zw}_{21} \cdot \id \ten \sKM{W}{w} \cdot \sRM{VW}{\tfrac{w}{z}} =\\
			&= \sRM{WV}{\tfrac{w}{z}}_{21}  \cdot \id \ten\sKM{W}{w} \cdot \sRM{VW}{zw} \cdot \sKM{V}{z} \ten \id 
		\end{aligned}
	}
\end{corollary}

\begin{proof}
It is enough to observe that, for any $n\geqslant 0$, 
\[
\psi^*(V_n(q))=(\ev{q}\circ\psi)^*(V_n)=\ev{q}^*(V_n)=V_n(q)\,.
\]
Thus, $V^\psi=V$, $W^\psi=W$, and \eqref{eq:standard-spectral-RE} follows from \eqref{eq:spectral-RE}.
\end{proof}

\begin{remark}
For any $a\in\F^\times$, we obtain an analogue of Corollary~\ref{cor:km-ons} by observing 
that $V_n(a)^{\psi}=V_n(q^2a^{-1})$.\hfill\rmkend
\end{remark}

\subsection{Spectral K-matrices for the invariant q-Onsager algebra}
We now consider the Satake diagram $(X,\tau)=(\{1\}, \id)$. Note that in this
case we have $\ga_0=1=\ga_1$.
The corresponding coideal 
subalgebra $\IOns$ 
is generated by $E_1,\, F_1,\, t_1^{\pm 1}$, and 
\[
B_0 \coloneqq F_0 - q^2 \Ad(\wt T_1)(E_0) t_0^{-1}\, .
\]
Following Section~\ref{sec:parabolick}, we choose the auxiliary Satake diagram $(Y,\eta)=(\{1\},\id)=(X,\tau)$ and we consider the universal K-matrix $K$ and the twisting operator $\psi$ given by the formulae \eqref{kYeta:def}. 
Note that $\psi$ is the same as in Section~\ref{ss:ons}. In particular, it descends to an automorphism of $\UqLsl{2}$, satisfies $\shift{z}\circ\psi=\psi\circ\shift{1/z}$, and is the identity on $\Uqsl{2}$. Then, the analogues of Theorem~\ref{thm:km-ons} and Corollary~\ref{cor:km-ons} 
hold for $\IOns$. 
The proofs are the same and therefore omitted.

\begin{remark}\label{rmk:semi-km-ons}
Other examples of spectral K-matrices for $\Ons$ and $\IOns$ can be obtained 
by choosing different auxiliary Satake diagrams. For instance, following Section~\ref{sec:GSatmin:finite}, one can choose $(Y,\eta)=(\emptyset,\id)$ and 
consider the {\em semi-standard} universal K-matrix and twisting operator given by the formulae \eqref{eq:semistandard-KM}, \ie $\gKM{}{}\coloneqq T_{\emptyset, (0\,1)}\cdot {\bm \ga}^{-1}\cdot\wtKM$ and $\psi\coloneqq\omega$. Note that, in this case, 
one has $\ev{a}\circ\omega=\omega\circ\ev{q^2a^{-1}}$ and, for any $n\geqslant 0$ and $a\in\F^\times$,
\[
\omega^*(V_n(a))=(\ev{a}\circ\omega)^*(V_n)=(\omega\circ\ev{q^2a^{-1}})^*(V_n)
\simeq\ev{q^2a^{-1}}^*(V_n)=V_n(q^2a^{-1})\,,
\]
where the third identity relies on the isomorphism of $\Uqsl{2}$-modules 
$\omega^*(V_n)\simeq V_n$.
Thus, the semi-standard universal K-matrix yields 
a formal intertwiner $V_n(az)\to V_n(\tfrac{q^2}{az})$ with respect to the action of the 
coideal subalgebra.
\hfill\rmkend
\end{remark}

\subsection{Spectral K-matrices for the augmented q-Onsager algebra}
Finally, we consider the Satake diagram $(X,\tau)=(\emptyset, (0\,1))$. 
The corresponding coideal subalgebra $\AOns$ is generated by
$t_0^{\pm 1} t_1^{\mp 1}=t_0^{\pm 2}$\,,
\[
B_0 \coloneqq F_0 - q \ga_0 E_1 t_0^{-1} \qq\mbox{and}\qq 
B_1 \coloneqq F_1 - q \ga_1 E_0 t_1^{-1}. 
\]
As in Remark~\ref{rmk:semi-km-ons}, we consider the semi-standard universal K-matrix and the twisting operator 
\[
\gKM{}{}\coloneqq T_{X,\tau}\cdot {\bm \ga}^{-1}\cdot\wtKM
\qq\mbox{and}\qq
\psi\coloneqq\omega\circ\tau\,.
\]
Note that the operator $T_{X,\tau}$ is in this case just a Cartan correction. Up to such correction,
$K$ and $\psi$ correspond to the standard universal K-matrix $K_{X,\tau}$ and $\theta_q(X,\tau)^{-1}$, respectively.\\

In this case, the procedure described in Section~\ref{ss:ons} does not immediately apply, since $\tau$ does not commute with the grading shift. 
To remedy this, we consider the \emph{principal} grading shift
\begin{equation}\label{eq:prin-grading-shift}
	\shift{z}^{\prin}: \UqLsl{2}[z,z^{-1}]\to\UqLsl{2}[z,z^{-1}]\,,
\end{equation}
given by $\shift{z}^{\prin}(t_{i})\coloneqq t_{i}$, $\shift{z}^{\prin}(E_{i})\coloneqq z E_{i}$ and $\shift{z}^{\prin}(F_{i})\coloneqq z^{-1}F_{i}$ for $i \in \{0,1\}$. 
Indeed, it satisfies $\shift{z}^{\prin}\circ\psi=\psi\circ\shift{1/z}^{\prin}$. 
For any $V\in\Repfd(\UqLsl{2})$ with action $\pi_V:\UqLsl{2}\to\End(V)$, we consider the infinite-dimensional representation $\shrep{V}{z}\coloneqq V\ten\Lfml{\F}{z}$,
with action given by $\pi_{V,z}\coloneqq\pi_V\circ\shift{z}^{\prin}$. 
With this correction, the analogue of Theorem~\ref{thm:km-ons} holds for $\AOns$.\\

Fix $a\in\F^\times$. 
Let ${\bm\beta}:\Qlat\to\F^\times$ be the group homomorphism given by  ${\bm\beta}(\alpha_0)=-qa^{-1}$ and ${\bm\beta}(\alpha_1)=-q^{-1}a$.
As in Section~\ref{ss:cartan-operators}, we obtain an algebra automorphism $\Ad({\bm\beta}): \UqLsl{2}\to\UqLsl{2}$. 
Following Section~\ref{sec:K:Cartanmod}, we consider the universal K-matrix and twisting operator given by \eqref{kG:def}, \ie
\[
\gKM{\bm\beta}{}\coloneqq {\bm\beta}\cdot T_{\emptyset,(0\,1)}\cdot {\bm \ga}^{-1}\cdot\wtKM
\qq\mbox{and}\qq
\psi_{\bm\beta}\coloneqq\Ad({\bm\beta})\circ\omega\circ\tau\,.
\]
By direct inspection, the twisting operator $\psi_{\bm\beta}$ satisfies
$\ev{a}\circ\psi_{\bm\beta}=\ev{q^2a^{-1}}$. Therefore, for any $n\geqslant 0$, 
we obtain $\psi_{\bm\beta}^*(V_n(a))= V_n(q^2a^{-1})$. In particular, for
$V=V_n(q)$, the universal K-matrix $K_{\bm\beta}$ specializes to a formal $\AOns$-intertwiner $V(z)\to V(1/z)$, yielding the analogue of Corollary~\ref{cor:km-ons} for $\AOns$.
It is to be expected that $K_{\bm \beta}$ is related to the \emph{generic K-matrices} for the augmented q-Onsager algebra given in \cite[Sec.~4.1.2]{BTs18}.



\begin{thebibliography}{BBBR95}

\bibitem[ATL18]{ATL18}
	A. Appel, V. Toledano Laredo,  
	{\it A $2$-categorical extension of Etingof-Kazhdan quantisation}. 
	Sel.~ Math. (N.S.) {\bf 24}, no.~4 (2018): 3529--3617.
	
\bibitem[ATL19a]{ATL19a}
	A. Appel, V. Toledano Laredo, 
	{\it Uniqueness of Coxeter structures on Kac-Moody algebras}.
	Adv. Math. {\bf 347} (2019): 1--104.

\bibitem[ATL19b]{ATL19b}
	A. Appel, V. Toledano Laredo, 
	{\it Coxeter categories and quantum groups}.
	Sel. Math. {\bf 25}, no.~3 (2019): 44. 	

\bibitem[ATL24a]{ATL15}
A. Appel, V. Toledano Laredo,
{\it Monodromy of the Casimir connection of a symmetrisable Kac-Moody algebra},
Invent. Math. {\bf 236} (2024), no.~2, 549--672. 

\bibitem[ATL24b]{ATL22}
A. Appel, V. Toledano~Laredo,
\emph{Pure braid group actions on category $\mathcal{O}$ modules}, Pure Appl. Math. Q. {\bf 20} (2024), no.~1, 29--79.
	
\bibitem[AV22]{AV22}
	A. Appel, B. Vlaar,
	{\it Trigonometric K-matrices for finite-dimensional representations of quantum affine algebras}, J. Eur. Math. Soc. (2025), published online first.
	

\bibitem[BBBR95]{BBBR95}
	V. Back-Valente, N. Bardy-Panse, H. Ben Massaoud, G. Rousseau,
	{\it Formes presque-d\'{e}ploy\'{e}es des alg\`{e}bres de Kac-Moody: Classification et racines relatives}.
	J. Algebra {\bf 171} (1995): 43--96.

\bibitem[BB17]{BB17}
P. Baseilhac\ and\ S. Belliard, {\it Non-Abelian symmetries of the half-infinite XXZ spin chain}. Nuclear Phys. B {\bf 916} (2017), 373--385.

\bibitem[BK15]{BK15}
	M. Balagovi\'{c}, S. Kolb,
	{\it The bar involution for quantum symmetric pairs}.
	Represent. Theory {\bf 19}, no.~8 (2015): 186--210.
	
\bibitem[BK19]{BK19}
	M. Balagovi\'{c}, S. Kolb,
	{\it Universal K-matrix for quantum symmetric pairs}.
	Journal f\"{u}r die reine und angewandte Mathematik (Crelles Journal) 2019, no.~747 (2019): 299--353.

\bibitem[Bri71]{Bri71}
	E. Brieskorn,
	{\it Die Fundamentalgruppe des Raumes der regulären Orbits einer endlichen komplexen Spiegelungsgruppe}.
	Inv. Math. {\bf 12}, no.~1 (1971): 57--61.
	
\bibitem[BS72]{BS72}
	E. Brieskorn\ and\ K. Saito, {\it Artin-Gruppen und Coxeter-Gruppen}. Invent. Math. {\bf 17} (1972), 245--271.

\bibitem[Bro12]{Bro12}
	A. Brochier, 
	{\it A Kohno-Drinfeld theorem for the monodromy of cyclotomic KZ connections}. 
	Comm. Math. Phys. {\bf 311}, no.~1 (2012): 55–96.
	
\bibitem[BTs18]{BTs18}
	P. Baseilhac, Z. Tsuboi,
	{\it Asymptotic representations of augmented q-Onsager algebra and boundary K-operators related to Baxter Q-operators}.
	Nucl.\ Phys.\ B {\bf 929} (2018): 397--437.

\bibitem[BW18]{BW18}
	H. Bao, W. Wang,
	{\it A new approach to Kazhdan-Lusztig theory of type B via quantum symmetric pairs}.
	Soci\'{e}t\'{e} math\'{e}matique de France (2018). 

\bibitem[BW21]{BW21}
	H. Bao, W. Wang,
	{\it Canonical bases arising from quantum symmetric pairs of Kac-Moody type}.
	Comp. Math. {\bf 157}, no.~7 (2021), pp.1507-1537.
	
\bibitem[BZBJ18]{BZBJ18}
	D. Ben-Zvi, A. Brochier, D. Jordan, 
	{\it Quantum character varieties and braided module categories}. 
	Sel. Math {\bf 24}, no.~5 (2018): 4711--4748.

\bibitem[Ch84]{Ch84}
	I. Cherednik,
	{\it Factorizing particles on a half-line and root systems}.
	Theor. and Math. Phys. {\bf 61}, no.~1 (1984): 977--983.

\bibitem[Ch92]{Ch92}
	I. Cherednik,
	{\it Quantum Knizhnik-Zamolodchikov equations and affine root systems}.
	Comm. Math. Phys. {\bf 150} (1992): 109--136.

\bibitem[CM18]{CM18}
	K. De Commer and M. Matassa,
	{\it Quantum flag manifolds, quantum symmetric spaces and their associated universal K-matrices}. 
	Adv. in Math. {\bf 366} (2020): 107029.

\bibitem[CP91]{CP91}
	V. Chari\ and\ A. Pressley, {\em Quantum affine algebras}. 
	Comm. Math. Phys. {\bf 142} (1991), no.~2, 261--283.
	
\bibitem[CP95]{CP95}
	V. Chari and A.N. Pressley,
	{\it A guide to quantum groups}.
	Camb.\ Univ.\ Press (1995).

\bibitem[Dav07]{Da07}
	A. Davydov,
	{\it Twisted automorphisms of Hopf algebras},
	Preprint at \href{http://arxiv.org/abs/0708.2757}{\tt arXiv:0708.2757}.

\bibitem[Del72]{Del72}
	P. Deligne, {\it Les immeubles des groupes de tresses g\'{e}n\'{e}ralis\'{e}s}, Invent. Math. {\bf 17} (1972), 273--302. 
	
\bibitem[Del90]{D90}
	P. Deligne, 
	{\it Cat\'{e}gories tannakiennes}.
	In {\it The Grothendieck Festschrift, Vol. II}, {\bf 87}, Progr. Math. (1990): 111--195.
	Birkh\"{a}user Boston, Boston, MA.

\bibitem[DG02]{DG02}
	G. Delius, A. George,
	{\it Quantum affine reflection algebras of type $d_n^{(1)}$ and reflection matrices}.
	 Lett. Math. Phys. {\bf 62} (2002): 211--217.
		
\bibitem[DK19]{DK19}
	L. Dobson, S. Kolb,
	{\it Factorisation of quasi K-matrices for quantum symmetric pairs}.
	Sel. Math. {\bf 25}, no.~4 (2019): 63.

\bibitem[DKM03]{DKM03}
	J. Donin, P. Kulish, A. Mudrov,
	{\it On a universal solution to the reflection equation}.
	Lett. Math. Phys. {\bf 63} (2003): 179--194.

\bibitem[DM03]{DM03}
	G. Delius, N. Mackay,
	{\it Quantum group symmetry in sine-Gordon and affine Toda field theories on the half-line}.
	Comm. Math. Phys. {\bf 233} (2003): 173--190.
	
\bibitem[tD98]{tD98}
	T. tom Dieck,
	{\it Categories of rooted cylinder ribbons and their representations}.
	J. reine angew. Math. {\bf 494} (1998): 36--63.

\bibitem[tDHO98]{tDHO98}
	T. tom Dieck and R. H\"{a}ring-Oldenburg,
	{\it Quantum groups and cylinder braiding}.
	Forum Math. {\bf 10}, no.~5 (1998): 619--639.

\bibitem[Dr85]{Dr85} 
	V. Drinfeld, 
	{\it Hopf algebras and the quantum Yang-Baxter equation}.
	 Soviet Math. Dokl. {\bf 32} (1985): 254--258.
	 
\bibitem[Dr86]{Dr86}
	V. Drinfeld, 
	{\it Quantum groups}.
	Proceedings of the International Congress of Mathematicians, Berkeley (1986), A. M. Gleason (ed): 798-820.
	Amer. Math. Soc., Providence, RI.

\bibitem[Dr90a]{Dr90a}
	V. Drinfeld, 
	{\it On almost cocommutative Hopf algebras}. 
	Leningrad Math. J. {\bf 1}, no.~2 (1990): 321--342; translated from Algebra i Analiz {\bf 1}, no.~2 (1989): 30--46.
	
\bibitem[Dr90b]{Dr90b}
	V. Drinfeld, 
	{\it On some unsolved problems in quantum group theory}.
	In {\it Quantum groups (Leningrad, 1990)}, 1--8, Lecture Notes in Math., 1510, Springer, Berlin.

\bibitem[E22]{E22}
P. Etingof, Private communication (2022).

\bibitem[EM03]{EM03}
	P. Etingof\ and\ A. Moura, 
	{\it Elliptic central characters and blocks of finite dimensional representations of quantum affine algebras}.
	Represent. Theory {\bf 7} (2003): 346--373.

\bibitem[Enr04]{E04}
	B. Enriquez, {\it Quasi-reflection algebras and cyclotomic associators,} Sel.~ Math. (N.S.) {\bf 13}, no.~3 (2007): 391--463. 

\bibitem[Enr10]{E10}
	B. Enriquez,
	{\it Half-balanced braided monoidal categories and Teichm\"{u}ller groupoids in genus zero}.
	Preprint at \href{http://arxiv.org/abs/1009.2652v2}{\tt arXiv:1009.2652v2}.	

\bibitem[ES18]{ES18}
	M.~Ehrig, C.~Stroppel,
	{\it Nazarov-Wenzl algebras, coideal subalgebras and categorified skew Howe duality}.
	Adv.~Math.~{\bf 331} (2018): 58--142.

\bibitem[FR92]{FR92}
	I. Frenkel, N. Reshetikhin,
	{\it Quantum affine algebras and holonomic difference equations}.
	Commun. Math. Phys. {\bf 146} (1992): 1--60.

\bibitem[GK91]{GK91}
	A. M. Gavrilik, A. U. Klimyk, 
	{\it q-Deformed orthogonal and pseudo-orthogonal algebras and their representations}.
	Lett. Math. Phys. \textbf{21} (1991), no. 3, 215–220.

\bibitem[GZ94]{GZ94}
	S. Ghoshal, A. Zamolodchikov,
	{\it Boundary S-matrices and boundary state in two-dimensional integrable quantum field theory}.
	Int. J. Mod. Phys. A \textbf{09} (1994): 3841--3886.
		
\bibitem[He84]{He84}
	A. Heck,
	{\it Involutive automorphisms of root systems}.
	J. Math. Soc. Japan {\bf 36}, no.~4 (1984): 643--658.


\bibitem[Jan96]{Jan96}
	J. C. Jantzen,
	{\it Lectures on quantum groups}.
	Grad. Stud. Math., vol. 6, Amer. Math. Soc.~(1996).

\bibitem[Ji86]{Ji86}
	M. Jimbo, 
	{\it A q-analogue of $U(\mfgl(N+1))$, Hecke algebra, and the Yang-Baxter equation}.
	Lett. Math. Phys. {\bf 11}, no.~3 (1986): 247--252.

\bibitem[Kac90]{Ka90}
	V. Kac,
	{\it Infinite dimensional Lie algebras}.
	3rd. ed., Cambridge University Press, 1990.

\bibitem[Kas90]{Kas90}
	M. Kashiwara, 
	{\it Crystalizing the q-analogue of universal enveloping algebras},
	Comm. Math. Phys. {\bf 133} (1990): 249--260.
	
\bibitem[KhT92]{KhT92}
	S. Khoroshkin, V. Tolstoy,
	{\it The universal R-matrix for quantum untwisted affine Lie algebras}.
	Funct. Anal. Appl. {\bf 26} (1992): 69--71.
	
\bibitem[Ko14]{Ko14}
	S. Kolb,
	{\it Quantum symmetric Kac-Moody pairs}.
	Adv. Math. {\bf 267} (2014): 395--469.
	
\bibitem[Ko20]{Ko20}
	S. Kolb,
	{\it Braided module categories via quantum symmetric pairs}.
	Proc. of the London Math. Soc. {\bf 121}, no.~1 (2020): 1--31.

\bibitem[Ko21]{Ko21}
	S. Kolb,
	{\it The bar involution for quantum symmetric pairs--hidden in plain sight}.
	Preprint at \href{http://arxiv.org/abs/2104.06120}{\tt arXiv:2104.06120}.

\bibitem[KR90]{KR90}
	A. N. Kirillov\ and\ N. Reshetikhin, 
	{\it $q$-Weyl group and a multiplicative formula for universal $R$-matrices}.
	Comm. Math. Phys. {\bf 134}, no.~2 (1990): 421--431.

\bibitem[KS92]{KS92}
	P.P. Kulish, E.K. Sklyanin,
	{\it Algebraic structures related to reflection equations}. 
	J. Phys. A: Math. Gen. {\bf 25}, no.~22 (1992): 5963.

\bibitem[KS95]{KS95}
	D. Kazhdan\ and\ Y. Soibelman, 
	{\it Representations of quantum affine algebras}.
	Sel.~ Math. (N.S.) {\bf 1}, no.~3 (1995): 537--595. 

\bibitem[KT09]{KT09}
	J. Kamnitzer, P. Tingley,
	{\it The crystal commutor and Drinfeld’s unitarized R-matrix}.
	Journal of Algebraic Combinatorics {\bf 29}, no.~3 (2009): 315--335.

\bibitem[KW92]{KW92}
	V. Kac, S. Wang,
	{\it On Automorphisms of Kac-Moody Algebras and Groups}.
	Adv. Math. {\bf 92} (1992): 129--195.

\bibitem[KY20]{KY20}
	S. Kolb, M. Yakimov,
	{\it Symmetric pairs for Nichols algebras of diagonal type via star products}.
	Adv. Math. {\bf 365} (2020): 107042.

\bibitem[K93]{K93}
	T. Koornwinder,
	{\it Askey-Wilson Polynomials as Zonal Spherical Functions on the $SU(2)$ Quantum Group}.
	SIAM Journal on Mathematical Analysis {\bf 24}, no.~3 (1993): 795--813.

\bibitem[Le99]{Le99}
	G. Letzter,
	{\it Symmetric Pairs for Quantized Enveloping Algebras}.
	J. Algebra {\bf 220} (1999): 729--767. 
		
\bibitem[Le02]{Le02}
	G. Letzter,
	{\it Coideal Subalgebras and Quantum Symmetric Pairs}.
	New Directions in Hopf Algebras, MSRI publications 43, Cambridge University Press (2002): 117--166.
	
\bibitem[Le03]{Le03}
	G. Letzter,
	{\it Quantum Symmetric Pairs and Their Zonal Spherical Functions}.
	Transformation Groups {\bf 8}, no.~3 (2003): 261--292

\bibitem[LS90]{LS90}
	S. Z. Levendorskij\ and\ Ya. S. Soibelman, 
	{\it Some applications of the quantum Weyl groups}, J. Geom. Phys. {\bf 7}, no.~2 (1990): 241--254. 

\bibitem[Lus90]{Lus90}
	G. Lusztig,
	{\it Canonical bases arising from quantized enveloping algebras}.
	J. Amer. Math. Soc. {\bf 3}, no.~2 (1990): 447--498.

\bibitem[Lus94]{Lus94}
	G. Lusztig,
	{\it Introduction to quantum groups}.
	Birkh\"auser, Boston, 1994.
	
\bibitem[NDS97]{NDS97}
	M. Noumi, M.S. Dijkhuizen and T. Sugitani,
	{\it Multivariable Askey-Wilson polynomials and quantum complex Grassmannians}, 
	AMS Fields Inst. Commun. {\bf 14} (1997): 167-177.

\bibitem[NS95]{NS95} 
	M. Noumi, T. Sugitani, 
	{\it Quantum symmetric spaces and related q-orthogonal polynomials}, in: Group Theoretical Methods in Physics (ICGTMP) (Toyonaka, Japan, 1994),
	World Scientific Publishing, River Edge, NJ, (1995): 28--40.
	
\bibitem[Rad92]{Rad92}
	D. E. Radford, {\em The structure of Hopf algebras with a projection}. 
	J. Algebra {\bf 92}, no.~2 (1985): 322--347.

\bibitem[RV16]{RV16}
	V. Regelskis, B. Vlaar,
	{\it Reflection matrices, coideal subalgebras and generalized Satake diagrams of affine type}.

\bibitem[RV20]{RV20}
	V. Regelskis, B. Vlaar,
	{\it Quasitriangular coideal subalgebras of $U_q(\mfg)$ in terms of generalized Satake diagrams}.
	Bull. London Math. Soc. {\bf 54}, no.~4 (2020): 693--715.

\bibitem[RV21]{RV21}
	V. Regelskis, B. Vlaar,
	{\it Pseudo-symmetric pairs for Kac-Moody algebras}.
	Preprint at \href{http://arxiv.org/abs/2108.00260}{\tt arXiv:2108.00260}.

\bibitem[Sk88]{Sk88}
	E. Sklyanin,
	{\it Boundary conditions for integrable quantum systems}.
	J. Phys. A: Math. Gen. {\bf 21} (1988): 2375--2389.

\bibitem[ST09]{ST09}
	N. Snyder, P. Tingley,
	{\it The half-twist for $U_q(\mathfrak g)$ representations}.
	Algebra \& Number Theory. {\bf 3}, no.~7 (2009): 809--834.

\bibitem[Ti10]{Ti10}
	P. Tingley, 
	{\it A formula for the $R$-matrix using a system of weight preserving endomorphisms}.
	Represent. Theory {\bf 14} (2010): 435--445.

\bibitem[Ti66]{Tits}
	J. Tits, 
	{\it Normalisateurs de tores I. Groupes de Coxeter \'{e}tendus}.
	J. Algebra \textbf{4} (1966): 96--116.

\end{thebibliography}
\end{document}